%% file: perfect-models.tex
\documentclass[11pt]{article}

\input{preamble.tex}

\def\dualityBC{\iota^\BC}
\def\dualityD{\iota^\D}

\usepackage{fullpage}
\usepackage[nomarkers,figuresonly,nofiglist]{endfloat}

\begin{document}
\title{Gelfand $W$-graphs for classical Weyl groups}

\author{
Eric Marberg\thanks{
Email: \tt emarberg@ust.hk
}
\and
Yifeng Zhang\thanks{
Email: \tt yzhangci@connect.ust.hk
}
}
\date{Department of Mathematics \\ Hong Kong University of Science and Technology}

\maketitle

\abstract{A Gelfand model for an algebra is a module given by a direct sum of irreducible submodules, 
with every isomorphism class  of irreducible modules represented exactly once.
We introduce the notion of a perfect model for a finite Coxeter group,
which is a certain set of discrete data (involving Rains and Vazirani's concept of a perfect involution)
that parametrizes a Gelfand model for the associated Iwahori-Hecke algebra. We
describe perfect models for all classical Weyl groups, excluding type D in even rank. 
The representations attached to these models simultaneously
generalize constructions of Adin, Postnikov, and Roichman (from type A to other classical types)
and of Araujo and Bratten (from group algebras to Iwahori-Hecke algebras).
We show that each Gelfand model derived from a perfect model has a canonical basis that gives rise to
a pair of related $W$-graphs, which we call Gelfand $W$-graphs.
For types BC and D, we prove that 
these $W$-graphs are dual to each other, a phenomenon which does not occur in type A.}

\setcounter{tocdepth}{2}
\tableofcontents

 \section{Introduction}

In this article we study certain uniform ways 
of constructing $W$-graphs for finite Weyl groups whose associated Iwahori-Hecke algebra
representations are multiplicity-free.
We start by defining a special kind of \emph{model} (in the sense of \cite[\S3]{BumpGinzburg}) that leads to these constructions.

\subsection{Perfect models}

Let $(W,S)$ be a Coxeter system with length function $\ell : W \to \NN := \{0,1,2,\dots\}$.
Define $\Aut(W,S)$ to be the group of \emph{Coxeter automorphisms} of $W$, that is, automorphisms $\varphi \in \Aut(W)$ with $\varphi(S) = S$.
Let $W^+ = W \rtimes \Aut(W,S)$ be the semidirect product group whose elements are the pairs $(w, \varphi)$ for $w \in W$ and $\varphi \in \Aut(W,S)$, with multiplication 
$ (v,\alpha)(w,\beta) := (v \cdot \alpha(w), \alpha\beta).$
Extend the length function of $W$ to $W^+$ by setting $\ell(w,\varphi) = \ell(w)$.
We view $W$ as a subgroup of $W^+$ by identifying $w \in W$ with $(w,1) \in W^+$.

\begin{example}
For $i > 0$, define permutations 
$
s_i := (i,i+1)(-i,-i-1),
$
$
s_0 := (-1,1),
$
and
$ s_{-i} = (i,-i-1)(-i,i+1)
$.
We realize the classical Weyl groups of types $\A_{n-1}$, $\BC_n$, and $\D_n$
as 
\[
S_n := \langle s_1,s_2,\dots, s_{n-1}\rangle
\subset \W_n := \langle s_0,s_1,s_2,\dots,s_{n-1}\rangle
\supset \WD_n := \langle s_{-1},s_1,s_2,\dots,s_{n-1}\rangle.\]
Here and throughout, $n$ is a fixed positive integer.
Each of these groups is a Coxeter group relative to the listed set
of simple generators.
The map $s_i \mapsto s_{n-i}$ extends to the only nontrivial Coxeter automorphism of $S_n$.
For $n>4$, the map interchanging $s_{-1}$ and $s_{1}$ while fixing the other generators extends to the only nontrivial Coxeter automorphism of $\WD_n$. 
When $n>2$ there are no nontrivial Coxeter automorphisms of $\W_n$.
\end{example}

Let $z = (w,\varphi) \in W^+$.
Following \cite{RV}, we say that $z$ is a \emph{perfect involution} if 
\[z^2= (z t )^4 = 1 \quad\text{for all reflections
$t \in  \{ wsw^{-1} : (w,s) \in W\times S\}.$}\]
Rains and Vazirani identify all such elements when $W$ is a finite Coxeter group
in \cite[\S9]{RV}; we review this classification in Section~\ref{perfect-subsect}.

Let $\sI = \sI(W,S)$ be the set of perfect involutions in $W^+$. 
The group $W$ acts on $\sI$ by conjugation
$v : (w,\varphi) \mapsto v (w,\varphi)v^{-1} =  (v\cdot w \cdot\varphi(v)^{-1},\varphi) .$
%
As in \cite{RV}, we say that an element $z \in \sI$ is \emph{$W$-minimal}
if $\ell(sz s) \geq \ell(z)$ for all $s \in S$, and \emph{$W$-maximal}
if $\ell(sz s) \leq \ell(z)$ for all $s \in S$.
Each $W$-conjugacy class in $\sI$ contains a unique $W$-minimal element
(and a unique $W$-maximal element when $W$ is finite) \cite[Cor. 2.10]{RV}.


If $C_W(z) := \{ w \in W: wz = z w\}$ is the $W$-stabilizer of a $W$-minimal element
$z\in \sI$,
 then $C_W(z)$ is a \emph{quasiparabolic subgroup} in the sense of Rains and Vazirani \cite[Thm. 4.6]{RV}. Among other consequences from \cite{RV}, this implies that the 
 permutation representation of $W$ acting on  $W/C_W(z)$ deforms to a representation of the \emph{Iwahori-Hecke algebra} of $(W,S)$.

Given  $J\subset S$, let $W_J := \langle  J\rangle$. 
Then $(W_J, J)$ is a Coxeter system
whose length function is the restriction of $\ell$.
Write $\sI_J := \sI(W_J,J)$ for the set of perfect involutions in $(W_J)^+$.
%

\begin{definition}\label{model-triple-def}
A \emph{model triple} $(J, z_{\min}, \sigma)$ for a finite Coxeter group $W$ consists of
a set $J\subset S$, a $W_J$-minimal element $z_{\min} \in  \sI_J$,
and a linear character (that is, a group homomorphism) $\sigma : W_J \to \{\pm 1\}$.
A set $\sP$ of model triples  is a \emph{perfect model} for $W$  if
\[\sum_{(J,z_{\min},\sigma) \in \sP} \Ind_{C_{W_J}(z_{\min})}^W \Res^{W_J}_{C_{W_J}(z_{\min})}(\sigma) = \sum_{\chi \in \Irr(W)} \chi,\]
where $\Irr(W)$ is the set of complex irreducible characters of $W$
and $\Ind$ and $\Res$ are the usual operations of induction and restriction.
\end{definition}

Our first main result identifies
perfect models for most classical Weyl groups.
The model  for type $\A$ given below is well-known; see \cite{APR2007,IRS,KV2004}.
What we describe for type $\BC$ (respectively, $\D$)
is a coarser version of the models in
\cite{APR2010,Baddeley,M2012}
(respectively, \cite{Caselli2009,CF,CM,M2011}).

\begin{definition}\label{main-thm1-def}
 Let $W \in \{ S_n, \W_n, \WD_n\}$ be a classical Weyl group.
 Write $\one=\one_W : w \mapsto 1$ and $\sgn=\sgn_W : w \mapsto (-1)^{\ell(w)}$ for the trivial and sign 
 representations.
\ben

\item[($\A$)] Assume $W=S_n$ with $n\geq 1$.  For each integer $0\leq k\leq \lfloor\frac{n}{2}\rfloor$, form a triple $(J,z_{\min},\sigma)$ 
by taking $J = \{s_1,s_2,\dots,s_{n-1}\} \setminus\{s_{2k}\}$
so that  $ \langle J\rangle  = S_{2k} \times S_{n-2k} $,
and then setting
\[z_{\min} = s_1s_3s_5\cdots s_{2k-1}
\quand \sigma = \one_{S_{2k}} \times \sgn_{S_{n-2k}}.\]

\item[($\BC$)] Assume $W=\W_n$ with $n\geq 2$.  For each integer $0\leq k\leq \lfloor\frac{n}{2}\rfloor$, form a triple
  $(J,z_{\min},\sigma)$ by taking
$J = \{s_0,s_1,s_2,\dots,s_{n-1}\} \setminus\{s_{2k}\}$
 so that  $ \langle J\rangle  = \W_{2k} \times S_{n-2k}  $,
 and then setting
\[z_{\min} = s_1s_3s_5\cdots s_{2k-1}
\quand \sigma = \one_{\W_{2k}} \times \sgn_{S_{n-2k}}.\]

\item[($\D$)] Assume $W=\WD_n$ with $n\geq 3$. For each integer $0<k\leq \lfloor\frac{n}{2}\rfloor$, form a triple
  $(J,z_{\min},\sigma)$ by taking
$J = \{s_{-1},s_1,s_2,\dots,s_{n-1}\} \setminus\{s_{2k}\}$
 so that $ \langle J\rangle  = \WD_{2k} \times S_{n-2k}  $,
 and then setting
\[z_{\min} = s_1s_3s_5\cdots s_{2k-1}
\quand\sigma = \one_{\WD_{2k}} \times \sgn_{S_{n-2k}}.\]
Also include one additional triple $(J,z_{\min},\sigma)=(\{s_1,s_2,\dots,s_{n-1}\}, 1, \sgn_{S_n})$.

\een
Let $\sP=\sP(W)$ be the set of
 $1+\lfloor \frac{n}{2}\rfloor$ triples $(J,z_{\min},\sigma)$ listed in each case.
 \end{definition}
 
 As we will explain in the next section, what distinguishes these models from prior work  is that our constructions lead to
  Iwahori-Hecke algebra representations. 

\begin{theorem}
\label{main-thm1}
If $W$ is one of the classical Weyl groups $S_{n-1}$, $\W_n$, or $\WD_{2n+1}$ for an integer $n\geq 2$,
then $\sP(W)$ is a perfect model of $W$.
\end{theorem}

This result does not construct a perfect model for $\WD_{n}$ when $n$ is even.
The sequel to this article \cite{MZ} will take up the problem of classifying 
the finite Coxter groups with perfect models; it will be shown there
that no perfect models exist in type $\D_{2n}$.
This property exactly mirrors the non-existence of \emph{involution models}
in type $\D_{2n}$, which is shown in \cite{Baddeley} (see \cite{Vinroot} for the definition of an involution model).
In fact, we will prove in \cite{MZ} that a finite Coxeter group $W$
has a perfect model if and only if it has an involution model.

As noted in \cite[\S3]{BumpGinzburg}, there is a connection between the perfect model
$\sP$ in type $\A_{n-1}$ 
and certain \emph{Klyachko models} for the finite general linear group $\GL_n(\FF_q)$.
It would be interesting to know if  $\sP$ in types $\BC_n$ and $\D_n$ is similarly related to
models for $\O_n(\FF_q)$ or $\Sp_n(\FF_q)$.


\subsection{Iwahori-Hecke algebra modules}\label{intro-sect2}


Let $x$ be an indeterminate.
The \emph{(single-parameter) Iwahori-Hecke algebra} of a Coxeter system $(W,S)$ is
the free $\LL$-module $\H = \H(W)$ with basis $\{H_w : w \in W\}$, equipped with the unique algebra
structure in which 
\[ H_s H_w  =\begin{cases} H_{sw} &\text{if }\ell(sw) > \ell(w) \\
H_{sw} + (x-x^{-1}) H_w &\text{if }\ell(sw) <\ell(w)
\end{cases}
\quad\text{for all $s \in S$ and $w \in W$.} \]
When $\KK$ is a field, we write $\H_\KK = \H_\KK(W) := \KK(x) \otimes_{\ZZ[x,x^{-1}]} \H(W)$ for the algebra given in exactly the same way
but defined over the scalar field of rational functions $\KK(x)$.

If $W$ is finite with splitting field $\KK\subset \CC$, then the algebra $\H_\KK$ is split \cite[Thm. 9.3.5]{GeckPfeiffer}.
(Each finite Weyl group has splitting field $\KK=\QQ$ \cite[Thm. 6.3.8]{GeckPfeiffer}.)
In this case, a \emph{Gelfand model} for $\H_\KK$ is defined to be a left $\H_\KK$-module
 isomorphic to the direct sum of all irreducible $\H_\KK$-modules.
 Our second main theorem  describes a pair of Gelfand models for $\H_\QQ$ when $W$ is a classical Weyl group 
 outside type $\D_{2n}$.
As we explain in Section~\ref{gel-sect}, any perfect model for $W$ gives rise to  
 a Gelfand model for $\H_\KK$, so this result
is a formal consequence of
Theorem~\ref{main-thm1}.

\begin{definition}\label{fg-def}
Our Gelfand models will be spanned by the following sets of involutions:
\begin{itemize}
\item Define $\cF_{n-1}^\A := \left\{ z \in S_{2n} : z=z^{-1}\text{ and }z(i) \neq i \text{ for all }i \in [2n]\right\}$.
\item Define $\cF_{n}^\BC := \left\{ z \in \W_{2n} : z=z^{-1}\text{ and }|z(i)| \neq i \text{ for all }i \in [2n]\right\}$.
\item Define $\cF^\D_n$ to  be the subset of 
 $z \in \cF_n^\BC$ for which $|\{ i\in [n] : z(i) < -i\}|$
is even.
\end{itemize}
We say that an integer $i>0$ is a \emph{visible descent} of a permutation $z$
if \[
z(i+1) < \min\{i,z(i)\}\quord z(i) < -i.\] 
Define $\cG^\A_{n-1}$, $\cG^\BC_n$, and $\cG^\D_n$ to be the respective subsets of
$\cF^\A_{n-1}$, $\cF^\BC_n$, and $\cF^\D_n$ consisting of the elements with  no visible descents greater than $n$.
\end{definition}

Let $\cF = \cF(W)$ be $\cF^\A_{n-1}$, $\cF^\BC_n$, or $\cF^\D_n$
according to whether $W$ is $S_n$, $\W_n$, or $\WD_n$. 
Define $\cG = \cG(W)$ analogously, and note that 
if $W \in \{S_n, \W_n,\WD_n\}$ then $\cG  =  \cF \cap \cG^\BC_n$.
We will describe the sets $\cG^\A_{n-1}$, $\cG^\BC_n$, and $\cG^\D_n$  more explicitly in Section~\ref{descent-sect}.

\begin{example}
The 10 elements of $\cG^\A_3$ are 
\[
\ba
&
(1,5)(2,6)(3,7)(4,8),
\\[-12pt]\\&
(1,2)(3,5)(4,6)(7,8),\quad
(1,3)(2,5)(4,6)(7,8),\quad
(1,4)(2,5)(3,6)(7,8),
\\&
(2,3)(1,5)(4,6)(7,8),\quad
(2,4)(1,5)(3,6)(7,8),\quad
(3,4)(1,5)(2,6)(7,8), 
\\[-12pt]\\&
(1,2)(3,4)(5,6)(7,8), \quad
(1,3)(2,4)(5,6)(7,8), \quad
(1,4)(2,4)(5,6)(7,8).
\ea
\]
The 6 elements of $\cG^\BC_2$ are
\[ 
\ba&
(-4,-2)(-3,-1)(1,3)(2,4),\ \
(-4,1)(-3,2)(-2,3)(-1,4), \ \
(-4,-2)(-3,1)(-1,3)(2,4),\\&
(-4,-1)(-3,2)(-2,3)(1,4), \ \
(-4, -3)(-2,-1)(1,2)(3,4),\ \
(-4, -3)(2,-1)(1,- 2)(3,4).
\ea
\]
The 3 elements of $\cG^\D_2$ are
\[
(-4,-2)(-3,-1)(1,3)(2,4),\ \
(-4,1)(-3,2)(-2,3)(-1,4),\ \
(-4, -3)(-2,-1)(1,2)(3,4).\]
\end{example}

When evaluating the length $\ell(z)$ of an element $z \in \cG$, we use the length
function afforded by viewing
$
\cG^\A_{n-1} \subset S_{2n}
$,
$
\cG^\BC_{n} \subset \W_{2n} 
$,
and
$
\cG^\D_{n} \subset \WD_{2n}.
$
Now fix $z \in \cG$
and define 
\be\label{intro-des-1} \Des^=(z) := \{ s \in S : sz=zs\}
\quand 
\Asc^=(z) := \{ s \in S : zsz \in \{ s_i : i > n\} \}.
\ee
Finally let 
\be\label{intro-des-2}
\ba
\Des^< (z) &:= \{ s \in S : \ell(sz) < \ell(z)\} - \Des^=(z)\sqcup \Asc^=(z) ,
\\
 \Asc^< (z) &:= \{ s \in S : \ell(sz) > \ell(z)\} - \Des^=(z)\sqcup \Asc^=(z) ,
 \ea\ee
where $\ell$ is the length function of $S_{2n}$ (if $W=S_n$), $\W_{2n}$ (if $W=\W_n$), or $\WD_{2n}$ (if $W=\WD_n$).
We refer to elements of $ \Des^=(z)$ and $ \Asc^=(z)$ as \emph{weak descents} and \emph{weak ascents} of $z$,
and to elements of $ \Des^<(z)$ and $ \Asc^<(z)$ as \emph{strict descents} and \emph{strict ascents}.
%

\begin{theorem}\label{main-thm2}
Assume $(W,S)$ is the Coxeter system of type $\A_{n-1}$, $\BC_n$, or $\D_n$ with $n\geq 2$.
Define 
$\cM = \cM(W)$ and $\cN = \cN(W)$ to be the
 free $\LL$-modules with respective bases
  $\{M_z : z \in \cG \}$ and $ \{N_z : z \in \cG \}$.
There is a unique left $\H$-module structure on $\cM$ in which
\[
H_s  M_z = \begin{cases} 
M_{s  zs}  &\text{if }s \in \Asc^<(z) \\
M_{s  zs} + (x - x^{-1}) M_z &\text{if }s\in \Des^<(z) \\
-x^{-1} M_z & \text{if }s \in \Asc^=(z) \\
xM_z & \text{if }s \in \Des^=(z) 
\end{cases}
\quad\text{for all $s \in S$ and $z \in \cG$},
\]
and there is a unique left $\H$-module structure on $\cN$ in which
\[
H_s  N_z = \begin{cases} 
N_{s  zs}  &\text{if }s \in \Asc^<(z) \\
N_{s  zs} + (x - x^{-1}) N_z &\text{if }s\in \Des^<(z) \\
x N_z & \text{if }s \in \Asc^=(z) \\ 
-x^{-1}N_z & \text{if }s \in \Des^=(z) 
\end{cases}
\quad\text{for all $s \in S$ and $z \in \cG$}.
\]
Moreover, if
$(W,S)$ is not of type
 $ \D_{n}$
 with $n$ even, then the left $\H_\QQ$-modules
 \[
  \cM_\QQ=   \cM_\QQ(W) := \QQ(x) \otimes_{\ZZ[x,x^{-1}]} \cM
  \quand \cN_\QQ=  \cN_\QQ(W) := \QQ(x) \otimes_{\ZZ[x,x^{-1}]} \cN
\]
are both Gelfand models.
\end{theorem}

The module $\cM_\QQ(S_n)$ is essentially the Gelfand model for $\H_\QQ(S_n)$
that Adin, Postnikov, and Roichman study in \cite{APR2007}, 
while 
$\cN(W)$ for each $W \in\{S_n,\W_n,\WD_n\}$ is a deformation\footnote{
In the notation of \cite[\S1.2]{APR2007}, one has $q=x^2$ and $T_i = -x H_{s_i}$.
On fixing $z \in \cG^\BC_{n} \supset \cG^\A_{n-1}$, define $w \in \W_n$ to have
 $w(i) = z(i)$ if $|z(i)| \leq n$,  $w(i) = i$ if $z(i) > n$, and $w(i) = -i$ if $z(i) < n$.
Then an isomorphism from $\cM_\QQ(S_n)$ to the module $V_n$ in \cite[\S1.2]{APR2007}
is  given by the linear map with $M_z \mapsto C_w$.
Similarly, if one sets $x=1$, then
 the linear map sending $N_z \mapsto e_w$ 
  is an isomorphism from $\cN(W)$ to the representation in \cite[\S2]{AraujoBratten}
  for each $W \in\{S_n,\W_n,\WD_n\}$.
} of the $W$-representation 
described by Araujo and Bratten in \cite{AraujoBratten}.
Our constructions for types $\BC_n$ and $\D_{2n+1}$ 
partly resolve an open problem mentioned in \cite[\S7]{APR2010}.


%
Theorem~\ref{main-thm2} implies that $\cM(S_n)$ and $\cN(S_n)$ are abstractly isomorphic 
(after extending scalars)
to the type $\A_{n-1}$ version of the Hecke algebra module studied by Lusztig and Vogan in \cite{LV2,LV};
see \cite[\S5.3]{LV}. It is not clear how to write down these isomorphisms in any concrete way, however.
When $W$ is of type $\BC_n$ and $\D_{n}$, the Iwahori-Hecke algebra modules in \cite{LV2,LV} 
have the same dimensions as $\cM$ and $\cN$ but 
are no longer multiplicity-free.


\subsection{Gelfand $W$-graphs}\label{intro-sect3}

Let $p \mapsto \overline p$ denote the ring automorphism of $\LL$ sending $x \mapsto x^{-1}$.
A map $\phi : \cA \to \cB$ between $\LL$-modules is \emph{antilinear}
if $\phi(pa) = \overline p \cdot \phi(a)$ for all $a \in \cA$. 
For the algebra $\H$, 
there is a unique antilinear map $\H \to \H$, written $H \mapsto \overline H$ and called the \emph{bar operator},
such that $\overline{H_w} = (H_{w^{-1}})^{-1}$ for all $w \in W$.
This map is a ring involution.

One way to define the well-known \emph{Kazhdan-Lusztig basis} of $\H$ \cite{KL} is as the set of unique elements $\underline H_w$ for $w \in W$ satisfying
$\overline{\underline H_w} = \underline H_w \in H_w + \sum_{ \ell(y) < \ell(w)} x^{-1} \ZZ[x^{-1}] H_y.$
Our third main theorem constructs analogous ``canonical bases'' for the $\H$-modules $\cM$ and $\cN$
from Theorem~\ref{main-thm2}.  
This result is formally reminiscent of \cite[Thms. 0.2 and 0.4]{LV2}.

Define an \emph{$\H$-compatible bar operator} on a left $\H$-module $\cA$
to be an antilinear map $A \mapsto\overline A$ with $\overline{HA} = \overline H \cdot \overline A$ 
for all $ H \in \H$ and $A \in \cA$.
For the rest of this section,
assume $W$ is a classical Weyl group and define $\cG$, $\cM$, and $\cN$ as in Theorem~\ref{main-thm2}.

\begin{theorem}\label{main-thm3}
There are unique $\H$-compatible bar operators on the modules $\cM$ and $\cN$
from Theorem~\ref{main-thm2}
satisfying
$\overline{M_z} = M_z$ and $\overline{N_z} = N_z$ for all $z \in \cG$ with 
$\Des^{<}(z) = \varnothing$.
Each of these bar operators is an involution.
Additionally, 
  $\cM$ has a unique basis $\{ \underline M_z : z \in \cG\}$ with
\[\overline{\underline M_z} = \underline M_z \in M_z + \sum_{\substack{y \in \cG \\ \ell(y) < \ell(z)}} x^{-1} \ZZ[x^{-1}] M_y.\]
The module $\cN$ likewise has a unique basis 
$\{ \underline N_z : z \in \cG\}$ with
\[\overline{\underline N_z} = \underline N_z \in N_z + \sum_{\substack{y \in \cG \\ \ell(y) < \ell(z)}} x^{-1} \ZZ[x^{-1}] N_y.\]
\end{theorem}

The structure constants for the multiplication map $\H \times \cM \to \cM$ and $\H \times \cN \to \cN$
in the bases $\{ \underline M_z\}$ and $\{\underline N_z\}$
may be encoded using Kazhdan and Lusztig's notion of a \emph{$W$-graph} from \cite{KL}.
We review the definition below, following the conventions of \cite{Stembridge}.

\begin{definition}
An \emph{$S$-labeled graph}
  $\Gamma= (V,\omega,I)$ is a set $V$ with maps
 \[\omega : V\times V \to \LL
\quand I : V \to \{\text{ subsets of $S$ }\}.\]
%
We often think of this structure as a weighted directed graph on $V$
with edges $u\xrightarrow{\omega(u,v)} v$ for each $u,v \in V$ with $\omega(u,v) \neq 0$.
  An $S$-labeled graph $\Gamma=(V,\omega,I)$ 
  is a \emph{$W$-graph}
if
   the free $\LL$-module $\cY(\Gamma)$ with basis $\{ Y_v : v \in V \}$ has a left $\H$-module structure 
   in which
  \be\label{hy-eq}
    H_s Y_u = \begin{cases} x Y_u &\text{if }s \notin I(u) \\ 
 -x^{-1} Y_u +  \ds\sum_{\substack{v \in V\\ s\notin I(v)}} \omega(u,v) Y_v &\text{if }s \in I(u)
  \end{cases}
\quad\text{for all $s \in S$ and $ u \in V$.}\ee
\end{definition}
%

Define ${\underline M_z}$ and ${\underline N_z}$ as in Theorem~\ref{main-thm3} and 
let $\m_{yz}, \n_{yz} \in \ZZ[x^{-1}]$ for $y,z \in \cG$
be such that
\[\underline M_z = \sum_{y \in \cG} \m_{yz} M_y
\quand
\underline N_z = \sum_{y \in \cG} \n_{yz} N_y.\]
Write
$ \mu^\m_{yz}$
and
$\mu^\n_{yz} $
for the coefficients of $x^{-1}$ in $\m_{yz}$ and $\n_{yz}$.
For $z \in \cG$, define
\[\Asc^\m(z) := \Asc^<(z) \sqcup \Asc^=(z)\quand \Asc^\n(z) := \Asc^<(z) \sqcup \Des^=(z).\] 
Let $\omega^\m : \cG \times \cG \to \ZZ$  and $\omega^\n : \cG \times \cG \to \ZZ$ be the  maps with 
\[
\ba
\omega^\m(y,z) &=\begin{cases} \mu^\m_{yz} + \mu^\m_{zy}
&\text{if } \Asc^\m(y)\not\subset \Asc^\m(z), \\
0&\text{otherwise},
\end{cases}
\\
\omega^\n(y,z)& =\begin{cases} \mu^\n_{yz} + \mu^\n_{zy}
&\text{if } \Asc^\n(y) \not\subset \Asc^\n(z), \\
0&\text{otherwise}.
\end{cases}
\ea
\]
Finally define 
$\Gamma^\m=\Gamma^\m(W) := (\cG,  \omega^\m, \Asc^\m)
$
and
$
\Gamma^\n=\Gamma^\n(W) := (\cG,  \omega^\n, \Asc^\n).
$

\begin{theorem}\label{main-thm4}
Let $W \in \{S_n, \W_n,\WD_n\}$. Then
 $\Gamma^\m$ and $\Gamma^\n$ are
(quasi-admissible) $W$-graphs. The linear maps $Y_z \mapsto \underline M_z$ and $Y_z \mapsto \underline N_z$ 
are isomorphisms $\cY(\Gamma^\m) \cong \cM$ and $\cY(\Gamma^\n)\cong \cN$.
\end{theorem}

For an explanation of the term \emph{quasi-admissible}, see Definition~\ref{quasi-admissible-def}.
The integer edge weights of these $W$-graphs are not always positive: the first negative edge weights in each type occur when $W\in\{S_8,W^\BC_4,W^\D_7\}$.
To actually compute 
$\underline M_z$ and $\underline N_z$, one can use \eqref{hy-eq}, which becomes an inductive formula after substituting 
$Y_z \mapsto \underline M_z$ or $Y_z \mapsto \underline N_z$.
For examples of these $W$-graphs, see Figures~\ref{abcd3-fig},
\ref{a4-fig}, \ref{bc4-fig}, \ref{d4-fig}, and \ref{table-fig}.

\begin{figure}
\begin{center}
$\Gamma^\m(S_4)=\boxed{\raisebox{-0.5\height}{\includegraphics[scale=0.34]{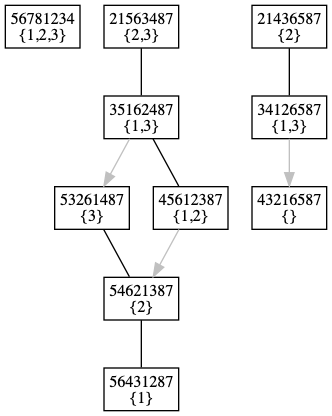}}}
\qquad
\boxed{\raisebox{-0.5\height}{\includegraphics[scale=0.34]{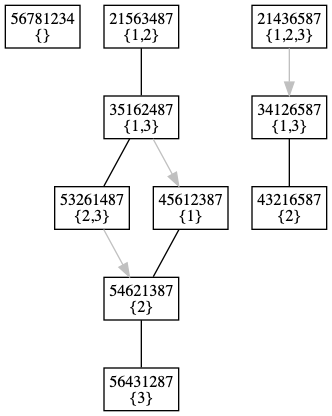}}}=\Gamma^\n(S_4)$
\\ \ \\ \ \\
$\Gamma^\m(\W_3)=\boxed{\raisebox{-0.5\height}{\includegraphics[scale=0.34]{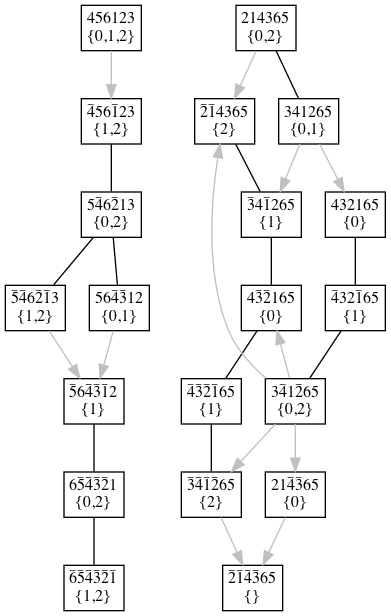}}}
\qquad
\boxed{\raisebox{-0.5\height}{\includegraphics[scale=0.34]{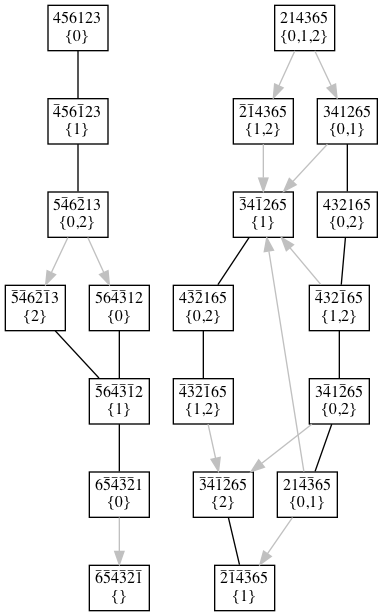}}}=\Gamma^\n(\W_3)$
\\ \ \\ \ \\
$\Gamma^\m(\WD_3)=\boxed{\raisebox{-0.5\height}{\includegraphics[scale=0.34]{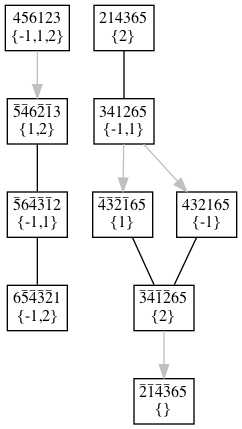}}}
\qquad
\boxed{\raisebox{-0.5\height}{\includegraphics[scale=0.34]{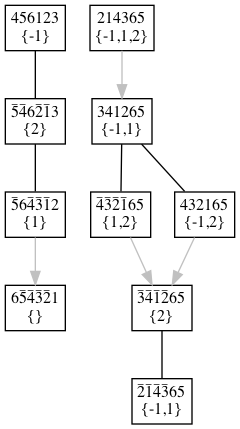}}}=\Gamma^\n(\WD_3)$
\end{center}
\caption{
The left boxes show $\Gamma^\m(W)$ and the right boxes show $\Gamma^\n(W)$ for $W \in \{S_4, \W_3, \WD_3\}$.
The pairs in types $\BC_3$ and $\D_3$ consist of dual $W$-graphs.
Here and in later figures, unlabeled edges $u \to v$ have weight $\omega(u,v)=1$,
 the sets below each permutation $z$ are $\Asc^\m(z)$ or $\Asc^\n(z)$ as appropriate,
 and undirected edges $u\ \dash\ v$ stand for pairs of directed edges $u \to v$ and $v \to u$.
}
\label{abcd3-fig}
\end{figure}

\begin{figure}
\begin{center}
$\ba \Gamma^\m(S_5)&=\boxed{\raisebox{-0.5\height}{\includegraphics[scale=0.43]{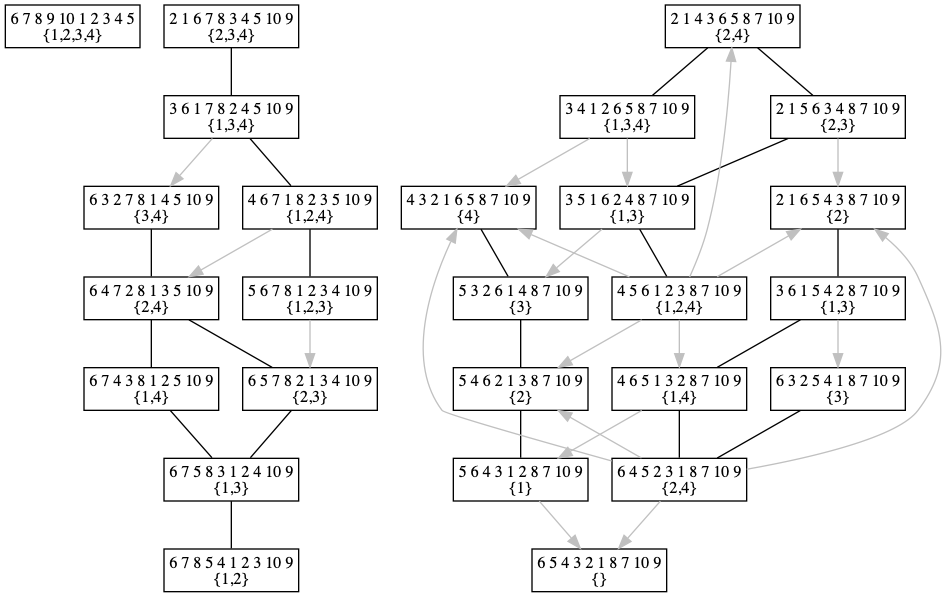}}}
\\ \ \\
\Gamma^\n(S_5)&=\boxed{\raisebox{-0.5\height}{\includegraphics[scale=0.43]{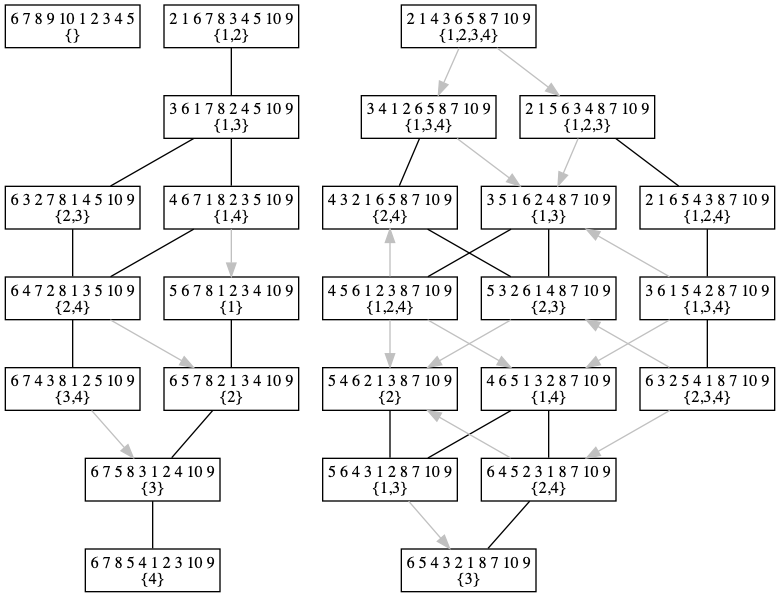}}} \ea$
\end{center}
\caption{The top box shows $\Gamma^\m(S_5)$ and the bottom box shows $\Gamma^\n(S_5)$.
These $W$-graphs are not dual or isomorphic.
}
\label{a4-fig}
\end{figure}

\begin{figure}
\begin{center}
$\boxed{\raisebox{-0.5\height}{\includegraphics[scale=0.33]{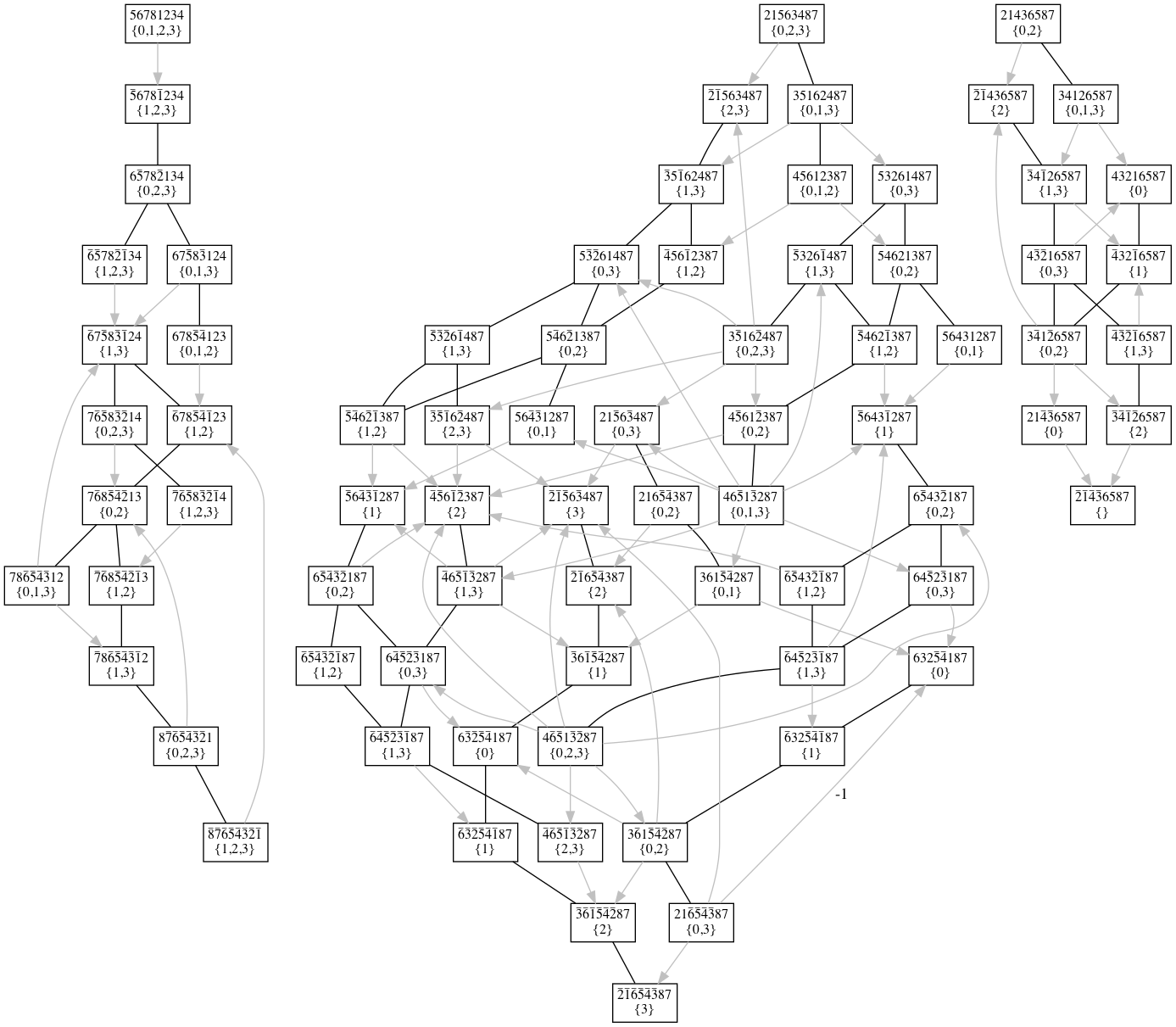}}}$
\end{center}
\caption{The box shows $\Gamma^\m(\W_4)$, which has a single edge $u \to v$ with weight
$\omega^\m(u,v) = -1$. All other edges $u\to v$ have weight $\omega^\m(u,v)=1$.
One can compute $\Gamma^\n(\W_4)$ from this picture using Theorem~\ref{bc-dual-thm}. 
}
\label{bc4-fig}
\end{figure}

\begin{figure}
\begin{center}
$\boxed{\raisebox{-0.5\height}{\includegraphics[scale=0.55]{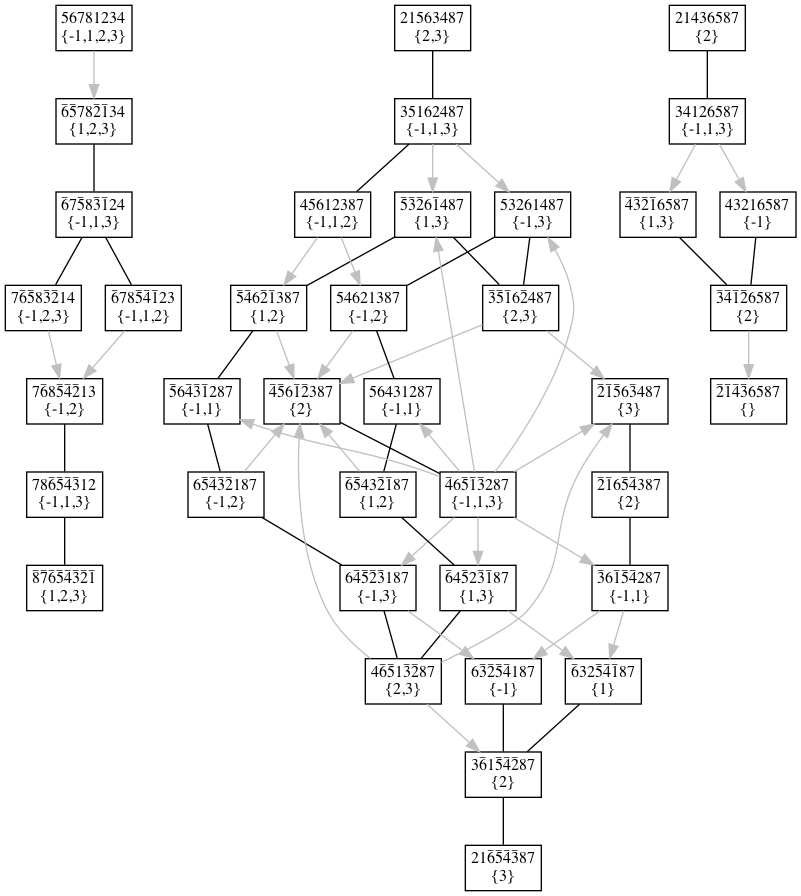}}}$
\end{center}
\caption{The box shows $\Gamma^\m(\WD_4)$, in which all edge weights are $\omega^\m(u,v)=1$.
One can compute $\Gamma^\n(\WD_4)$ from this picture using Theorem~\ref{d-dual-thm}.
}
\label{d4-fig}
\end{figure}

\begin{figure}
\begin{center}
\begin{tabular}{|l|llllll|}
\hline
$W$-graph & \#Vertices & \#Edges & Edge Weights & \#WCC & \#Cells & \#Molecules  \\
\hline
$\Gamma^\m(S_2)$ &  2 & 0 & $\varnothing$& 2 & 2 & 2  \\
$\Gamma^\m(S_3)$ &  4 & 3 & $\{1\}$& 2 & 3 & 3   \\
$\Gamma^\m(S_4)$ &  10 & 13  & $\{1\}$ & 3 & 5 & 5   \\
$\Gamma^\m(S_5)$ &  26 & 59 & $\{1\}$ & 3 & 7 & 7  \\
$\Gamma^\m(S_6)$ &  76 & 238 & $\{1\}$ & 4 & 11 & 11  \\
$\Gamma^\m(S_7)$ &  232 & 998 & $\{1\}$ & 4 & 15 & 15   \\
$\Gamma^\m(S_8)$ &  764 & 4230 & $\{1\}$ & 5 & 22 & 22   \\
$\Gamma^\m(S_9)$ &  2620 & 18467 & $\{-1, 1\}$ & 5 & 30 & 30  \\
$\Gamma^\m(S_{10})$ & 9496  & 83869   & $\{-1,1, 2, 3\}$ & 6  &  42 & 42  \\
&&&&&& \\
$\Gamma^\n(S_2)$ &  2 & 0 & $\varnothing$ & 2 & 2 & 2  \\
$\Gamma^\n(S_3)$ &  4 & 3  & $\{1\}$& 2 & 3 & 3  \\
$\Gamma^\n(S_4)$ &  10 & 13   & $\{1\}$ & 3 & 5 & 5    \\
$\Gamma^\n(S_5)$ &  26 & 57  & $\{1\}$& 3 & 7 & 7  \\
$\Gamma^\n(S_6)$ &  76 & 227  & $\{1\}$& 4 & 11 & 11 \\
$\Gamma^\n(S_7)$ &  232 & 931  & $\{1\}$& 4 & 15 & 15   \\
$\Gamma^\n(S_8)$ &  764 & 3863 & $\{1\}$ & 5 & 22 & 22   \\
$\Gamma^\n(S_9)$ &  2620 & 16437  & $\{1\}$& 5 & 30 & 30  \\
$\Gamma^\n(S_{10})$ & 9496  & 72182   & $\{1\}$ & 6  &  42 & 42 \\
&&&&&& \\
$\Gamma^\m(\W_2)$ &  6 & 6 & $\{1\}$ & 2 & 4 & 4  \\
$\Gamma^\m(\W_3)$ &  20 & 36  & $\{1\}$& 2 & 8 & 8  \\
$\Gamma^\m(\W_4)$ &  76 & 206   & $\{-1,1\}$ & 3 & 15 & 15    \\
$\Gamma^\m(\W_5)$ &  312 & 1217  & $\{-1,1\}$& 3 & 26 & 26  \\
$\Gamma^\m(\W_6)$ &  1384 & 7335  & $\{-1,1,2\}$& 4 & 44 & 46 \\
$\Gamma^\m(\W_7)$ &  6512 & 46066  & $\{-1,1,2,3\}$& 4 & 72 & 76   \\
&&&&&& \\
$\Gamma^\m(\WD_2)$ &  3 & 1 & $\{1\}$ & 2 & 3 & 3  \\
$\Gamma^\m(\WD_3)$ &  10 & 14  & $\{1\}$& 2 & 5 & 5  \\
$\Gamma^\m(\WD_4)$ &  38 & 87   & $\{1\}$ & 3 & 10 & 11    \\
$\Gamma^\m(\WD_5)$ &  156 & 534  & $\{1\}$& 3 & 16 & 18  \\
$\Gamma^\m(\WD_6)$ &  692 & 3262  & $\{1\}$& 4 & 29 & 36 \\
$\Gamma^\m(\WD_7)$ &  3256 & 20640  & $\{-1,1,2,3\}$& 4 & 45 & 59   \\
$\Gamma^\n(\WD_8)$ &  16200 & 137576  & $\{-1,1,2,3,4\}$& 5 & 75 & 109   \\
\hline
\end{tabular}
\end{center}
\caption{This table lists some data related to the $W$-graphs $\Gamma^\m(W)$ and $\Gamma^\n(W)$.
The sets in the fourth column are $\{ \omega(u,v) : u,v \in V\} \setminus\{0\}$ for each  graph
$\Gamma = (V,\omega, I)$.
The fifth column counts the weakly connected components in each directed graph.
The last column lists the number of molecules in each $W$-graph,
where a \emph{molecule} is a connected component for the subgraph in which we retain only
the edges $u \to v$ with a complementary edge $v \to u$.
We have omitted $\Gamma^\n(W)$ in types $\BC$ and $\D$
since all of the listed data is the same as for $\Gamma^\m(W)$ by Theorems~\ref{bcbc-dual-thm} and \ref{dd-dual-thm}. 
The vertex counts in column two are \cite[A000085, A000898, and A000902]{OEIS}.
The cell and molecule counts for $W=S_n$ appear to be the number of partitions of $n$ \cite[A000041]{OEIS}.
More exotically (and perhaps coincidentally), the molecule counts for $W=\WD_n$ match \cite[A162891]{OEIS} for $n\leq 8$.
}
\label{table-fig}
\end{figure}

When $W$ is finite with splitting field $\KK\subset \CC$,
we say that a $W$-graph $\Gamma$
is a \emph{Gelfand $W$-graph} 
if
$\cY(\Gamma)_\KK := \KK(x) \otimes_{\ZZ[x,x^{-1}]} \cY(\Gamma)$ is a Gelfand model for $\H_\KK(W)$.
\begin{corollary} If $W \in \{S_n ,\W_n, \WD_{2n+1}\}$ then $\Gamma^\m$ and $\Gamma^\n$ are Gelfand $W$-graphs.
\end{corollary}

It is notable that we can define these $W$-graphs in a somewhat uniform fashion.
Ignoring this prerogative,
Gelfand $W$-graphs exist for 
every  crystallographic
finite Coxeter group $W$, simply because
 all irreducible representations of $\H_\KK$ can be afforded by $W$-graphs \cite{Gyoja}.


It is an interesting open problem to classify the cells in  $\Gamma^\m$ and $\Gamma^\n$,
where a \emph{cell} in a $W$-graph is defined to be a strongly connected component for the underlying directed graph.
  This is because if $C \subset V$ is a cell in $\Gamma = (V,\omega,I)$ 
  then
  $\Gamma|_C := (C, \omega|_{C\times C}, I|_C)$
   is itself a $W$-graph, which defines a \emph{cell representation}
   $\cY(\Gamma|_C)$ of the algebra $\H$.
 
Our final results show that
in types $\BC$ and $\D$, the \emph{a priori} distinct cell classification problems for $\Gamma^\m$ and $\Gamma^\n$
are equivalent.
Two $W$-graphs $\Gamma = (V,\omega,I)$ and
$ \Gamma' = (V',  \omega',  I')$ are \emph{dual} via some
 bijection $V \to V'$, written $v\mapsto v'$,
 if $\omega'(u',v') = \overline{\omega(v,u)}$ and $ I'(v') = S \setminus I(v)$
for all $u,v \in V$.
In this case
 $v\mapsto v'$  sends the cells in $\Gamma$ to
the cells in $\Gamma'$.

\begin{theorem}\label{bcbc-dual-thm}
Assume $W=\W_n$. Then $\Gamma^\m$ and $\Gamma^\n$
are dual.
\end{theorem}

The  version of this duality for type $\D$ requires a brief technical digression.
Assume $W=\WD_n$
and write $\diamond$ for the Coxeter automorphism with $w \mapsto w^\diamond := s_0ws_0$.
This map interchanges $s_{-1}$ and $s_1$ while fixing the other simple generators of $\WD_n$.

 For $z \in\cG$
define $ \Asc^{\m|\diamond}(z)=\Asc^\m(z)$ and $ \Asc^{\n|\diamond}(z)=\Asc^\n(z)$
if $n + |\{ i \in [n] : |z(i)| > n\}| $ is divisible by 4 
and otherwise let $ \Asc^{\m|\diamond}(z)= \Asc^\m(z^\diamond) $ and $ \Asc^{\n|\diamond}(z)= \Asc^\n(z^\diamond) $. 
We will see in Section~\ref{revisit-sect} that the modified triples 
$\widetilde\Gamma^\m := (\cG,\omega^\m,  \Asc^{\m|\diamond})
$ and $ \widetilde\Gamma^\n := (\cG,\omega^\n,  \Asc^{\n|\diamond})$
are both  $\WD_n$-graphs.
They have
the same underlying directed graph structure as $\Gamma^\m$ and $\Gamma^\n$.

\begin{theorem}\label{dd-dual-thm}
Assume $W=\WD_n$. Then
 $\Gamma^\m$ is dual to  $\widetilde\Gamma^\n$ 
and
 $ \Gamma^\n$ is dual to $\widetilde\Gamma^\m$.
\end{theorem}

One can observe these dualities in 
Figure~\ref{abcd3-fig}, which shows $\Gamma^\m$ and $\Gamma^\n$ for $W \in \{\W_3,\D_3\}$.

\begin{corollary}
If $W=\WD_{n}$ when $n$ is odd, then  
$ \widetilde\Gamma^\m$ and $\widetilde\Gamma^\n$
are Gelfand $W$-graphs.
\end{corollary}

When proving Theorems~\ref{bcbc-dual-thm} and \ref{dd-dual-thm} in Section~\ref{revisit-sect}, we will identify the relevant duality maps $\dualityBC_n : \cG^\BC_{n}  \to \cG^\BC_{n} $
and 
$\dualityD_n : \cG^\D_{n}  \to \cG^\D_{n} $
explicitly. Referring to \eqref{iotaBC-eq} and \eqref{iotaD-eq} for the definitions of these bijections,
we can state the following corollary:

\begin{corollary}
A subset $C \subset \Gamma^\m(\W_n)$ is a cell if and only if
$\dualityBC_n(C)$ is a cell in $ \Gamma^\n(\W_n)$.
Likewise, a subset $C \subset \Gamma^\m(\WD_n)$ is a cell if and only if
$\dualityD_n(C)$ is a cell in $ \Gamma^\n(\WD_n)$.
\end{corollary}

By contrast,
the $S_n$-graphs $\Gamma^\m(S_n)$ and $\Gamma^\n(S_n)$ 
seem to be dual if and only if $n\leq 2$,
and there does not seem to be a simple relationship between their cells
in general.
We will study the cells in these Gelfand $W$-graphs more carefully in our followup work \cite{MZ}.
Here, we mention just one related conjecture, which holds at least for $n\leq 10$ (but not in types $\BC$ or $\D$).

%

%


\begin{conjecture}
The cell representations associated to $\Gamma^\m(S_n)$ and $\Gamma^\n(S_n)$
are all irreducible.
\end{conjecture}

The rest of this paper is organized as follows.
Section~\ref{prelim-sect} contains some preliminary results and background material.
 In Section~\ref{models-sect} we verify that the sets $\sP$ in Definition~\ref{main-thm1-def} are perfect models.
Section~\ref{foll-sect} constructs certain general $W$-graphs parametrized by subsets $J \subset S$ 
paired with linear characters $\sigma:W_J\to\{\pm1\}$. 
 In Section~\ref{duality-sect} we prove a duality theorem for these $W$-graphs.
 Sections~\ref{gel-sect} and \ref{revisit-sect} derive most of the results sketched in this introduction.
Section~\ref{equiv-sect} discusses a notion of equivalence for perfect models,
which will be studied further in \cite{MZ}.

\subsection*{Acknowledgements}

This work was partially supported by Hong Kong RGC Grant ECS 26305218.
We thank Brendan Pawlowski and Joel Brewster Lewis for many useful discussions.

\section{Preliminaries}\label{prelim-sect}

Our main references for this section are the article \cite{RV} and the books \cite{CCG,GeckPfeiffer}.

\subsection{Perfect involutions}\label{perfect-subsect}

The sets of perfect involutions $\sI(W,S)$ for all classical Weyl groups are described in
\cite[\S9]{RV}. We review this classification below. Recall that we identify each $w \in W$ with $(w,1) \in W^+$.
\ben
\item[($\A$)] The symmetric group $S_n$ of permutations of $[n] := \{1,2,\dots,n\}$
is a Coxeter group relative to the generators $\{s_1,s_2,\dots,s_{n-1}\}$ where $s_i:=(i,i+1)$.
The only non-identity element of $\Aut(S_n,\{s_1,s_2,\dots,s_{n-1}\})$ is the inner automorphism $\Ad(w_0)$
induced by  $w_0 = n\cdots 321$. 
When $n$ is odd every perfect involution in $S_n^+$ is central.
When $n$ is even there are three $S_n$-minimal perfect involutions: $1\in S_n$,
$s_1s_3s_5\cdots s_{n-1} \in S_n$, and $(1, \Ad(w_0))\in S_n^+$.

\item[($\BC$)] The group $\W_n$ of permutations of $\{-n,\dots,-1,0,1,\dots,n\}$ that commute with $x\mapsto-x$
is a Coxeter group relative to the generators $\{s_0,s_1,\dots,s_{n-1}\}$ 
with $s_0 := (-1,1)$ and $s_i := (i,i+1)(-i,-i-1)$ for $i>0$. 
The one-line representation of $w \in \W_n$ is the word $w(1)w(2)\cdots w(n)$ where we write $\bar i$ for $-i$.
When $n>2$ there are only trivial Coxeter automorphisms 
and  $w=w^{-1}\in \W_n$  if perfect if and only if the map
 $|w| : i \mapsto |w(i)|$ fixes all or no elements of $[n]$.
The $\W_n$-minimal perfect involutions consist of the $n+1$ elements
$\theta_i := \bar 1 \bar 2 \cdots \bar i (i+1)\cdots n$ for $i \in \{0,1,\dots,n\}$  plus $s_1s_3\cdots s_{n-1}$ if $n$ is even.

\item[($\D$)] The subgroup $\WD_n$ of $w \in \W_n$ such that $|\{ i \in [n] : w(i) < 0\}|$ is even 
is a Coxeter group relative to the generators $\{s_{-1},s_1,s_2,\dots,s_{n-1}\}$ where $s_{-1} := (1,-2)(-1,2)$
and $s_i $ for $i\geq 0$ is as in ($\BC$).
Assume $n>4$. Then the unique nontrivial Coxeter automorphism is $w\mapsto w^\diamond := s_0ws_0$.
The perfect involutions in $(\WD_n)^+$ are the pairs $(w, \varphi)$ with $\varphi \in \{1,\diamond\}$ and $w^{-1} = \varphi(w)$ such that $|w| \in S_n$ is $1$ or fixed-point-free.
The $\WD_n$-minimal perfect involutions consist of 
$\theta_i$ for each even $0\leq i\leq n$, $\hat \theta_i := (1 \bar2 \bar3 \cdots \bar i (i+1)\cdots n, \diamond)$ for each odd $1\leq i\leq n$,   plus both $s_1s_3s_5\cdots s_{n-1}$ and $s_{-1}s_3s_5\cdots s_{n-1}$ if $n$ is even.
\een
From now on, when we write $s_i$ for $i \in \ZZ$ we mean the permutations defined in
($\BC$) and ($\D$); everything in ($\A$) remains true on changing ``$s_i:=(i,i+1)$'' to ``$s_i:=(i,i+1)(-i,-i-1)$.''

Let $(W,S)$ be any Coxeter system and write $\sI=\sI(W,S)$.
The group $W$ acts on $\sI$ by
conjugation.
Rains and Vazirani prove that $\sI$ is a \emph{quasiparabolic $W$-set} relative to this action and the height function $\frac{1}{2}\ell$ \cite[Thm. 4.6]{RV}. Concretely, this means that the following holds:
\begin{lemma}[See \cite{RV}] \label{QP-lem} Suppose $s \in S$, $t \in \{wsw^{-1} : (w,s) \in W\times S\}$, and $z \in \sI$. Then:
\ben
\item[(a)] If $\ell(tzt) = \ell(z)$ then $tzt=z$.
\item[(b)] If $\ell(tzt) > \ell(z)$ and $\ell(stzts) < \ell(szs)$, then $szs= tzt$.
\een
\end{lemma}

\subsection{Descent sets}\label{descent-sect}

Assume $(W,S)$ is a finite classical Coxeter system,
so that $W$ is $S_n$, $\W_n$, or $\WD_n$. 
It is useful to recall an explicit description of the right descent sets for elements of  these groups:
 
\begin{proposition}[{See \cite[Props. 8.1.2 and 8.2.2]{CCG}}]
\label{old-des-prop}
Let $w \in W$ where $W\in \{S_n,\W_n,\WD_n\}$.
\begin{itemize} 
\item[(a)] If $i \in [n-1]$ 
then $\ell(ws_i) < \ell(w) $ if and only if $w(i) > w(i+1)$.

\item[(b)] If $W=\W_n$ then $\ell(ws_0) <\ell(w) $ if and only if $0 > w(1)$.

\item[(c)] If $W= \WD_n$ then $\ell(ws_{-1}) < \ell(w) $ if and only if 
 $-w(2) > w(1)$.
\end{itemize}
\end{proposition}

Recall the definitions of $\cF = \cF(W)$ and $\cG =\cG(W)$ from Definition~\ref{fg-def}.
Note that $\cG^\A_{n-1}  =\cG(S_n) \subset \cG^\BC_n = \cG(\W_n)\supset \cG^\D_n =  \cG(\WD_n)$.
We can characterize these sets
more explicitly. 

\begin{proposition}\label{new-des-prop1}
Suppose $z \in \cF^\BC_n=\left\{ z \in \W_{2n} : z=z^{-1}\text{ and }|z(i)| \neq i \text{ for all }i \in [2n]\right\}$.
Then $z  \in \cG^\BC_n$
if and only if 
there is an integer $0 \leq i \leq n$ with 
$ i \equiv n \modu 2)$ such that
\[
-n \leq z(n+1) < z(n+2) < \dots < z(n+i) \leq n
\quad\text{while}\quad
z(n+i+2j) =n+i+2j-1
\]
for 
each 
 integer $j>0$ with $n+i+2j \leq 2n$.
 \end{proposition}
 
 \begin{proof}
By definition $z=z^{-1}$ and $|z(i)| \neq i$ for all $i \in[n]$.
 If $z$ has the given properties then $z$ has no visible descents greater than $n$ so $z \in \cG^\BC_n$.
 Instead assume  $z \in \cF^\BC_n$ has no visible descents greater than $n$.
 If $n < a \leq 2n$ then $-n \leq z(a)$, since otherwise $a$ or $-z(a)$ would be a visible descent exceeding $n$.
 If there exists a maximal $a \in \ZZ$ with $z(a)+1 < a$,
 then  $z(i+1) < \min\{i,z(i)\}$ for some $i$ with $z(a) \leq i < a$. Therefore if $n \leq z(a)<a$ then $a = z(a)+1$.
 Let $i>0$ be maximal with $z(n+i) \leq n$ or set $i=0$ if no such integer exists.
 Then $-n\leq z(n+1)<z(n+2)<\dots<z(n+i) \leq n$ as otherwise some number in $\{n+1,n+2,\dots,2n-1\}$
 would be a visible descent, and we also have $
z(n+i+2j) =n+i+2j-1
$
for 
each 
 integer $j>0$ with $n+i+2j \leq 2n$, so $z$ has the desired form.
 \end{proof}

The preceding result suggests the following notation.

\begin{definition}\label{suggests-def}
Suppose $w =w^{-1} \in W^\BC_n$. Let 
\[a_1 >a_2  > \dots >a_p
\qquad\text{(respectively, $ b_1<b_2<\dots<b_q$)}
\]
be the numbers $a \in [n]$ with $w(a) = -a$ (respectively, $b \in [n]$ with $w(b) = b$).
Define $\underline w$ to be the unique element of $\cG^\BC_n$ mapping 
$-a_i \mapsto n + i$ for $i \in [p]$ and $b_i \mapsto n + p + i$ for $i \in [q]$ along with $c \mapsto w(c)$ for $c \in [n] \setminus \{ a_1,a_2,\dots,a_p, b_1,b_2,\dots,b_q\}$. We can write this down using the formula
\[ \underline w := \prod_{i=1}^p (a_i,n+i) (-a_i,-n-i) \cdot \prod_{i=1}^q (b_i,n+p+i)(-b_i,-n-p-i) \cdot 
\prod_{\substack{i \in [n-p-q] \\ i\text{ odd}}} s_{n+p+q+i} \cdot w .\] 
\end{definition}

For example, for $w=2134 = (1,2)(3)(4) \in S_4$ we have
\[
\underline{w} = (1,2)(3,5)(4,6)(7,8) = 21563487 \in \cG^\A_3.
\]
For $w =\bar 3 2 \bar 1 \bar 4 \bar5  = (1,-3)(3,-1)\cdot (2)(-2)\cdot (-4,4)\cdot (-5,5) \in \W_5$ we likewise have
\be\label{bc-under-eq}
\ba \underline{w} &= 
 (1,-3)(3,-1)\cdot (2,8)(-2,-8)\cdot (-4,7)(4,-7)\cdot (-5,6)(5,-6) \cdot (9,10)
 \\&=
\bar 3, 8, \bar 1, \bar 7, \bar 6, \bar 5, \bar 4, 2,10,9 \in \cG^\BC_5.
\ea
\ee
When $W \in \{S_n, \W_n\}$, the map $w \mapsto \underline w$ is a bijection 
\be\{ w=w^{-1} \in S_n\} \xrightarrow{\sim} \cG^\A_{n-1}
\quand
\{ w=w^{-1} \in \W_n\} \xrightarrow{\sim} \cG^\BC_{n}
.\ee
If we set  $e(w) := | \{ i \in [n]  : w(i) < -i\}|$ for $w \in \W_n$,
then the same map is a bijection
\be\label{dun-eq}
\left\{ w=w^{-1} \in \WD_n : e(w)\text{ is even}\right\} \sqcup
\left\{ w=w^{-1} \in \W_n -  \WD_n : e(w)\text{ is odd}\right\}
\xrightarrow{\sim} \cG^\D_n.
\ee

Finally, to do calculations with the modules $\cM$ and $\cN$ from Theorem~\ref{main-thm2},
it is helpful to have a direct characterization of the  sets 
$\Des^=(v)$, $\Asc^=(v)$, $\Des^<(v)$, and $\Asc^<(v)$.

\begin{proposition}\label{new-des-prop2}
Suppose  $v \in \cG$ where $W\in \{S_n, \W_n, \WD_n\}$.
\ben
\item[(a)] Suppose $s=s_i$ for $i \in [n-1]$. Then:
\begin{itemize}
\item $s \in \Des^=(v)$ if and only if $v(i)=i+1$ or $ v(i) = -i-1$.
\item $s \in \Asc^=(v)$ if and only if $n<v(i) < v(i+1)$ or $v(i) < v(i+1) < -n$.
\item $s \in \Des^<(v)$ if and only if $i+1\neq v(i)>v(i+1)\neq i$.
\item $s \in \Asc^<(v)$ in all other cases.

\end{itemize}

\item[(b)] Suppose $W=\WB_n$. Then 
$s_0 \in \Des^<(v)$ if $z(1) < 0$ and $s_0\in \Asc^<(v)$ if $v(1) >0$.

\item[(c)] Suppose $W=\WD_n$. Then:
\begin{itemize}
\item $s_{-1} \in \Des^=(v)$ if and only if $v(1) =2$ or $v(1)= - 2$.

\item $s_{-1}  \in \Asc^=(v)$ if and only if $v(1) < -n < n < v(2)$ or $ v(2) < -n < n < v(1)$.

\item $s_{-1}  \in \Des^<(v)$ if $s_{-1} \notin \Des^=(v) \sqcup \Asc^=(v)$ and $-v(2) > v(1)$.
%
%
\item $s_{-1}  \in \Asc^<(v)$ if  $s_{-1} \notin \Des^=(v) \sqcup \Asc^=(v)$ and $-v(2) < v(1)$.
\end{itemize}

\een
\end{proposition}

\begin{proof}
This  is straightforward from  \eqref{intro-des-1} and \eqref{intro-des-2}
using Propositions~\ref{old-des-prop} and \ref{new-des-prop1}. 
\end{proof}

\subsection{Parabolic induction}

Suppose $\cM$ is a right $\cA$-module and $\cN$ is a left $\cA$-module for some ring $\cA$.
Then one can form the tensor product $\cM \otimes_\cA \cN$
as the quotient of the additive group $\cM \times \cN$ by the subgroup generated by all elements of the form
$(M+M', N) - (M,N) - (M',N)$ or
$ (M,N+N') - (M,N)-(M,N')$ or
$ (Ma,N)-(M,aN)$
for $M,M'\in \cM$, $N,N'\in \cN$, and $a \in \cA$.
We write $M\otimes_\cA N$ for the element of this quotient represented by $(M,N)$.
The tensor product $\cM\otimes_\cA\cN$ is naturally an additive abelian group. 
When $\cB$ is another ring and $\cM$ is a $(\cB,\cA)$-bimodule,
 $\cM \otimes_\cA \cN$ is also a left $\cB$-module.

Let $(W,S)$ be a Coxeter system with Iwahori-Hecke algebra $\H$. Fix a subset $J \subseteq S$ and let $\H_J$ be the Iwahori-Hecke algebra of $(W_J,J)$.
%
Suppose $\cM$ is a left $\H_J$-module. The algebra $\H$ is itself  an $(\H,\H_J)$-bimodule 
so, as in \cite[\S9.1]{GeckPfeiffer}, we can form the induced left $\H$-module
\[\Ind_{J}^{S}(\cM) := \H \otimes_{\H_J} \cM.\]
Assume further that $\cM$  is free as a $\LL$-module with basis $\{ M_z : z \in Z\}$.
Then $\Ind_J^S(\cM)$ is free as an $\LL$-module with basis 
$\{ H_w \otimes_{\H_J} M_z : (w,z) \in W^J\times Z\}$
where \be
\label{W^J-eq}
W^J := \{ w \in W : \ell(ws) > \ell(w)\text{ for all }s \in J\}.
\ee
If $s \in S$ and $w \in W^J$, then either $sw \in W^J$ in which case 
\[ H_s  H_w \otimes_{\H_J} M_z = \begin{cases} H_{sw}\otimes_{\H_J} M_z  &\text{if }\ell(sw) > \ell(w), \\
H_{sw}\otimes_{\H_J} M_z  + (x-x^{-1}) H_{w}\otimes_{\H_J} M_z &\text{if }\ell(sw) < \ell(w),\end{cases}
\]
or $sw \notin W^J$ in which case $t := w^{-1} ws \in J$  (see \cite[Lem. 3.2]{HowlettYin}) and 
\[
H_s  H_w \otimes_{\H_J} M_z 
=
H_w \otimes_{\H_J} H_tM_z.
\]

\subsection{Canonical bases}

We have adapted the following definitions from \cite{Webster}.

\begin{definition}
Suppose $\cM$ is a free $\LL$-module.
A \emph{(balanced) pre-canonical structure} on $\cM$ consists of 
a \emph{bar involution} $\psi : \cM \to \cM$ that is an antilinear map with $\psi^2=1$,
along with  a \emph{standard basis} $\{ M_z\}_{z\in Z}$ with partially ordered index set $(Z,\leq)$ satisfying
\[ \psi(M_z) \in M_z + \sum_{y<z} \ZZ[x,x^{-1}] M_{y}\qquad\text{for all }z \in Z.\]
Assume further that $\cM$ is  an $\cA$-module for some $\LL$-algebra $\cA$ with its own pre-canonical structure in which the bar involution is also denoted $\psi$.
A pre-canonical structure
$(\psi, \{ M_z\}_{z\in Z})$ on $\cM$ is \emph{$\cA$-compatible}
if $\psi(AM) = \psi(A)\psi(M)$ for all $A \in \cA$ and $M \in \cM$.
\end{definition}

Pre-canonical structures are closely related to \emph{$P$-kernels} as defined in \cite{StanleyPKernel};
see \cite[\S2.2]{M2014}.

\begin{definition}
A \emph{canonical basis} for 
a pre-canonical structure
$(\psi, \{ M_z\}_{z\in Z})$ 
on 
a free $\LL$-module $\cM$
is a basis $\{\underline M_z\}_{z \in Z}$ with
$\underline M_z = \psi(\underline M_z) \in M_z + \sum_{y<z} x^{-1} \ZZ[x^{-1}] M_{y}$
for all $z \in Z$.
\end{definition}


Let $(W,S)$ be a Coxeter system with Iwahori-Hecke algebra $\H$.
The \emph{Bruhat order} on  $W$ is the transitive closure of the relation with $w<wt$ for all $w \in W$ and $t \in \{ wsw^{-1} : (w,s) \in W\times S\}$ with 
$\ell(w) < \ell(wt)$.
The basis $\{ H_w : w \in W\}$, in which $W$ is ordered by $<$, together with 
the bar operator $H\mapsto \overline H$ is a pre-canonical structure on 
$\H$; see \cite[\S6.1]{CCG}.
 The \emph{Kazhdan-Lusztig basis} from \cite{KL} is the unique canonical basis for this structure.
 This uniqueness is typical:
 
%

\begin{lemma}[{See \cite[Prop. 1.10]{Webster}}]\label{webster-lem}
A pre-canonical structure on a free $\LL$-module admits at most one canonical basis.
\end{lemma}

Although saying anything concrete about this basis is difficult in general, in many situations 
a unique canonical basis is guaranteed to exist.
A \emph{lower interval} in a poset $(Z,\leq)$ is a set of the form $\{ y \in Z: y \leq z\}$ for a fixed $z \in Z$.
The following lemma is equivalent to a result of Du.

\begin{lemma}[{See \cite[Thm. 1.2 and Rmk. 1.2.1(1)]{Du}}] \label{du-lem}
A pre-canonical structure $(\psi, \{ M_z\}_{z\in Z})$
admits a unique canonical basis if all lower intervals in the indexing poset $(Z,\leq)$ are finite.
\end{lemma}

This lemma applies, for example, to the pre-canonical structure on $\H$.
The following proposition summarizes the main applications of the material in this section.

 \begin{proposition}\label{induced-prop}
Fix a subset $J\subseteq S$, write $\H_J$ for the Iwahori-Hecke algebra of $(W_J,J)$, and assume $\cM$ is an $\H_J$-module
with a pre-canonical structure
 $(\psi, \{M_z\}_{z \in Z})$ that is $\H_J$-compatible.
 The basis $\{ H_w \otimes_{\H_J} M_z \}_{(w,z) \in W^J\times Z}$,
partially ordered such that $(w,z) \leq (w',z')$ if $w\leq w'$ in Bruhat order and or if 
$w=w'$ and $z\leq z'$, 
along with the antilinear map sending 
\[H  \otimes_{\H_J} M\mapsto \overline{H}\otimes_{\H_J} \psi(M)\quad\text{for }(H,M) \in \H\times \cM,\]
is an $\H$-compatible pre-canonical structure on $\Ind_J^S(\cM)$.

 \end{proposition}
 
 \begin{proof}
It is easy to check that the antilinear map  sending 
$(H,M)\mapsto (\overline{H}, \psi(M))$ preserves the kernel of $\cH \times \cM \to \Ind_J^S(\cM)$,
so descends to an involution of $ \Ind_J^S(\cM)$. If we denote this map again by $\psi$,
then it is clear from the definition that $\psi(H_1 \cdot H_2\otimes_{\H_J} M) = \overline{H_1H_2}\otimes_{\H_J} \psi(M) = \overline{H_1} \cdot \psi(H_2\otimes_{\H_J} M)$ for all $H_1,H_2 \in \H$ and $M\in\cM$.
Now observe that 
\[ \psi(H_w \otimes_{\H_J} M_z) = \overline{H_w} \otimes_{\H_J} \psi(M) \in H_w \otimes_{\H_J} M_z + \sum_{
\substack{
(v,y) \in W \times Z  \\
v \leq w,\ y\leq z \\ 
(v,y) \neq (w,z)}
} \LL H_v \otimes_{\H_J} M_y.
\]
This is almost the desired triangularity condition for a bar involution on $\Ind_J^S(\cM)$, except the sum on the right varies over
$ W\times Z$ rather than $ W^J\times Z$. This is no problem, however,
since whenever $v \in W$ but $v\notin W^J$ we can write $v=v'v''$ where $v' \in W^J$ has $v' < v$ and $v'' \in W_J$,
in which case $H_v \otimes_{\H_J} M_y = H_{v'} \otimes_{\H_J} H_{v''}M_y \in \sum_{y' \in Z} \LL H_{v'} \otimes M_{y'}$.
 \end{proof}

%
 
 \section{Results}

This section contains  proofs of the theorems in the introduction and a few more general results.

 \subsection{Models for classical groups}\label{models-sect}
 
We split the proof of  Theorem~\ref{main-thm1} into three parts
 for types $\A$, $\BC$, and $\D$ below. In view of the description of the perfect involutions for
 classical Weyl groups in Section~\ref{perfect-subsect},
it remains only to show that $\sum_{(J,z_{\min},\sigma)} \Ind_{C_{W_J}(z_{\min})}^W\Res_{C_{W_J}(z_{\min})}^{W_J}(\sigma) = \sum_{\chi  \in \Irr(W)} \chi$ when the first sum ranges over the model triples listed in 
Definition~\ref{main-thm1-def}. We check this directly after recalling a few standard facts about the irreducible characters of $S_n$, $\W_n$, and $\WD_n$.

 Fix an integer $n\in \NN$. 
 A \emph{partition} of $n$ is a weakly decreasing sequence of nonnegative integers 
 $\lambda = (\lambda_1 \geq \lambda_2 \geq \dots\geq 0)$ with $\sum_i \lambda_i = n$.
 To indicate that $\lambda$ is a partition of $n$ we write $\lambda \vdash n$.
 We also let $\ell(\lambda)$ denote the number of parts $\lambda_i>0$.
 The \emph{(Young) diagram} of $\lambda \vdash n$ is the set of positions 
 $
 \D_\lambda :=\{ (i,j) \in [\ell(\lambda)] \times \ZZ: 1 \leq j \leq \lambda_i\},$
 oriented as in a matrix.
We let $\chi^\lambda$ denote the usual irreducible character of $S_n$ indexed by $\lambda \vdash n$ 
(as given, for example, by \cite[Def. 5.4.4]{GeckPfeiffer}), so that $\lambda \mapsto \chi^\lambda$ is a bijection from the set of partitions of $n$ to $\Irr(S_n)$.

Given $p,q\in \NN$ with $p+q=n$, we identify 
$S_p \times S_q = \langle s_1,\dots,s_{p-1},s_{p+1},\dots,s_{n-1}\rangle \subset S_n$.
If $G$ and $H$ are finite groups, then for any characters $\chi : G\to \CC$ and $\psi  : H \to \CC$
 we write $\chi \times \psi  :G\times H \to \CC$
for the character  that maps $(g,h) \mapsto \chi(g) \psi(h) $.
 The \emph{Littlewood-Richardson coefficients} $c_{\lambda\mu}^\nu\geq 0$
 are the integers with 
 $\Ind_{S_p\times S_q}^{S_n}(\chi^\lambda \times \chi^\mu) = \sum_{\nu \vdash n} c_{\lambda\mu}^\nu \chi^\nu.$
The well-known \emph{Pieri formula} states that
 if $\mu = (1,1,\dots,1)\vdash q$, so that $\chi^\mu = \sgn$,
 then $c_{\lambda\mu}^\nu$ is either one or zero, with $c_{\lambda\mu}^\nu = 1$ if and only if $\D_\nu$ is formed by adding $q$ positions to $\D_\lambda$ in distinct rows.
 
When $W=S_n$, Theorem~\ref{main-thm1} is equivalent to the main result in
\cite{IRS}. The proof below, included for completeness, is not significantly different from the short argument in \cite{IRS}.

 \begin{proof}[Proof of Theorem~\ref{main-thm1} in type $\A$; see \cite{IRS}]
 Let $z_{\min} = s_1s_3s_5\cdots s_{2k-1}$ and write $\one $ for the trivial character.
Then $\Ind_{C_{S_{2k}}(z_{\min})}^{S_{2k}}(\one) = \sum_\lambda \chi^\lambda$ where the sum is over the partitions $\lambda \vdash 2k$ with all even parts \cite[Lem. 1]{IRS}. As   $H_k := C_{S_{2k} \times S_{n-2k}}(z_{\min}) \cong C_{S_{2k}}(z_{\min}) \times S_{n-2k}$,
it follows that
 $\Ind_{H_k}^{S_n} \Res_{H_k}^{S_{2k}\times S_{n-2k}}(\one \times \sgn)=\Ind_{S_{2k}\times S_{n-2k}}^{S_n} (\Ind_{C_{S_{2k}}(z_{\min})}^{S_{2k}}(\one) \times \sgn)=  \sum_\lambda \chi^\lambda$
 where the sum is over $\lambda \vdash n$ with $n-2k$ odd parts \cite[Lem. 2]{IRS}.
Summing over $k$ gives $\sum_{\lambda \vdash n}\chi^\lambda$ as needed.
 \end{proof}
 
 A \emph{bipartition} of $n$ is a pair $(\lambda,\mu)$ of partitions with $\lambda \vdash p$ and $\mu \vdash q$ for some $p,q \in \NN$ satisfying $p+q=n$.
 To indicate this situation, we write $(\lambda,\mu) \vdash n$. 
 The irreducible characters $\chi^{(\lambda,\mu)}$ of $\W_n$ are uniquely indexed by the bipartitions $(\lambda,\mu)\vdash n$.
We define $\chi^{(\lambda,\mu)}$ as in \cite[\S5.5]{GeckPfeiffer},
 so that $\one = \chi^{((n),\emptyset)}$ and $\sgn = \chi^{(\emptyset,(1,1,\dots,1))}$.
Fix $p,q\in \NN$ with $p+q=n$ and identify 
\[
 \ba
 \W_p \times S_q &= \langle s_0, s_1,\dots,s_{p-1},s_{p+1},\dots,s_{n-1}\rangle \subset \W_n,
 \\
 \W_p \times \W_q &= \langle s_0, s_1,\dots,s_{p-1},t_p,s_{p+1},\dots,s_{n-1}\rangle \subset \W_n,
\ea 
 \]
  where we set $t_0 := s_0$, $t_n:=1$, and $t_p := s_p \cdots s_2s_1s_0s_1s_2\cdots s_p  \in \W_n$ for $0<p<n$.
 
 \begin{lemma}\label{bc-irr-lem}
 Suppose $(\lambda,\mu)\vdash p$. Then $ \Ind_{\W_p \times S_q}^{\W_n}( \chi^{(\lambda,\mu)}\times \sgn)=\sum_{(\nu,\gamma)} \chi^{(\nu,\gamma)}$  where the sum is over all bipartitions $(\nu,\rho)\vdash n$ such that $\D_\nu$ is formed by 
 adding $ k$ positions to $\D_\lambda$ in distinct rows and $\D_\gamma$ is formed by adding $l$ positions to $\D_\mu$ in distinct rows, where $k+l = q$.
 \end{lemma}
 
 This lemma is a straightforward exercise and probably well-known. We include a short proof.

 \begin{proof}
 Using the construction of $\Irr(\W_n)$ in \cite[\S5.5.4]{GeckPfeiffer},
one can show that 
 $\Ind_{S_n}^{\W_n}(\chi^\nu) = \sum_{(\lambda,\mu)\vdash n} c_{\lambda\mu}^\nu \chi^{(\lambda,\mu)}
$
and
$
 \Ind_{\W_p\times \W_q}^{\W_n} ( \chi^{(\lambda,\alpha)} \times \chi^{(\mu,\beta)}) = \sum_{(\nu,\gamma) \vdash n} c_{\lambda\mu}^\nu c_{\alpha\beta}^{\gamma} \chi^{(\nu,\gamma)}$ for all $\nu \vdash n$, $(\lambda,\alpha)\vdash p$, and $(\mu,\beta)\vdash q$. 
 Thus
 $\Ind_{S_n}^{\W_n}(\sgn) = \sum_{p+q=n} \chi^{(1^p,1^q)}$
 where $1^n := (1,1,\dots,1)\vdash n$.
Now apply these facts to
$  \Ind_{\W_p \times S_q}^{\W_n}( \chi^{(\lambda,\mu)}\times \sgn)=\Ind_{\W_p \times \W_q}^{\W_n}( \chi^{(\lambda,\mu)} \times \Ind_{S_q}^{\W_q}(\sgn))  $.
 \end{proof}

  \begin{proof}[Proof of Theorem~\ref{main-thm1} in type $\BC$]
  Again let $z_{\min} = s_1s_3s_5\cdots s_{2k-1}$.
  It follows as a special case of \cite[Prop. 1]{Baddeley}
that $\Ind_{C_{\W_{2k}}(z_{\min})}^{\W_{2k}}(\one) $ is the sum of $ \chi^{(\lambda,\mu)}$ over all $(\lambda,\mu) \vdash 2k$ in which  $\lambda$ and $\mu$ both have all even parts.
As
 $H_k := C_{\W_{2k} \times S_{n-2k}}(z_{\min}) \cong C_{\W_{2k}}(z_{\min}) \times S_{n-2k}$, 
 Lemma~\ref{bc-irr-lem} implies that
 $\Ind_{H_k}^{\W_n} \Res_{H_k}^{\W_{2k}\times S_{n-2k}}(\one \times \sgn)=\Ind_{\W_{2k}\times S_{n-2k}}^{\W_n} \(\Ind_{C_{\W_{2k}}(z_{\min})}^{\W_{2k}}(\one) \times \sgn\)=  \sum_{(\lambda,\mu)} \chi^{(\lambda,\mu)}$
 where the sum is over all $(\lambda,\mu) \vdash n$ in which 
$\lambda$ and $\mu$ together have $n-2k$ odd parts in total.
Summing this over all $k\in \NN$ with $0\leq 2k \leq n$  gives $\sum_{(\lambda,\mu) \vdash n}\chi^{(\lambda,\mu)}$ as needed.
  \end{proof}
  
  An \emph{unordered bipartition} of $n$ is a set $\{\lambda,\mu\}= \{\mu,\lambda\}$ in which $\lambda \vdash p$ and $\mu \vdash q$ for some $p,q \in \NN$ with $p+q=n$. To indicate this situation we write
  $\{\lambda,\mu\} \vdash n$. 
  When $n$ is even and $\nu\vdash n/2$, we consider the singleton set $\{\nu\} = \{\nu,\nu\} $ to be an unordered bipartition of $n$.
  
  For each $\{\lambda,\mu\}\vdash n$, define $\chi^{\{\lambda,\mu\}} = \Res_{\WD_n}^{\W_n}(\chi^{(\lambda,\mu)})$.
  If $n$ is odd then $\{\lambda,\mu\}\mapsto \chi^{\{\lambda,\mu\}}$ is a bijection 
  from the set of unordered bipartitions of $n$ to $\Irr(\WD_n)$.
  If $n$ is even and $\nu \vdash n/2$ then $\chi^{\{\nu\}}$ decomposes as the sum of two distinct irreducible characters $\chi^{\{\nu\}}_+$ and $\chi^{\{\nu\}}_-$. In this case, 
  every irreducible character of $\WD_n$ is uniquely given either by $\chi^{\{\lambda,\mu\}}$
  for some $\{\lambda,\mu\}\vdash n$ with $\lambda \neq \mu$ or by 
  $\chi^{\{\nu\}}_+$ or  $\chi^{\{\nu\}}_-$ for some $\nu \vdash n/2$.
For either parity of $n$ we have 
  \be\label{d-irr-eq} \Ind_{\WD_n}^{\W_n}(\chi^{\{\lambda,\mu\}}) = \chi^{(\lambda,\mu)} + \chi^{(\mu,\lambda)}
  \quand 
   \Ind_{\WD_n}^{\W_n}(\chi^{\{\nu\}}_\pm)  =
\chi^{(\nu,\nu)}\ee
for all $\{\lambda,\mu)\vdash n$ with $\lambda \neq \mu$ 
and $\nu \vdash n/2$  \cite[\S5.6]{GeckPfeiffer}.
  
    \begin{proof}[Proof of Theorem~\ref{main-thm1} in type $\D$]
    Assume $n$ is odd and
     let $z_{\min} = s_1s_3s_5\cdots s_{2k-1}$.
    The centralizer
    $C_{\WD_{2k} \times S_{n-2k}}(z_{\min})$ is equal to $H_k:= C_{\W_{2k} \times S_{n-2k}}(z_{\min})$
while
$\Res_{H_k}^{\WD_{2k}\times S_{n-2k}}(\one \times \sgn) = 
\Res_{H_k}^{\W_{2k}\times S_{n-2k}}(\one \times \sgn)$,
so
$\sum_{k=0}^{\lfloor \frac{n}{2}\rfloor} \Ind_{\WD_n}^{\W_n} \Ind_{H_k}^{\WD_n} \Res_{H_k}^{\WD_{2k}\times S_{n-2k}}(\one \times \sgn) = \sum_{(\lambda,\mu) \vdash n} \chi^{(\lambda,\mu)}$. Since $n$ is odd,  \eqref{d-irr-eq} implies that
$\sum_{k=0}^{\lfloor \frac{n}{2}\rfloor} \Ind_{H_k}^{\WD_n} \Res_{H_k}^{\WD_{2k}\times S_{n-2k}}(\one \times \sgn) = \sum_{\{\lambda,\mu\} \vdash n} \chi^{\{\lambda,\mu\}}$
as needed.
    \end{proof}
    
\subsection{$\H$-modules and $W$-graphs}\label{foll-sect}

Throughout this section we fix a finite Coxeter system $(W,S)$, a subset $J \subset S$,
and a linear character $\sigma : W_J \to \{\pm 1\}$.
Define $\cK^J := W^J \times \sI_J$
where $W^J$ is the set of minimal length coset representatives from \eqref{W^J-eq}
and $\sI_J$ is the set of perfect involutions in $(W_J)^+$.
The main results in this section construct 
a $W$-graph with vertex set $\cK^J$.

If $\tau = (w,z) \in \cK^J$, 
then $z$ is $W_J$-conjugate to a unique $W_J$-minimal $z_{\min} \in \sI_J$ and we
define
\be\label{h-def-eq}
\h(\tau) := \ell(w) + \tfrac{1}{2}(\ell(z) - \ell(z_{\min})).\ee
If $s \in S$ and $w \in W^J$ then either $sw \in W^J$ or $w^{-1} sw \in J$,
so for $\tau=(w,s) \in \cK^J$ we may define
\[ s\tau := \begin{cases} 
(sw,z) & \text{if }sw \in W^J \\ 
(w,tzt) &\text{if }sw \notin W^J\text{ and }t=w^{-1}sw .
 \end{cases}
\]
This formula extends to an
action of $W$ on $\cK^J$ with the property that $\h(s\tau) - \h(\tau) \in \{-1,0,1\}$
whenever $s \in S$ and $\tau \in \cK^J$.
Indeed, restricted to the orbit of $(1,z_{\min})$ for some $W_J$-minimal $z_{\min} \in \sI_J$,
our $W$-action is isomorphic to the usual action of $W$ on the set of left cosets $W / H$ for $H=C_{W_J}(z_{\min})$ via the map 
$(u,v\cdot z_{\min} \cdot v^{-1}) \mapsto uvH$.

\begin{proposition}\label{h-eq-prop}
If $s \in S$ and $\tau \in \cK^J$ then $\h(s\tau) = \h(\tau)$ if and only if $s\tau = \tau$.
\end{proposition}

\begin{proof}
Suppose $\tau = (w,z).$ If $\h(s\tau) = \h(\tau)$
then we must have $sw \notin W^J$ and $\ell(tzt) = \ell(z)$ for $t=w^{-1}sw \in J$, but this implies 
that $tzt = z$ by Lemma~\ref{QP-lem}, so $s\tau = \tau$.
\end{proof}

\begin{proposition}\label{same-orbit-prop}
Two elements $\tau = (w,z)$ and $\tau'=(w',z')$ in $\cK^J$
belong to the same $W$-orbit if and only if $z$ and $z'$ belong to the same $W_J$-conjugacy class in $\sI_J$.
\end{proposition}

\begin{proof}
If $w \in W^J$ and $s \in S$ are such that $\ell(sw)  < \ell(w)$, then $sw \in W^J$ (see \cite[Lem. 3.2]{HowlettYin}),
so $\tau$ belongs to the $W$-orbit of $(1,z)$ and $\tau'$ belongs to the $W$-orbit $(1,z')$. 
Thus if $z$ and $z'$ are in the same $W_J$-conjugacy class then $\tau$ and $\tau'$ are in the same $W$-orbit.
The converse holds since applying the map  $(v,y)\mapsto y$
to $\tau$ and $s\tau$ for $s \in S$ yields $W_J$-conjugate elements of $\sI_J$.
\end{proof}


\begin{definition}\label{prec-def}
The \emph{quasiparabolic Bruhat order} on $\sI_J$ (see \cite[\S5]{RV}) is the transitive closure of the relation
with $z < tzt$ for all $z \in \sI_J$ and all reflections $t \in  \{ wsw^{-1} : (w,s) \in W_J \times J\}$ with $\ell(z) < \ell(tzt)$.
We write $\prec$ for the partial order on $\cK^J$ that has $\tau = (w,z) \prec \tau' = (w',z')$
if and only if $\tau$ and $\tau'$ belong to the same $W$-orbit such that $\h(\tau) < \h(\tau')$ and either
\begin{itemize}
\item  $w<w'$ in the Bruhat order on $W$, or 
\item $w=w'$ and $z<z'$ in the quasiparabolic Bruhat order on $\sI_J$.
\end{itemize}
The dependence of this partial order on $J$ is suppressed in our notation.
\end{definition}

It is convenient to adopt the following terminology: for $\tau = (w,z) \in \cK^J$, define
\be\label{asc-des-def}
\ba
\Asc^<(\tau;\sigma) &:= \{s \in S: \h(s\tau) > \h(\tau)\},\\
\Des^<(\tau;\sigma) &:= \{s \in S: \h(s\tau) < \h(\tau)\},\\
\Asc^=(\tau;\sigma) &:= \{s \in S: \h(s\tau) = \h(\tau)\text{ and }\sigma(w^{-1} sw) = -1\},  \\
\Des^=(\tau;\sigma) &:= \{s \in S: \h(s\tau) = \h(\tau)\text{ and }\sigma(w^{-1} sw) = 1\}.
\ea
\ee
We refer to the elements of $\Asc^<(\tau;\sigma) $ and $\Des^<(\tau;\sigma) $ as strict ascents and strict descents of $\tau$,
and to elements of $\Asc^=(\tau;\sigma) $ and $\Des^=(\tau;\sigma) $ as weak ascents and weak descents.

\begin{theorem}\label{m-thm}
Assume $(W,S)$ is a finite Coxeter system. Choose a subset $J \subset S$ and a linear character $\sigma : W_J \to \{\pm1\}$.
Let $\cM^{\sigma}$
be the free $\LL$-module with 
basis $\{ M^\sigma_\tau : \tau \in \cK^J\}$.
\ben
\item[(a)] There is a unique $\H$-module structure on $\cM^{\sigma}$ in which
\[
H_s M^\sigma_\tau = \begin{cases}
M^\sigma_{s\tau}  &\text{if }s \in \Asc^<(\tau;\sigma)\\
M^\sigma_{s\tau} + (x-x^{-1}) M^\sigma_\tau &\text{if }s \in \Des^<(\tau;\sigma) \\
-x^{-1} M^\sigma_\tau &  \text{if }s \in \Asc^=(\tau;\sigma) \\
x M^\sigma_\tau& \text{if }s \in \Des^=(\tau;\sigma)
\end{cases}
\quad\text{for all }s \in S\text{ and } \tau \in \cK^J.\]

\item[(b)] There is a unique $\H$-compatible bar operator $M \mapsto \overline {M}$ on $\cM^{\sigma}$
that
fixes the basis element $ M^\sigma_{\tau} \in \cM^{\sigma}$
for each $\tau \in \cK^J$ with $\Des^<(\tau;\sigma)=\varnothing$.
This bar operator is an involution with 
\be\label{ut-eq} \overline{M^\sigma_\tau} \in M^\sigma_\tau + \sum_{\upsilon \prec \tau} \ZZ[x,x^{-1}] M^\sigma_\upsilon.\ee

\item[(c)] Finally, the $\H$-module $\cM^{\sigma}$ has a unique basis $\{ \underline M^\sigma_\tau : \tau \in \cK^J\}$ satisfying
\[\overline{\underline M^\sigma_\tau} = \underline M^\sigma_\tau \in M^\sigma_\tau + \sum_{ \upsilon \prec \tau} x^{-1} \ZZ[x^{-1}] M^\sigma_\upsilon\qquad\text{for all }\tau \in \cK^J.\]
\een
\end{theorem}

\begin{proof}
When $J=S$, we abbreviate by writing $M^\sigma_{z}$ in place of $M^\sigma_{(1,z)}$ for 
$z \in \sI(W,S)$.
It suffices to prove part (a) in the case when $J=S$, since when $J\subset S$ 
the module $\cM^\sigma(W)$ that we describe can be identified with $\Ind_J^S(\cM^\sigma(W_J))$ 
via by the linear map sending $M^\sigma_{(w,z)} \mapsto H_w \otimes_{\cH_J} M^\sigma_{z}$ for $(w,z) \in \cK^J$.
Thus, assume $J=S$ so that $\sigma$ is a linear character of $W$.
In this special case, one can deduce part (a) by essentially repeating the proof of \cite[Thm. 7.1]{RV}
with minor adjustments. In particular, 
when $\sigma $ is $\one$ or $\sgn$, part (a) is literally just \cite[Thm. 7.1]{RV} combined with \cite[Thm. 4.6]{RV}.
We sketch the argument for $\sigma \notin\{\one,\sgn\}$ below,   glossing over some of the details   in \cite[\S7.1]{RV}.

We need to check that when $J=S$ our formula for the action on $\H$ on $\cM^\sigma$ obeys the quadratic relations
$(H_s-x)(H_s+x^{-1}) = 0$ and the braid relations $H_sH_tH_s \cdots = H_t H_sH_t \cdots $ (where the number of factors 
on both sides is the order of $st$ in $W$)
for all $s, t \in S$. The quadratic relations are easy to verify directly.

For the braid relations, we may assume without loss of generality that $S= \{s,t\}$ and 
then replace $\cM^\sigma$
by the submodule $\cT^\sigma$ spanned by the $W$-conjugacy class of a single $W$-minimal element $z_{\min}\in \sI(W,S)$.
Let $m$ be the order of $st$ in $W$.
As explained in the proof of  \cite[Thm. 7.1]{RV}, the centralizer of $z_{\min}$ in $W$
then has the form $\langle (st)^{m'}\rangle$ or $\langle (st)^{m'}, s\rangle$ or $\langle (st)^{m'},t\rangle$ for some $m'$ dividing $m$.
If $m/m'$ is odd then we can assume that  $\sigma(s) = \sigma(t)$,
since  $\sigma(s) \neq \sigma(t)$ only if $m$ is even but if $m'$ is even 
then $\sigma$ has the same restriction to $C_{W}(z_{\min})$ as $\one$ or $\sgn$.
Therefore, we can pass from $W$ to the quotient $W/\langle (st)^{m'}\rangle$
and
 further reduce to the case when $C_{W}(z_{\min})$ is  $\{1\}$ or $\langle s\rangle$ or $\langle t\rangle$.
 Following the proof of  \cite[Thm. 7.1]{RV},
 the desired braid relations then hold since $\cT^\sigma$ can be identified with either the regular representation of $\cH=\cH(\langle s,t\rangle)$
or with a module induced from a one-dimensional representation of $\cH(\langle s\rangle)$ or $\cH(\langle t\rangle)$.

This completes the proof of part (a). To prove part (b) we continue to assume $J=S$.
Fix an element $(z,\theta) \in \sI=\sI(W,S)$ and let $ (z_{\min},\theta)$
be the unique $W$-minimal element  in the same $W$-conjugacy class.
Let $z_{\min} = s_1s_2\cdots s_l$ be any reduced expression
and define $\|z\|_\sigma :=  x^{a_+} / (-x)^{a_-}$ where $a_\pm = |\{ i \in [l] : \sigma(s_i) = \pm1\}|$. Because $\sigma(s)  = \sigma(t)$ whenever $s,t \in S$ and $st$ has odd order, $\|z\|_\sigma$ does not depend on
the reduced expression.
We claim that setting 
\be\label{bar-op-concrete} \overline{M^\sigma_{(z,\theta)}} := \|z\|_\sigma \cdot \overline {H_z} \cdot M_{(z^{-1},\theta)}\ee
extends to an  antilinear map $\cM^\sigma \to \cM^\sigma$ 
satisfying $\overline{M^\sigma_{(z_{\min},\theta)}} = M^\sigma_{(z_{\min},\theta)}$
and $\overline{HM} = \overline{H}\cdot \overline{M}$ for all $H \in \H$ and $M \in \cM^\sigma$.
When $\sigma$ is $\one$ or $\sgn$ this claim is precisely \cite[Thm. 4.19]{Marberg},
and the proof of that result carries over to our slightly more general setting with minimal changes.
Now, 
since $ \overline{M^\sigma_{(z,\theta)}} = \overline{H_w} \cdot M^\sigma_{(z_{\min},\theta)}$
for any $w\in W$ of minimal length with $w\cdot z_{\min} \cdot \theta(w)^{-1} = z$,
and since $\overline{H_w} \in H_w + \sum_{v<w} \LL H_v$,
it follows that $\overline{M^\sigma_{\tau}} \in M^\sigma_{\tau} + \sum_{\upsilon < \tau} \LL M^\sigma_\upsilon$
for any $\tau \in \sI$, where $<$ is the quasiparabolic Bruhat order,
and that the map \eqref{bar-op-concrete} is an involution.
Thus, if $J=S$ then $\cM^\sigma$ has an $\H$-compatible pre-canonical structure with standard basis 
$ \{ M^\sigma_\tau\} $ ordered by the quasiparabolic Bruhat order.

Now let $J\subset S$ and define $< $ to be the partial order on $\cK^J$ that has $(v,y) \leq (w,z)$ if $v\leq w$ in  Bruhat order
 or if $v=w$ and $y\leq z$ in  quasiparabolic Bruhat order. 
This is slightly different from
$\prec$ in Definition~\ref{prec-def}.
Observe that $\tau=(w,z) \in \cK^J$ has $\Des^<(\tau;\sigma) = \varnothing$ if and only if $w=1$ and $z \in \sI_J$ is $W_J$-minimal.
Identifying $M^\sigma_{(w,z)} \in \cM^\sigma(W)$ with  $H_w \otimes_{\cH_J} M^\sigma_{z} \in \H\otimes_{\H_J} \cM^\sigma(W_J)$, we deduce from the previous paragraph and Proposition~\ref{induced-prop}
that $\cM^\sigma(W)$ has an $\H$-compatible bar operator $M \mapsto \overline{M}$
that fixes $M^\sigma_\tau$ if $\Des^<(\tau;\sigma)=\varnothing$
and  has
 \be
 \label{ut-eq2}
 \overline{M^\sigma_\tau} \in M^\sigma_\tau + \sum_{\upsilon \leq \tau} \ZZ[x,x^{-1}] M^\sigma_\upsilon
 \quad\text{for all }\tau \in \cK^J
 .\ee
 For each $\tau=(w,z) \in \cK^J$, there is a unique $\tau_{\min}=(1,z_{\min}) \in \cK^J$ in the $W$-orbit of $\tau$ with $\Des^<(\tau_{\min};\sigma) = \varnothing$
 and we have $M^\sigma_\tau = H_v M^\sigma_{\tau_{\min}}$ for some $v \in W$.
 Since then $\overline{M^\sigma_\tau} = \overline{H_v} M^\sigma_{\tau_{\min}}$, it follows that our $\H$-compatible
 bar operator on $\cM^\sigma$ is unique and an involution.
 It remains only to check that \eqref{ut-eq2} implies the stronger triangular property \eqref{ut-eq},
 but this is also easy to derive from the identities 
  $\overline{M^\sigma_\tau} = \overline{H_v} M^\sigma_{\tau_{\min}}$ and
   $ \overline{H_v}  \in H_v + \sum_{u\in W, u<v} \LL H_u$.
 
This completes the proof of part (b).
We have shown that $(M\mapsto \overline M, \{ M^\sigma_\tau\}_{\tau \in \cK^J})$ with $\cK^J$ ordered by $\prec$
is a pre-canonical structure on $\cM^\sigma$.
Since the poset $(\cK^J,\prec)$ clearly has finite lower intervals,
part (c) is immediate from Lemma~\ref{du-lem}.
\end{proof}

\begin{example}
If $J = \varnothing$ then $\cM^\sigma$ is isomorphic to the left regular representation of $\H$
and the canonical basis $\{ \underline M^\sigma_\tau : \tau \in \cK^J\}$ may be identified with the usual Kazhdan-Lusztig basis from \cite{KL}.
More generally, if $\sigma$ is $\one$ or $\sgn$, then 
  $\{ \underline M^\sigma_{(w,1)}: w \in W^J\}$ may be identified 
with the parabolic Kazhdan-Lusztig basis studied in \cite{Deodhar}.
\end{example}

Our next goal is to explain how the structure constants of $\H \times \cM^\sigma \to \cM^\sigma$ in the basis 
$\{ \underline M^\sigma_\tau\}$ define a $W$-graph on the set $\cK^J$. 
This is already done for the special case $J=S$ in \cite{Marberg}. What we describe 
for general subsets $J \subset S$, apart from differences in notation, is what one gets by applying Howlett and Yin's
method of $W$-graph induction from \cite{HowlettYin,HowlettYin2} to the constructions in \cite{Marberg}.
However, since presenting things in terms of $W$-graph induction makes it somewhat difficult to do explicit computations, we will develop our new $W$-graphs from scratch
without assuming any background from \cite{HowlettYin,HowlettYin2,Marberg}.

Define $\m^\sigma$ to be the  $\cK^J\times \cK^J$ matrix with entries $\m^\sigma_{\upsilon\tau} \in \ZZ[x^{-1}]$
such that
\[\underline M^\sigma_\tau = \sum_{\upsilon \in \cK^J} \m^\sigma_{\upsilon\tau}  M^\sigma_\upsilon
.
\]
Write
$ \mu^{\m|\sigma}_{\upsilon\tau} := [x^{-1}]  \m^\sigma_{\upsilon\tau} \in \ZZ$
for the coefficient of $x^{-1}$ in $\m^\sigma_{\upsilon\tau}$.
For the next lemma, we define
\[
\Des^{\m|\sigma}(\tau) := \Des^<(\tau;\sigma) \sqcup \Des^=(\tau;\sigma)
\quand
\Asc^{\m|\sigma}(\tau) := \Asc^<(\tau;\sigma) \sqcup \Asc^=(\tau;\sigma).
\]
Let $\delta_{ij}$ denote the usual Kronecker delta function.

\begin{lemma}\label{cb-lem}
Suppose $s \in S $ and $\tau \in \cK^J$. Let $\underline H_s = H_s + x^{-1}$.
\ben
\item[(a)] If $s \in \Des^{\m|\sigma}(\tau) $ then $\underline H_s \underline M^\sigma_\tau =  
(x+x^{-1}) \underline M^\sigma_\tau$.

\item[(b)] If $s \in \Asc^{\m|\sigma}(\tau)$ then 
$\ds\underline H_s \underline M^\sigma_\tau =
(1-\delta_{\tau,s\tau}) \underline M^\sigma_{s \tau} + 
\sum_{\substack{\upsilon \prec \tau  \\ s \in \Des^{\m|\sigma}(\upsilon)}}
\mu^{\m|\sigma}_{\upsilon\tau} \underline M^\sigma_\upsilon.$

\een
\end{lemma}

\begin{proof}
For part (b), we observe that the difference of the two sides is a bar invariant element of $ x^{-1}\ZZ[x^{-1}]\spanning \{ M^\sigma_\tau : \tau \in \cK^J\}$, but the only such element is zero. 
For part (a), our argument is similar to the proof of \cite[Theorem 5.3]{HowlettYin}.
First let $s \in \Des^=(\tau;\sigma)$. Then $\underline M^\sigma_\tau = M^\sigma_\tau + \sum_{\upsilon \prec\tau} f_{\upsilon\tau} \underline M^\sigma_\upsilon$
for certain $f_{\upsilon\tau} \in x^{-1}\ZZ[x^{-1}]$. Since $H_s M^\sigma_\tau = x M^\sigma_\tau$, we have
\[ (H_s - x) \underline M^\sigma_\tau = \sum_{ \upsilon \prec \tau} f_{\upsilon\tau}\cdot (H_s-x)\cdot \underline M^\sigma_\upsilon= \sum_{ \upsilon \prec \tau} f_{\upsilon\tau} \cdot (\underline H_s-x-x^{-1})\cdot \underline M^\sigma_\upsilon\]
which by induction using both parts can be written as 
\[  (H_s - x) \underline M^\sigma_\tau =-(x+x^{-1}) \sum_{\substack{ \upsilon \prec \tau \\ s \in \Asc^{\m|\sigma}(\upsilon)}} f_{\upsilon\tau} \underline M^\sigma_\upsilon + \sum_{\substack{\upsilon \in \cK^J \\ s \in \Des^{\m|\sigma}(\upsilon)}} a_{\upsilon\tau} f_{\upsilon\tau} \underline M^\sigma_\upsilon\]
for certain integers $a_{\upsilon\tau}\in \ZZ$. The left hand side is bar invariant by construction, but since $\{ \underline M^\sigma_\tau\}_{\tau \in \cK^J}$ is a basis, the right hand side is bar invariant only if \[\overline{f_{\upsilon\tau}} = f_{\upsilon\tau} \in x^{-1}\ZZ[x^{-1}] \cap x\ZZ[x]= 0\quad\text{whenever $s \in \Asc^{\m|\sigma}(\upsilon)$.}\]
Thus $(H_s-x)\underline M^\sigma_\tau = 0$ so $\underline H_s \underline M^\sigma_\tau = (x+x^{-1})\underline M^\sigma_\tau$.
Now let $s \in \Des^<(\tau;\sigma)$. Then $s \in \Asc^<(s\tau)$ so by induction using (b)
we have $\underline M^\sigma_\tau = \underline H_{s} \underline M^\sigma_{s\tau} - \sum_{\upsilon\prec s\tau, s \in \Des^{\m|\sigma}(\upsilon)} \mu^{\m|\sigma}_{\upsilon,s\tau}\underline M^\sigma_\upsilon.$
The desired identity follows by induction using part (a) and $ \underline H_s \underline H_s  = (x+x^{-1}) \underline H_s$.
\end{proof}

Given 
$\upsilon,\tau \in \cK^J$
and   $s \in S$,
define $P^{\m|\sigma}_{\upsilon\tau} := x^{\h(\tau)-\h(\upsilon)} \m^\sigma_{\upsilon\tau} \in \LL$
and let 
\[
\Sigma_{\upsilon\tau}^{\m|\sigma}(s) := \sum_{\substack{
\kappa \in  \cK^J 
\\
\upsilon \preceq \kappa \prec s\tau \\
s \in \Des^{\m|\sigma}(\kappa)
}}  x^{\h(\tau)-\h(\kappa)} \cdot \mu^{\m|\sigma}_{\kappa,s\tau} \cdot P^{\m|\sigma}_{\upsilon\kappa}.
\]
By
comparing coefficients of $M^\sigma_\upsilon$ on both sides of the identities in Lemma~\ref{cb-lem},
one obtains the following formulas for $P^{\m|\sigma}_{\upsilon\tau} $. We omit the proof, which is straightforward.
\begin{corollary}
\label{sigma-cor}
 Let $\upsilon,\tau \in  \cK^J$ and $s \in S$.
\ben
\item[(a)] If $s \in \Des^{\m|\sigma}(\tau)$ then $P^{\m|\sigma}_{\upsilon\tau} = P^{\m|\sigma}_{s\upsilon,\tau}$.

\item[(b)] If $s \in \Des^<(\tau;\sigma) \cap \Des^{\m|\sigma}(\upsilon)$ then 
$ P^{\m|\sigma}_{\upsilon\tau} = x^2 P^{\m|\sigma}_{\upsilon,s\tau} + P^{\m|\sigma}_{s\upsilon,s\tau} -\Sigma_{\upsilon\tau}^{\m|\sigma}(s) .$

\item[(c)] If $s \in \Des^<(\tau;\sigma) \cap \Asc^<(\upsilon;\sigma)$ then 
$ P^{\m|\sigma}_{\upsilon\tau}  =  P^{\m|\sigma}_{\upsilon,s\tau} + x^2P^{\m|\sigma}_{s\upsilon,s\tau} - \Sigma_{\upsilon\tau}^{\m|\sigma}(s) .$

\item[(d)] If $s \in \Des^{\m|\sigma}(\tau) \cap \Asc^=(\upsilon;\sigma)$ then 
$ P^{\m|\sigma}_{\upsilon\tau} =  0 .$

\een
\end{corollary}



The previous two results make it possible to compute $\underline M^\sigma_\tau $ and $\m^\sigma_{\upsilon\tau}$
by induction in a relatively efficient manner.
The following properties are immediate by induction using Corollary~\ref{sigma-cor},
together with the fact that $\underline M^\sigma_\tau = \sum_{\upsilon\in\cK^J} \m^\sigma_{\upsilon\tau}M^\sigma_\upsilon \in M^\sigma_\tau +\sum_{\upsilon \prec \tau} x^{-1}\ZZ[x^{-1}] M^\sigma_\upsilon$.

\begin{corollary}\label{e-cor}
Suppose $\upsilon,\tau \in \cK^J$.
Then 
$P^{\m|\sigma}_{\tau\tau} = 1$,
$P^{\m|\sigma}_{\upsilon\tau} = 0$ if $\upsilon \not \preceq \tau$,
and 
$P^{\m|\sigma}_{\upsilon\tau} \in \ZZ[x^{2}]$.
If $\h(\tau) -\h(\upsilon)$ is even or $\upsilon \not \prec \tau$ then $\mu^{\m|\sigma}_{\upsilon\tau} = 0$,
and
if $\upsilon \prec\tau$ then $\deg( P^{\m|\sigma}_{\upsilon\tau} )< \h(\tau)-\h(\upsilon)$.
\end{corollary}


Suppose $\alpha \in \Aut(W,S)$.
Then $\alpha(J) \subset S$ and $\alpha(\sI_J) = \sI_{\alpha(J)}$, and 
we can consider the linear character $\sigma\circ \alpha^{-1} : W_{\alpha(J)} \to \{ \pm 1\}$.
For $\tau = (w,z) \in \cK^J$, let $\alpha(\tau) := (\alpha(w),\alpha(z)) \in \cK^{\alpha(J)}$.

\begin{corollary}\label{alpha-cor}
Suppose $\upsilon,\tau \in \cK^J$ and $\alpha \in \Aut(W,S)$. Then 
$\m^{\sigma}_{\upsilon\tau} = \m^{\sigma\circ \alpha^{-1}}_{\alpha(\upsilon)\alpha(\tau)}$.
\end{corollary}

\begin{proof}
As $s \in S$ is a strict/weak ascent/descent of $\tau$ if and only if $\alpha(s)$ is 
the same type of ascent/descent for $\alpha(\tau)$ (relative to $\sigma \circ \alpha^{-1}$),
this follows from Corollary~\ref{sigma-cor} by induction.
\end{proof}

We require one other lemma.

\begin{lemma}\label{delta-lem}
If $\upsilon,\tau \in \cK^J$ and $s \in \Des^{\m|\sigma}(\tau) \cap \Asc^{\m|\sigma}(\upsilon)$ then 
$\mu^{\m|\sigma}_{\upsilon\tau} = \delta_{s\upsilon,\tau}$.
\end{lemma}

\begin{proof}
Suppose $s\in \Des^{\m|\sigma}(\tau) \cap \Asc^{<}(\upsilon)$.
Then $P^{\m|\sigma}_{\upsilon\tau} = P^{\m|\sigma}_{s\upsilon,\tau}$ so
$\m^\sigma_{\upsilon\tau} = x^{-1} \m^\sigma_{s\upsilon,\tau}$.
In this case, if $s\upsilon = \tau$ then $\mu^{\m|\sigma}_{\upsilon\tau} = 1$
and otherwise $\m^\sigma_{\upsilon\tau} \in x^{-2}\ZZ[x^{-1}]$ so $\mu^{\m|\sigma}_{\upsilon\tau}=0$.
If instead $s \in \Des^{\m|\sigma}(\tau) \cap \Asc^{=}(\upsilon;\sigma)$, then
$s\upsilon=\upsilon \neq \tau$ and $\mu^{\m|\sigma}_{\upsilon\tau}=0$ by
Corollary~\ref{sigma-cor}(d).
\end{proof}

Define $\omega^{\m|\sigma} : \cK^J \times  \cK^J \to \ZZ$ by the formula
\[
\omega^{\m|\sigma}(\tau,\upsilon) :=\begin{cases} \mu^{\m|\sigma}_{\upsilon\tau} + \mu^{\m|\sigma}_{\tau\upsilon}
&\text{if }\Asc^{\m|\sigma}(\tau)\not\subset\Asc^{\m|\sigma}(\upsilon) \\
0&\text{if }\Asc^{\m|\sigma}(\tau)\subset \Asc^{\m|\sigma}(\upsilon).
\end{cases}
\]
Observe that if $\Asc^{\m|\sigma}(\tau)\not\subset \Asc^{\m|\sigma}(\upsilon)\not\subset\Asc^{\m|\sigma}(\tau)$
then $\omega^{\m|\sigma}(\tau,\upsilon) = \omega^{\m|\sigma}(\upsilon,\tau)$.

\begin{definition}\label{quasi-admissible-def}
A $W$-graph  $\Gamma=(V,\omega,I)$ is \emph{quasi-admissible} 
  if
  \ben
  \item[(1)] the corresponding directed graph is bipartite,
\item[(2)] we always have $\omega(u,v) \in \ZZ$,
\item[(3)]  we have
  $\omega(u,v) =0$ 
 whenever  $I(u)\subset I(v)$,
 and 
 \item[(4)] we have
$\omega(u,v) = \omega(v,u)$
  whenever
   $I(u)\not\subset I(v) \not\subset I(u)$.
   \een
A quasi-admissible $W$-graph is \emph{admissible} in the sense of \cite{Stembridge}
if every $\omega(u,v) \in \NN$.
\end{definition}


\begin{theorem}\label{mw-thm}
The $S$-labeled graph  $\Upsilon^{\m|\sigma}:=(\cK^J, \omega^{\m|\sigma}, \Asc^{\m|\sigma})$ is a
quasi-admissible $W$-graph.
The linear map with $Y_\tau \mapsto \underline M^\sigma_\tau$ is an isomorphism of $\H$-modules $\cY(\Upsilon^{\m|\sigma}) \xrightarrow{\sim} \cM^{\sigma}$.
\end{theorem}

This $W$-graph is usually not admissible.

\begin{proof}
To prove that $\Upsilon^{\m|\sigma}$ is a $W$-graph
and that the linear map with $Y_\tau \mapsto \underline M^\sigma_\tau$ is an $\H$-module isomorphism $\cY(\Upsilon^{\m|\sigma}) \xrightarrow{\sim} \cM^{\sigma}$,
it suffices to show that  
\be\label{hmo-eq} \underline H_s \underline M^\sigma_\tau = \sum_{\substack{
\upsilon \in \cK^J \\ 
s \in \Des^{\m|\sigma}(\upsilon)}
} \omega^{\m|\sigma}(\tau,\upsilon) \underline M^\sigma_\upsilon
\quad\text{for all $\tau \in \cK^J$ and $s \in \Asc^{\m|\sigma}(\tau)$}.
\ee
To this end,
suppose $\upsilon,\tau \in \cK^J$ with $ \h(\tau) \leq \h(\upsilon)$.
Then
$\omega^{\m|\sigma}(\tau,\upsilon)  = \mu^{\m|\sigma}_{\tau\upsilon}$ if 
there is some $s \in \Asc^{\m|\sigma}(\tau)\cap \Des^{\m|\sigma}(\upsilon)$ and 
$\omega^{\m|\sigma}(\tau,\upsilon)=0$ otherwise;
in the former case $\mu^{\m|\sigma}_{\tau\upsilon} = \delta_{\upsilon,s\tau}$
by Lemma~\ref{delta-lem}.
Thus
$\omega^{\m|\sigma}(\tau,\upsilon) = 1 $
if 
$\tau \prec s\tau = \upsilon$ for some $s \in S $
and $\omega^{\m|\sigma}(\tau,\upsilon) = 0$ otherwise.
On the other hand, if $\upsilon \in \cK^J$ with $\h(\upsilon)< \h(\tau)$,
then
$\omega^{\m|\sigma}(\tau,\upsilon) = \mu^{\m|\sigma}_{\upsilon\tau}$ if 
there exists $s\in \Asc^{\m|\sigma}(\tau)\cap \Des^{\m|\sigma}(\upsilon)$
and otherwise $\omega^{\m|\sigma}(\tau,\upsilon)=0$.
Hence, if $s\in \Asc^<(\tau;\sigma)$ then
\[ \sum_{\upsilon\in\cK^J\text{ and } s \in \Des^{\m|\sigma}(\upsilon)} \omega^{\m|\sigma}(\tau,\upsilon) \underline M^\sigma_\upsilon
=
\underline M^\sigma_{s\tau} + \sum_{s \in \Des^{\m|\sigma}(\upsilon)} \mu^{\m|\sigma}_{\upsilon\tau}\underline M^\sigma_\upsilon\]
which is equal to $\underline H_s \underline M^\sigma_\tau$ by Lemma~\ref{cb-lem}.
If
$s\in \Asc^=(\tau;\sigma)$ then \eqref{hmo-eq} follows similarly.

To show that the $W$-graph $\Upsilon^{\m|\sigma}$  is quasi-admissible, it remains only to check that 
the corresponding directed graph is bipartite.
This holds since Corollary~\ref{e-cor} implies
that $\tau \to \upsilon$ is a directed edge in $\Upsilon^{\m|\sigma}$
only if $\h(\tau) \not\equiv \h(\upsilon) \modu 2)$.
\end{proof}

We say that $u\leftrightarrow v$ is a \emph{bidirected edge} in a $W$-graph $\Gamma = (W,\omega,I)$
if $\omega(u,v) \neq 0 \neq \omega(v,u)$.

\begin{corollary}\label{mw-cor}
Suppose $\upsilon,\tau \in \cK^J$.
Then $\upsilon \leftrightarrow \tau$ is a bidirected edge in $\Upsilon^{\m|\sigma}$
if and only if $ \upsilon \prec s \upsilon = \tau$ for some $s \in S$ and $\Des^{\m|\sigma}(\upsilon) \cap \Asc^{\m|\sigma}(\tau)$ is nonempty.
\end{corollary}

Replacing the linear character $\sigma : W_J \to\{\pm 1\}$ by $\sgn \cdot \sigma$
gives dual forms of the above constructions. To refer to these, our convention is to replace each ``m'' by ``n'' and write 
\[\cN^\sigma := \cM^{\sgn \cdot \sigma}, \quad 
N_\tau^{\sigma} := M_\tau^{\sgn\cdot \sigma}, \quad
\underline N_\tau^{\sigma} := \underline M_\tau^{\sgn\cdot \sigma},\quad
\n^\sigma_{\upsilon\tau} := \m^{\sgn\cdot\sigma}_{\upsilon\tau},\quad 
\mu^{\n|\sigma}_{\upsilon\tau} := \mu^{\m|\sgn\cdot\sigma}_{\upsilon\tau},\]
and so forth.
This means that $\cN^{\sigma}=\LL\spanning\{ N^\sigma_\tau : \tau \in \cK^J\}$ 
is the free $\LL$-module with the unique
$\H$-module structure satisfying
\be
H_s N^\sigma_\tau = \begin{cases}
N^\sigma_{s\tau}  &\text{if }s \in \Asc^<(\tau;\sigma)\\
N^\sigma_{s\tau} + (x-x^{-1}) N^\sigma_\tau &\text{if }s \in \Des^<(\tau;\sigma) \\
xN^\sigma_\tau &  \text{if }s \in \Asc^=(\tau;\sigma) \\
-x^{-1} N^\sigma_\tau& \text{if }s \in \Des^=(\tau;\sigma)
\end{cases}
\quad\text{for all }s \in S\text{ and } \tau \in \cK^J,
\ee
where the weak/strict ascent/descent sets are as in \eqref{asc-des-def}.
 This $\H$-module has a unique $\H$-compatible bar operator $N \mapsto \overline {N}$,
which turns out to be an involution,
that
fixes $ N^\sigma_{\tau} \in \cN^{\sigma}$
for each $\tau \in \cK^J$ with $\Des^<(\tau)=\varnothing$.
Finally, the elements $\underline N_\tau^{\sigma} = \sum_{\upsilon \in \cK^J}  \n^\sigma_{\upsilon\tau}N^\sigma_{\upsilon}$
for $\tau \in \cK^J$ provide the unique basis of $\cN^\sigma$ satisfying 
\[\overline{\underline N^\sigma_\tau} = \underline N^\sigma_\tau \in N^\sigma_\tau + \sum_{ \upsilon \prec \tau} x^{-1} \ZZ[x^{-1}] N^\sigma_\upsilon\qquad\text{for all }\tau \in \cK^J.\]
As usual $ \mu^{\n|\sigma}_{\upsilon\tau} = [x^{-1}]  \n^\sigma_{\upsilon\tau} \in \ZZ$
is the coefficient of $x^{-1}$ in $\n^\sigma_{\upsilon\tau}$.
For $\tau \in \cK^J$ define
\[
\Des^{\n|\sigma}(\tau) := \Des^<(\tau;\sigma) \sqcup \Asc^=(\tau;\sigma)
\qquand
\Asc^{\n|\sigma}(\tau) := \Asc^<(\tau;\sigma) \sqcup \Des^=(\tau;\sigma),
\]
so that $\Des^{\n|\sigma}(\tau) = \Des^{\m|\sgn \cdot\sigma}(\tau)$
and $\Asc^{\n|\sigma}(\tau) = \Asc^{\m|\sgn \cdot\sigma}(\tau)$.
To do calculations with the basis $\{\underline N^\sigma_\tau\}$, it is helpful to restate 
Lemma~\ref{cb-lem}
with  $\sigma$ replaced by $\sgn \cdot \sigma$:

\begin{lemma}\label{cb-lem2}
Suppose $s \in S $ and $\tau \in \cK^J$. Let $\underline H_s = H_s + x^{-1}$.
\ben
\item[(a)] If $s \in \Des^\n(\tau) $ then $\underline H_s \underline N^\sigma_\tau =  
(x+x^{-1}) \underline N^\sigma_\tau$.

\item[(b)] If $s \in \Asc^\n(\tau)$ then 
$\ds\underline H_s \underline N^\sigma_\tau =
(1-\delta_{\tau,s\tau}) \underline N^\sigma_{s \tau} + 
\sum_{\substack{\upsilon \prec \tau  \\ s \in \Des^\n(\upsilon)}}
\mu^\n_{\upsilon\tau} \underline N^\sigma_\upsilon.$
\een
\end{lemma}


Continuing our notational conventions,
define $\omega^{\n|\sigma} : \cK^J \times  \cK^J \to \ZZ$ by the formula
\[
\omega^{\n|\sigma}(\upsilon,\tau) :=\begin{cases} \mu^{\n|\sigma}_{\upsilon\tau} + \mu^{\n|\sigma}_{\tau\upsilon}
&\text{if }\Asc^{\n|\sigma}(\upsilon)\not\subset\Asc^{\n|\sigma}(\tau) \\
0&\text{if }\Asc^{\n|\sigma}(\upsilon)\subset \Asc^{\n|\sigma}(\tau).
\end{cases}
\]
Again, we have $\omega^{\n|\sigma} = \omega^{\m|\sgn\cdot\sigma}$.
Also let $\Upsilon^{\n|\sigma} :=(\cK^J, \omega^{\n|\sigma}, \Asc^{\n|\sigma}).$
For future reference, we repeat the versions of Theorem~\ref{mw-thm} and Corollary~\ref{mw-cor}
given by replacing $\sigma$ by $\sgn\cdot \sigma$.

\begin{theorem}\label{nw-thm}
The $S$-labeled graph $\Upsilon^{\n|\sigma}$ is a
quasi-admissible $W$-graph.
The linear map with $Y_\tau \mapsto \underline N^\sigma_\tau$ is an isomorphism of $\H$-modules $\cY(\Upsilon^{\n|\sigma}) \xrightarrow{\sim} \cN^{\sigma}$.
\end{theorem}

\begin{corollary}
Suppose $\upsilon,\tau \in \cK^J$.
Then $\upsilon \leftrightarrow \tau$ is a bidirected edge in $\Upsilon^{\n|\sigma}$
if and only if $ \upsilon \prec s \upsilon = \tau$ for some $s \in S$ and $\Des^{\n|\sigma}(\upsilon) \cap \Asc^{\n|\sigma}(\tau)$ is nonempty.
\end{corollary}

\subsection{Duality}\label{duality-sect}

In this section $(W,S)$ continues to be a finite Coxeter system with Iwahori-Hecke algebra $\H$,
and $\sigma : W_J \to\{\pm 1\}$
is a fixed linear character for some subset $J\subset S$.
There exists a unique $\LL$-algebra automorphism $\Theta : \H \to \H$
with 
\[\Theta(H_s) = - \overline{H_s} = -H_s^{-1} = -H_s + x-x^{-1}
\quand
\Theta(H_w) = (-1)^{\ell(w)} \overline{H_w}\]
for all $s \in S$ and $w \in W$. Recall that $\cK^J := W^J \times \sI_J$. Define  
$ \Phi^{\sigma} : \cM^{\sigma} \to \cN^{\sigma}
$ and
$\Psi^{\sigma} : \cN^{\sigma} \to \cM^{\sigma}$
to be the $\LL$-linear maps satisfying
\[ \Phi^{\sigma}(M^\sigma_\tau) = (-1)^{\h(\tau)} \overline{N^\sigma_\tau}
\quand
\Psi^{\sigma} (N^\sigma_\tau) = (-1)^{\h(\tau)} \overline{M^\sigma_\tau}
\quad\text{for $\tau \in \cK^J $.}\]
The following lemma extends several observations in \cite[\S3]{Marberg} which apply to the case $J=S$.

\begin{lemma}\label{phi-psi-lem}
Suppose $H \in \H$, $M \in \cM^{\sigma}$, and $N \in \cN^{\sigma}$. The following properties hold:
\ben
\item[(a)] One has
$\Phi^{\sigma}(HM) = \Theta(H) \Phi^{\sigma}(M)$ and $\Psi^{\sigma}(HN) = \Theta(H) \Psi^{\sigma}(N)$.

\item[(b)] It holds that
$\Phi^{\sigma}(\overline M) = \overline{\Phi^{\sigma}(M)}$ and $\Psi^{\sigma}(\overline N) = \overline{ \Psi^{\sigma}(N)}$.

\item[(c)] The maps $\Phi^{\sigma}:\cM^{\sigma} \to \cN^{\sigma}$ and $\Psi^{\sigma} :\cN^{\sigma} \to \cM^{\sigma}$ are inverse bijections.
\een
\end{lemma}

\begin{proof}
Because $\Theta$ is an algebra automorphism, it suffices to check part (a) when $H = H_s$ for some $s \in S$ and
when
 $M = M^\sigma_\tau$ and $N=N^\sigma_\tau$ for some $\tau \in \cK^J$. This is straightforward from 
 the explicit formulas in Theorem~\ref{m-thm}.
 
  Since $\Phi^\sigma(M^\sigma_\tau) \in (-1)^{\h(\tau)} N^\sigma_\tau + \sum_{\upsilon \prec \tau} \ZZ[x,x^{-1}] N^\sigma_\upsilon$
 for all $\tau \in \cK^J$ by Theorem~\ref{m-thm}(b), the map $\Phi^\sigma$ is a linear bijection  $\cM^\sigma \to \cN^\sigma$.
 Part (a) implies that we also have $(\Phi^\sigma)^{-1}(HN) =\Theta(H) (\Phi^\sigma)^{-1}(N)$ for all $H \in \H$ and $N \in \cN^\sigma$.
 It follows that 
 the map $M \mapsto (\Phi^\sigma)^{-1}(\overline{\Phi^\sigma(M)})$
 is another $\H$-compatible bar operator on $\cM^\sigma$ fixing each $M^\sigma_\tau$ with $\h(\tau) = 0$,
 so we must have $ (\Phi^\sigma)^{-1}(\overline{\Phi^\sigma(M)}) = \overline{M}$ for all $M \in \cM^\sigma$
 by Theorem~\ref{m-thm}(b), or equivalently
$\overline{\Phi^\sigma(M)} = \Phi^\sigma( \overline{M})$.
 One can derive the other identity in part (b) by a similar argument.
 
Part (b) implies that $\Psi^\sigma(\Phi^{\sigma}(M^\sigma_\tau)) =
 (-1)^{\h(\tau)} \Psi^\sigma(\overline{N^\sigma_\tau})
 =
  (-1)^{\h(\tau)} \overline{\Psi^\sigma(N^\sigma_\tau)} = M^\sigma_\tau$
  and likewise that $\Phi^\sigma(\Psi^{\sigma}(N^\sigma_\tau))= N^\sigma_\tau$
  for all $\tau \in \cK^J$
  since
  the bar operators on $\cM^\sigma$ and $\cN^\sigma$ are involutions.
   Thus $\Phi^{\sigma}$ and $\Psi^{\sigma}$ are inverse maps, as claimed in part (c).
\end{proof}

For $\tau \in \cK^J$ define 
\[\underline {\hat M}^\sigma_\tau :=  \sum_{\upsilon \in \cK^J} (-1)^{\h(\tau) - \h(\upsilon)}  \cdot \overline{\n^{\sigma}_{\upsilon\tau}} \cdot M^\sigma_\upsilon
\quand
\underline {\hat N}^\sigma_\tau :=  \sum_{\upsilon \in \cK^J} (-1)^{\h(\tau) - \h(\upsilon)}  \cdot \overline{\m^{\sigma}_{\upsilon\tau}} \cdot N^\sigma_\upsilon.
\]

\begin{corollary}\label{phi-psi-cor}
If $\tau \in \cK^J$ then $\underline {\hat M}^\sigma_\tau =  (-1)^{\h(\tau)} \Psi^{\sigma}(\underline N^\sigma_\tau)$
and
$\underline {\hat N}^\sigma_\tau =  (-1)^{\h(\tau)} \Phi^{\sigma}(\underline M^\sigma_\tau)$.
\end{corollary}

\begin{proof}
Evidently $\underline {\hat M}^\sigma_\tau = \overline{(-1)^{\h(\tau)} \Psi^{\sigma}(\underline N^\sigma_\tau)}$,
which  is equal to
\[(-1)^{\h(\tau)} \overline{ \Psi^{\sigma}(\underline N^\sigma_\tau)}
=
(-1)^{\h(\tau)}  \Psi^{\sigma}(\overline{\underline N^\sigma_\tau})
=
(-1)^{\h(\tau)}  \Psi^{\sigma}(\underline N^\sigma_\tau)\]
by Lemma~\ref{phi-psi-lem} and the definition of $\underline N^\sigma_\tau$. The identity for 
$\underline {\hat N}^\sigma_\tau$ follows similarly.
\end{proof}

\begin{corollary}
The set $\left\{ \underline {\hat M}^\sigma_\tau : \tau \in \cK^J\right\}$ is the unique basis of $\cM^{\sigma}$ satisfying
\[\overline{\underline {\hat M}^\sigma_\tau} = \underline {\hat M}^\sigma_\tau \in M^\sigma_\tau + \sum_{ \upsilon \prec \tau} x\ZZ[x] M^\sigma_\upsilon\qquad\text{for all }\tau \in \cK^J.\]
The set $\left\{ \underline {\hat N}^\sigma_\tau : \tau \in \cK^J\right\}$ is the unique basis of $\cN^{\sigma}$ satisfying
\[\overline{\underline {\hat N}^\sigma_\tau} = \underline {\hat N}^\sigma_\tau \in N^\sigma_\tau + \sum_{ \upsilon \prec \tau} x\ZZ[x] N^\sigma_\upsilon\qquad\text{for all }\tau \in \cK^J.\]
\end{corollary}

\begin{proof}
The fact that $\underline {\hat M}^\sigma_\tau$ and $\underline {\hat N}^\sigma_\tau$ have these properties 
is straightforward from Lemma~\ref{phi-psi-lem}, Corollary~\ref{phi-psi-cor}, and the bar-invariance of 
$\underline { M}^\sigma_\tau$ and $\underline { N}^\sigma_\tau$.
Next, observe that if $\{X_\tau : \tau \in \cK^J\}$ is a basis of $\cN^\sigma$ with $\overline{X_\tau} = X_\tau \in N^\sigma_{\tau} + \sum_{\upsilon \prec \tau} x\ZZ[x] N^\sigma_\upsilon$ then 
$(-1)^{\h(\tau)} \overline{\Psi^\sigma(X_\tau)}  = (-1)^{\h(\tau)} \Psi^\sigma(X_\tau) $
must coincide with $\underline M_\tau$ by the uniqueness assertion in Theorem~\ref{m-thm}(c).
Since $\Phi^\sigma$ and $\Psi^\sigma$ are inverse maps, it follows that $X_\tau = (-1)^{\h(\tau)} \Phi^\sigma(\underline M_\tau) = \underline {\hat N}^\sigma_\tau$
by Corollary~\ref{phi-psi-cor}.
Our claim about the uniqueness of the basis $\{ \underline {\hat M}^\sigma_\tau : \tau \in \cK^\sigma\}$ of $\cM^\sigma$ follows by substituting $\sgn\cdot \sigma$ for $\sigma$.
\end{proof}

Let $w_0$ denote the unique longest element of $W$ and let $w_J$ denote the unique longest element of $W_J$. Both elements are involutions and we have $\theta(w_0) = w_0$ for all $\theta \in \Aut(W,S)$ along with $\theta(w_J) = w_J$ for all $\theta \in \Aut(W_J,J)$. 
In addition,  $\ell(vw_0) = \ell(w_0v)= \ell(v) - \ell(w_0)$ for all $v \in W$
while $\ell(vw_J) = \ell(w_Jv)= \ell(v) - \ell(w_J)$ for all $v \in W_J$ \cite[\S2.3]{CCG}.
Define 
\be\label{j-vee-eq}
J^\vee := w_0 J w_0 \subset S\ee
and observe that $v \mapsto v \cdot w_J \cdot w_0$ is a bijection $W^J \to W^{J^\vee}$ \cite[Cor. 2.3.3]{CCG}.
Given $v \in W$, let $\Ad(v)$ denote the inner automorphism 
$w \mapsto vwv^{-1}$.
Now, for $\tau=(w,z) \in \cK^J$, let 
\be\label{tau-vee-eq}
\tau^\vee := ( w \cdot w_J \cdot w_0, z^\vee)\in \cK^{J^\vee} = W^{J^\vee} \times \sI_{J^\vee}
\ee
where if $z = (y,\theta)$ for $y \in W_J$ and $\theta \in \Aut(W_J,J)$ then 
we set 
\be\label{wz-vee-eq}
z^\vee := (y^\vee,\theta^\vee)\text{ where } \begin{cases} y^\vee := w_0\cdot w_J \cdot  y\cdot  w_0,  \\ 
\theta^\vee := \Ad(w_0)\circ \Ad(w_J)\circ  \theta \circ \Ad(w_0).
\end{cases}
\ee
Observe that $\theta^\vee$ is an involution of $W_{J^\vee}$ preserving $J^\vee$ and that
\[ y^\vee\cdot  \theta^\vee(y^\vee) =   
(w_0\cdot w_J \cdot  y\cdot  w_0)
\cdot 
(w_0 \cdot \theta(y) \cdot w_J \cdot w_0) = 1,
\]
so $\tau^\vee \in \cK^{J^\vee}$. 
The map
$\tau \mapsto \tau^\vee$ is a bijection $\cK^J \to \cK^{J^\vee}$.
This operation depends on $J$ although this is suppressed in our notation.
Recall the partial order $\prec$ from Definition~\ref{prec-def}.

\begin{lemma}\label{vee-lem}
Let $\upsilon,\tau \in \cK^J$ and $w\in W$. Then $(w\tau)^\vee = w\cdot  \tau^\vee$.
It also holds that
$\upsilon \prec \tau$ if and only if $\tau^\vee \prec \upsilon^\vee$.
\end{lemma}

\begin{proof}
Suppose $\upsilon = (v,y)$ and $\tau = (w,z)$. 
For the first claim it suffices show that $(s\tau)^\vee = s \tau^\vee$ when $s\in S$.
The basic properties of minimal length coset representatives (see \cite[Chapter 2]{CCG})
imply that $sw \in W^J$ if and only if $sww_Jw_0 \in W^{J^\vee}$ and that
\[ v \leq w\quad  \Leftrightarrow\quad v w_J \leq ww_J\quad \Leftrightarrow\quad ww_Jw_0 \leq vw_Jw_0.\]
It is straightforward from these observations to check that $(s\tau)^\vee = s \tau^\vee$ as needed.
Likewise, it holds that $\ell(z^\vee) = \ell(w_J) - \ell(z)$,
from which it follows that $z \mapsto z^\vee$ maps $W_J$-conjugacy classes in $\sI_J$ to $W_{J^\vee}$-conjugacy classes
in $\sI_{J^\vee}$ while reversing quasiparabolic Bruhat order.
Thus
\be\label{h-tau-eq}
\h(\tau^\vee)= \ell(w_0) - \ell(w_J) + \tfrac{1}{2}(\ell(z_{\max}) - \ell(z_{\min})) - \h(\tau)
\ee
where $z_{\min}$ and $z_{\max}$ are the unique $W_J$-minimal and $W_J$-maximal elements of the $W_J$-conjugacy class
of $z$. Noting Proposition~\ref{same-orbit-prop}, we conclude that $\upsilon$ and $\tau$ belong to the same $W$-orbit and have
$\h(\upsilon) < \h(\tau)$ if and only if $\upsilon^\vee$ and $\tau^\vee$ belong to the same $W$-orbit and have $\h(\tau^\vee) < \h(\upsilon^\vee)$.  When these conditions hold, moreover, it holds that $\upsilon \prec \tau$ if and only if $\tau^\vee \prec \upsilon^\vee$.
\end{proof}

Define $\sigma^\vee : W_{J^\vee} \to \{\pm1\}$ to be the linear character 
$\sigma^\vee := \sigma \circ \Ad(w_0)$.
If $J=S$ then  $\sigma = \sigma^\vee$ since characters are class functions,
but if $J \neq S$ then we could have $\sigma \neq \sigma^\vee$.

\begin{corollary}\label{des-vee-cor}
Let $s \in S$. Then $sw \in W^J$ if and only if $sw^\vee \in W^{J^\vee}$. Moreover:
\ben
\item[(a)] $ \Asc^<(\tau; \sigma) = \Des^<(\tau^\vee; \sigma^\vee)$ and $ \Des^<(\tau; \sigma) = \Asc^<(\tau^\vee; \sigma^\vee)$.
\item[(b)] $ \Asc^=(\tau; \sigma) = \Asc^=(\tau^\vee; \sigma^\vee)$ and $\Des^=(\tau; \sigma) = \Des^=(\tau^\vee; \sigma^\vee)$.
\een
Consequently $\Asc^{\m|\sigma}(\tau) =\Des^{\n|\sigma^\vee}(\tau^\vee)$
and
$\Des^{\m|\sigma}(\tau) =\Asc^{\n|\sigma^\vee}(\tau^\vee)$.
\end{corollary}

\begin{proof}
Since $ \Asc^<(\tau;\sigma) = \{ s \in S : \tau \prec s\tau\}$ and 
$ \Des^<(\tau;\sigma) = \{ s \in S : s\tau \prec \tau\}$, part (a) is  immediate from Lemma~\ref{vee-lem}.
Part (b) is easy to check using Proposition~\ref{h-eq-prop}.
\end{proof}

We may now prove the main results of this section.
\begin{theorem}\label{inversion-thm}
If $\upsilon,\tau \in \cK^J$ then $\sum_{\kappa \in \cK^J} (-1)^{\h(\upsilon) + \h(\kappa)} \cdot \m^\sigma_{\upsilon\kappa} \cdot \n^{\sigma^\vee}_{\tau^\vee \kappa^\vee}= \delta_{\upsilon\tau}.$
\end{theorem}

\begin{proof}
Let $\cL^\sigma$ be the set of $\LL$-linear maps $\cM^\sigma \to \LL$. For $\tau \in \cK^J$ let $L^\sigma_\tau \in \cL^\sigma$
be the element with $L^\sigma_\tau(M^\sigma_\upsilon)= \delta_{\upsilon\tau}$ for all $\upsilon \in \cK^J$.
Then $\cL^\sigma$ has a unique $\H$-algebra structure in which $H_s L : M \mapsto L(H_sM)$
for $s\in S$, $L \in \cL^\sigma$, and $M \in \cM^\sigma$. In fact, the linear map $\cM^\sigma \to \cL^\sigma$ with $M^\sigma_\tau \mapsto L^\sigma_\tau$ is an isomorphism of $\H$-modules.

Define $\beta : \cM^{\sigma^\vee} \to \cL^\sigma$ to be the linear map with $\beta : \overline{M^{\sigma^\vee}_{\tau^\vee}} \mapsto L^\sigma_\tau$ for $\tau \in \cK^J$.
This is a linear bijection which by Corollary~\ref{des-vee-cor} is also an $\cH$-module isomorphism.
Write $L \mapsto \overline{L}$ for the antilinear map $\cL^\sigma \to \cL^\sigma$
with $\overline L : M\mapsto \overline{L(\overline M)}$.
One easily checks that $\overline{HL} =\overline H \cdot \overline L$ for all $H\in \H$ and $L \in \cL^\sigma$.
It is evident from \eqref{ut-eq} 
and the fact that $\overline{M_{s\tau}} = \overline{H_s} \overline{M_\tau} =  (H_s +x^{-1}-x)  \overline{M_\tau}$
that $\overline{L^\sigma_\tau} = L^\sigma_\tau$ if and only if $\Asc^<(\tau)=\varnothing$.
In view of Corollary~\ref{des-vee-cor} and the uniqueness assertion in Theorem~\ref{m-thm}(b),
 the map $\cM^{\sigma^\vee} \to \cM^{\sigma^\vee}$
given by $M \mapsto \beta^{-1}(\overline{\beta(M)})$ must coincide with the usual bar operator,
so $\beta(\overline{M}) = \overline{\beta(M)}$
for all $M \in \cM^{\sigma^\vee}$.

Now fix $\upsilon,\tau\in \cK^J$.
Since $\underline {\hat M}^{\sigma^\vee}_{\upsilon^\vee}$ is bar invariant
and since $\h(\upsilon^\vee) - \h(\kappa^\vee) = \h(\kappa) - \h(\upsilon)$ if 
$\kappa^\vee$ and $\upsilon^\vee$ are in the same $W$-orbit
(which is necessary for $\n^{\sigma^\vee}_{\kappa^\vee\upsilon^\vee}\neq 0$)
by \eqref{h-tau-eq}, we have 
\be\label{last-express} \beta\(\underline {\hat M}^{\sigma^\vee}_{\upsilon^\vee}\)(\underline M^\sigma_\tau) = \sum_{\kappa \in \cK^J} (-1)^{\h(\upsilon) + \h(\kappa)} \cdot \n^{\sigma^\vee}_{\kappa^\vee\upsilon^\vee} \cdot \m^{\sigma}_{\kappa\tau}.\ee
The last expression is bar invariant since if $\overline{M_1} = M_1$ and $\overline{M_2} = M_2$ then 
\[\overline{\beta(M_1)(M_2)} = \overline{\beta(M_1)}(\overline {M_2})  =  \beta(\overline {M_1})(\overline {M_2})
=  \beta( {M_1})( {M_2}).\]
On the other hand, the right side of \eqref{last-express}
 belongs to $\delta_{\upsilon\tau} + x^{-1}\ZZ[x^{-1}]$ since
 $\m^\sigma_{\kappa\tau}$ is $1$ when $\kappa = \tau$ and in $x^{-1}\ZZ[x^{-1}]$ otherwise,
 and likewise for $ \n^{\sigma^\vee}_{\kappa^\vee\upsilon^\vee}$.
For these observations to hold simultaneously, the right side of \eqref{last-express}
must in fact be equal to $\delta_{\upsilon\tau}$.

Thus, if $A$ is the $\cK^J\times \cK^J$ matrix with entries $A_{\upsilon\kappa} := (-1)^{\h(\upsilon)+\h(\kappa)} \cdot \n^{\sigma^\vee}_{\kappa^\vee\upsilon^\vee}$
and $B$ is the $\cK^J\times \cK^J$ matrix with entries $B_{\kappa\tau} := \m^\sigma_{\kappa\tau}$, then $AB$ is the $\cK^J\times \cK^J$ identity matrix, so $A^{-1}=B$ and $BA$ is also the identity matrix, which is equivalent 
to the desired identity.
\end{proof}

Recall that two $W$-graphs $\Gamma = (V,\omega,I)$ and
$ \Gamma' = (V',  \omega',  I')$ are \emph{dual} via some
 bijection $V \to V'$, written $v\mapsto v'$,
 if $\omega'(u',v') = \overline{\omega(v,u)}$ and $ I'(v') = S \setminus I(v)$
for all $u,v \in V$.
 
\begin{theorem}\label{dual-thm}
The $W$-graphs $\Upsilon^{\m|\sigma}$ and $\Upsilon^{\n|\sigma^\vee}$
(respectively, $\Upsilon^{\n|\sigma}$ and $\Upsilon^{\m|\sigma^\vee}$)
are dual via the map $\tau \mapsto \tau^\vee$.\end{theorem}

\begin{proof}
Suppose $\upsilon,\tau \in \cK^J$. 
If these elements are not in the same $W$-orbit then neither are $\upsilon^\vee$ and $\tau^\vee$ by Corollary~\ref{vee-lem}, so $\mu^{\m|\sigma}_{\upsilon\tau} = \mu^{\n|\sigma^\vee}_{\tau^\vee \upsilon^\vee} = 0$ by 
Corollary~\ref{e-cor}.
Assume $\upsilon$ and $\tau$ belong to the same $W$-orbit.
If $\h(\tau) -\h(\upsilon)$ is even then so is $\h(\upsilon^\vee) - \h(\tau^\vee)$ by \eqref{h-tau-eq},
so $\mu^{\m|\sigma}_{\upsilon\tau} = \mu^{\n|\sigma^\vee}_{\tau^\vee \upsilon^\vee} = 0$ by 
Corollary~\ref{e-cor}.
Further assume $\h(\tau) -\h(\upsilon)$ is odd.
Because $\m^\sigma_{\upsilon\kappa}$ is either $1$ when $\upsilon = \kappa$ or in $x^{-1}\ZZ[x^{-1}]$ when $\upsilon \neq \kappa$ (and similarly for $\n^{\sigma^\vee}_{\tau^\vee \kappa^\vee}$),
taking the coefficient of $x^{-1}$ on both sides of the identity in Theorem~\ref{inversion-thm} gives 
$\mu^{\n|\sigma^\vee}_{\tau^\vee \upsilon^\vee} -\mu^{\m|\sigma}_{\upsilon\tau}  = \mu^{\n|\sigma^\vee}_{\tau^\vee \upsilon^\vee} + (-1)^{\h(\upsilon) + \h(\tau)} \mu^{\m|\sigma}_{\upsilon\tau} = 0$.
Thus we have $\mu^{\m|\sigma}_{\upsilon\tau} = \mu^{\n|\sigma^\vee}_{\tau^\vee \upsilon^\vee}$ in all cases.
Since $\omega^{\m|\sigma}$ and $\omega^{\n|\sigma^\vee}$ take integer values, it is immediate 
from Corollary~\ref{des-vee-cor} that $\Upsilon^{\m|\sigma}$ and $\Upsilon^{\n|\sigma^\vee}$ are dual via the map $\kappa\mapsto \kappa^\vee$.
Repeating this argument with $\sigma$ replaced by $\sgn\cdot \sigma$ shows that $\Upsilon^{\n|\sigma}$ and $\Upsilon^{\m|\sigma^\vee}$ are dual via the same map.
\end{proof}

\subsection{Gelfand $W$-graphs}\label{gel-sect}

Let $(W,S)$ be a finite Coxeter system with Iwahori-Hecke algebra $\H$.
Suppose $\TT = (J,z_{\min},\sigma)$ is a model triple for $W$
as in Definition~\ref{model-triple-def}, so that 
$z_{\min} \in \sI_J$ is $W_J$-minimal.
Define \be
\cK^\TT := W^J \times \left\{ v \cdot z_{\min} \cdot v^{-1} : v \in W_J \right\} \subset \cK^J
\ee
and then let 
$
\cM^\TT := \LL\spanning\left\{ M^\sigma_\tau : \tau \in \cK^\TT\right\}
$
and
$\cN^\TT := \LL\spanning\left\{ N^\sigma_\tau : \tau \in \cK^\TT\right\}.
$
The sets $\cM^\TT$ and $\cN^\TT$ are $\H$-submodules of $\cM^\sigma$
and  $\cN^\sigma$, respectively.
If $\sP$ is a set of model triples for $W$, then we define two associated  $\H$-modules as the direct sums
\be\cM^\sP := \bigoplus_{\TT \in \sP} \cM^\TT \quand \cN^\sP := \bigoplus_{\TT \in \sP} \cN^\TT.\ee
For any field $\KK$, let 
$(\cM^\sP)_\KK := \KK(x) \otimes_{\ZZ[x,x^{-1}]} \cM^\sP$
and
$(\cN^\sP)_\KK := \KK(x) \otimes_{\ZZ[x,x^{-1}]} \cN^\sP$.

\begin{theorem}\label{abs-main-thm}
Suppose $W$ is a finite Coxeter group with splitting field $\KK \subset \CC$.
If $\sP$ is a perfect model for $W$, then the $\H_\KK(W)$-modules 
$(\cM^\sP)_\KK$ and $(\cN^\sP)_\KK$
are Gelfand models.
\end{theorem}

If $W$ is a finite Weyl group then we can take $\KK=\QQ$ by \cite[Thm. 6.3.8]{GeckPfeiffer}.

\begin{proof}
When $x=1$, the $\H$-module $\cM^\TT$ for $\TT = (J,z_{\min},\sigma)$
becomes a $W$-module whose
character is $\Ind_{C_{W_J}(z_{\min})}^W \Res_{C_{W_J}(z_{\min})}^{W_J}(\sigma)$
and $\cM^\sP$ becomes a module containing each irreducible $W$-representation exactly once.
If $x=1$ then the same is true for $\cN^\sP$ since its character is the character
of $\cM^\sP$ multiplied by $\sgn$. 
This implies that 
$(\cM^\sP)_\KK$ are $(\cN^\sP)_\KK$ are Gelfand models for $\H_\KK$ 
since the latter algebra is isomorphic to the group algebra $\KK(x) W$ \cite[Thm. 8.1.7]{GeckPfeiffer}.
\end{proof}

If $\upsilon,\tau \in \cK^J$ then
$\upsilon \preceq \tau$ can only hold if $\{\upsilon,\tau \}\subset \cK^\TT$ for some model triple $\TT$.
It follows that each set $\cK^\TT$ is a union of weakly connected components in  
 $\Upsilon^{\m|\sigma}$ and $\Upsilon^{\n|\sigma}$. 
Define $\Upsilon^{\m}(\TT)$ and $\Upsilon^{\n}(\TT)$ to be the $W$-graphs formed by
restricting $\Upsilon^{\m|\sigma}$ and $\Upsilon^{\n|\sigma}$ to the vertex set $\cK^\TT$.
It is clear from Theorems~\ref{mw-thm} and \ref{nw-thm}
that $\cY(\Upsilon^{\m}(\TT)) \cong \cM^\TT$ and $\cY(\Upsilon^{\n}(\TT)) \cong \cN^\TT$.

There is an obvious notion of \emph{disjoint union} for $W$-graphs which preserves (quasi-)admissibility.
If $\sP$ is any set of model triples for $W$, then we define two quasi-admissible  $W$-graphs by \be
 \Upsilon^{\m}(\sP) := \bigsqcup_{\TT \in \sP} \Upsilon^{\m}(\TT)\quand\Upsilon^{\n}(\sP) := \bigsqcup_{\TT \in \sP} \Upsilon^{\n}(\TT).
 \ee
The following is clear from Theorem~\ref{abs-main-thm}
and our definition of a Gelfand $W$-graph.

\begin{corollary}\label{gelfand-w-thm}
If $\sP$ is a perfect model, then $\Upsilon^{\m}(\sP)$
and $\Upsilon^{\n}(\sP) $ are Gelfand $W$-graphs.
\end{corollary}

Our next goal is to explain how to derive Theorems~\ref{main-thm2}, \ref{main-thm3}, and \ref{main-thm4}
in the introduction from the preceding results applied to  the perfect models $\sP(W)$ in Definition~\ref{main-thm1-def}.
Concretely, this amounts to describing for each classical Weyl group $W$ a bijection between the sets 
$\cG(W)$ and $\bigsqcup_{\TT \in \sP(W)} \cK^\TT$ that is descent-preserving in an appropriate sense.

For the rest of this section, we fix $W \in \{S_n, \W_n, \WD_n\}$ and write $S$ for its set of simple generators.
Recall the definition of $\cG=\cG(W)$ from Section~\ref{descent-sect}.
We again write $\cG^\A_{n-1}\subset \cG^\BC_n\supset \cG^\D_n$ 
for   $\cG$ when $W$ is $S_n$, $\W_n$, or $\WD_n$, respectively.
Choose $v \in \cG$
and suppose 
\[\{ i \in [n] : |v(i)| \in [n]\} = \{a_1 < a_2 < \dots < a_m\}\] for some numbers  $a_i \in [n]$. Note that $m$ must be even.
Let 
\[\psi :\{-m,\dots,-1,0,1,\dots,m\} \to \{-a_m < \dots < -a_1 < 0 < a_1 < \dots < a_m\}\] 
be the unique order-preserving bijection. Write $[\pm n] := \{ i \in \ZZ :  |i| \in [n]\}$
and suppose \[\{ i \in [\pm n]: v(i) > n\} = \{ b_1 < b_2 <\dots < b_{n-m}\}\]
for some $b_i \in [\pm n]$. 
We next define 
$J \subset S,$ $ \sigma : W_J \to \{\pm 1\},$ $ w \in W^J,$ and $ z\in \sI_J$
as follows:
\ben
\item[($\A$)] Assume $W=S_n$  so that $S = \{s_1,s_2,\dots,s_{n-1}\}$
and each $b_i \in [n]$. Let $J = S \setminus\{s_m\}$
so that $W_J = S_m \times S_{n-m}$ and define $\sigma = \one_{S_m} \times \sgn_{S_{n-m}}$.
Let \be\label{wab-eq}w= a_1a_2\cdots a_m b_1b_2\cdots b_{n-m} .\ee
As $u \in S_n=W$ belongs to $W^J$ if and only if $u(1) < \dots < u(m)$
and $u(m+1) < \dots < u(n)$, we have $w \in W^J$.
Finally 
define $z\in \sI_J$
to have $z(i) = \psi^{-1} \circ v \circ \psi(i)$ for 
 $i \in [m]$ and $z(i) = i$ for $i \in [n]\setminus [m]$.
For example, if $n=6$ and \[v = (1,4)(2,7)(3,6)(5,8)(9,10)(11,12) \in \cG^\A_5 \subset S_{12}\] then
$\{a_1<a_2<a_3<a_4\} = \{1<3<4<6\}$ and $\{b_1<b_2\} = \{2 < 5\}$ so
\[ J = \{s_1,s_2,s_3,s_5\},\quad w = 134625 \in S_6,\quand z = (1,3)(2,4) \in \langle J\rangle\subset S_6.\]

\item[($\BC$)] Assume $W=\W_n$ so that $S = \{s_0,s_1,\dots,s_{n-1}\}$.
Let $J =S \setminus\{s_m\}$
so that $W_J = \W_m \times S_{n-m}$ and define $\sigma = \one_{\W_m} \times \sgn_{S_{n-m}}$.
A permutation $u \in \W_n=W$ belongs to $W^J$ if and only if $0<u(1) < \dots < u(m)$
and $u(m+1) < \dots < u(n)$.
We can therefore define $w \in W^J$ and $z \in \sI_J$ in the same way as for type $\A$. 
For example, if $n=4$ and \[v = (1,-4)(-1,4)(2,6)(-2,-6)(3,-5)(-3,5)(7,8)(-7,-8) \in \cG^\BC_4 \subset \W_{8}\] then
$\{a_1<a_2\} = \{1<4\}$ and $\{b_1<b_2\} = \{-3<2\}$ so
\[ J = \{s_0,s_1,s_3\},\quad w = 14\bar 3 2 \in \W_4,\quand z = (1,-2)(-1,2) \in \langle J\rangle\subset \W_4.\]

\item[($\D$)] Assume $W=\WD_n$  so that $S = \{s_{-1},s_1,\dots,s_{n-1}\}$.
If $m=0$ then let $J =S\setminus\{s_{-1}\}$ and otherwise let 
$J =S \setminus \{s_m\}$.
Then $W_J = \WD_m \times S_{n-m}$ and we define 
$\sigma = \one_{\WD_m} \times \sgn_{S_{n-m}}$.
A permutation $u \in \WD_n=W$ belongs to $W^J$ if and only if $|u(1)| < u(2) < \dots < u(m)$
and $u(m+1) < \dots < u(n)$.
If $|\{ i \in [n-m] : b_i < 0\}|$ is even then we define
$w\in W^J$ and $z \in \sI_J$ in the same way as for types $\A$ and $\BC$.
If this number is odd then we set 
 $w :=  \overline{ a_1} a_2\cdots a_m b_1b_2\cdots b_{n-m} \in W^J$
and define $z  \in \sI_J$
to have $s_0zs_0(i) = \psi^{-1} \circ v \circ \psi(i)$ for 
 $i \in [m]$ and $z(i) = i$ for $i \in [n]\setminus [m]$.
For example, if $n=4$ and 
 \[v = (1,-4)(-1,4)(2,6)(-2,-6)(3,-5)(-3,5)(7,8)(-7,-8) \in \cG^\D_4 \subset \WD_{8}\] then $m=2$ and
exactly one $b_i$ is negative, so
\[ J = \{s_{-1},s_1,s_3\},\quad w = \bar14\bar 3 2 \in \WD_4,\quand z = (1,2)(-1,-2) \in \langle J\rangle\subset \WD_4.\]
\een
Putting everything together, in each case we write 
\be \vartheta_W(v) := (J,\sigma)\quand \varphi_W(v) := (w,z) \in \cK^J.\ee
The ascent and descent sets in the following lemma are as in
 \eqref{intro-des-1}, \eqref{intro-des-2}, and \eqref{asc-des-def}.

\begin{lemma}\label{vartheta-lem}
Fix $W \in \{S_n, \W_n, \WD_n\}$ and $v \in \cG$. Let $(J,\sigma) = \vartheta_W(v)$ and $\tau  =\varphi_W(v)$.
Then $\Des^=(v) = \Des^=(\tau;\sigma)$, $\Asc^=(v) = \Asc^=(\tau;\sigma)$,
$\Des^<(v) = \Des^<(\tau;\sigma)$, and $\Asc^<(v) = \Asc^<(\tau;\sigma)$.
Furthermore, if $s \in \Des^<(v)  \sqcup \Asc^<(v) $ then $\varphi_W(svs) = s\tau$.
\end{lemma}

\begin{proof}
Let $I_m$ for $0\leq m \leq n$ be the set of elements $y=y^{-1} \in W $ with $|y(i)| \neq i$ for $i \in [m]$ and $y(i) = i$ for $i>m$.
If $m$ is not even then $I_m$ is empty.
Given $y \in I_m$, define $\iota(y) \in \W_{2n}$ to be the element that maps $i \mapsto y(i)$ for $i\in [m]$,
$m + i \mapsto n + i$ for $i \in [n-m]$, and $2n-m + i \mapsto 2n-m+i+1$ for each odd $i \in [m]$.
 Also let $J_m = S \setminus \{s_m\}$, except  when $m=0$ and $W=\WD_n$ in which case we set $J_0 = S\setminus\{s_{-1}\} =\{s_1,s_2,\dots,s_{n-1}\}$.

The following claims may be checked using Propositions~\ref{old-des-prop} and \ref{new-des-prop1}.
Suppose $ (w,z) = \tau =\varphi_W(v)$ and let $m = |\{ i \in [n] : |v(i)| \leq n\}|$. Then $m$ is even and $J=J_m$, 
and    $w $ and $z$ are the unique elements in $W^{J_m}$ and $I_m$, respectively, such that $w\cdot \iota (z) \cdot w^{-1} = v$.
 Moreover,  \be\label{vzw-eq}
 \ell(v) = \ell(z) + 2\ell(w) + \kappa_m\quad\text{where $\kappa_m := \ell(\iota(s_1s_3\cdots s_{m-1})) - \tfrac{m}{2}$.}\ee
 From these properties, it suffices to show that  $\Des^=(v) = \Des^=(\tau;\sigma)$
 and
  $\Asc^=(v) = \Asc^=(\tau;\sigma)$.
If this holds and $s \in \Des^<(v) \sqcup \Asc^<(v)= \Des^<(\tau;\sigma) \sqcup \Asc^<(\tau;\sigma)$,
  then either $sw \in W^{J_m}$ and $svs = (sw) \iota(z) (sw)^{-1}$,
  in which case $\varphi_W(sws) = (sw,z) = s\tau$
  and 
  \be\label{svs-1eq}\ell(svs) - \ell(v) = 2 (\ell(sw) - \ell(w)) = 2 (\h(s\tau) - \h(\tau)),\ee
  or $sw\notin W^{J_m}$ but $t:= w^{-1}sw \in J_m$.
  In the latter case we must have $t\notin \{s_{m+1},s_{m+2},\dots,s_{n-1}\}$ since $s \notin \Asc^=(\tau;\sigma)$,
  so $tzt \in I_m$ and $svs = wt\cdot \iota(z) \cdot tw^{-1} = w\cdot \iota(tzt)\cdot w^{-1}$,
  which implies that $\varphi_W(sws) = (w,tzt) = s\tau$ and 
   \be\label{svs-2eq}
   \ell(svs) - \ell(v) =\ell(tzt) - \ell(z) = 2 (\h(s\tau) - \h(\tau)).\ee
In both cases we have $\varphi_W(sws)  = s\tau$,
and the equations \eqref{svs-1eq} and \eqref{svs-2eq}
show that  $\Des^<(v) = \Des^<(\tau;\sigma)$ and $\Asc^<(v)= \Asc^<(\tau;\sigma)$.

Showing that $\Des^=(v) = \Des^=(\tau;\sigma)$
 and
  $\Asc^=(v) = \Asc^=(\tau;\sigma)$ reduces to a simpler claim.
Continue to write $\tau = (w,z)$. Fix  $s \in S$ and let $t = w^{-1}sw$.
Note that $s$ belongs to $ \Des^=(\tau;\sigma)$ 
if and only if $tzt=z$ and $t \in J_m \setminus\{s_{m+1},s_{m+2},\dots,s_{n-1}\}$,
while $s$ belongs to $\Asc^=(\tau;\sigma)$
if and only if $tzt=z$ and $t \in \{s_{m+1},s_{m+2},\dots,s_{n-1}\}$.
Since we have $\iota(tzt) = t\cdot \iota(z)\cdot t$ whenever $t \in J_m \setminus\{s_{m+1},s_{m+2},\dots,s_{n-1}\}$,
it follows that if $s \in \Des^=(\tau;\sigma)$ then 
\[v = w\cdot\iota(z) \cdot w^{-1} = wt\cdot\iota(z) \cdot tw^{-1} = svs\] so $s \in \Des^=(v)$.
Similarly, as $\iota(z) \cdot t\cdot  \iota(z) \in \{ s_i : i>n\}$ whenever $t \in \{s_{m+1},s_{m+2},\dots,s_{n-1}\}$,
it follows that if $s \in \Asc^=(\tau;\sigma)$ then
$vsv = w\cdot\iota(z)\cdot  t \cdot \iota(z) \cdot w^{-1} = \iota(z)\cdot t \cdot\iota(z) \in  \{ s_i : i>n\}$
so $s \in \Asc^=(v)$.
Thus
$\Des^=(v) \supset \Des^=(\tau;\sigma)$
 and
  $\Asc^=(v) \supset \Asc^=(\tau;\sigma)$.

It remains only to show the reverse containments.
We check this in detail for type $\D$ using Proposition~\ref{new-des-prop2}; the arguments for types $\A$ and $\BC$ are similar.
Assume $W=\WD_n$.
Define $a_1,a_2,\dots,a_m$ and $b_1,b_2,\dots,b_{n-m}$ relative to $v$ as before the lemma
and let $e=|\{ i \in [n-m] : b_i < 0\}|$.
 If $s_{\pm 1} \in \Des^=(v)$, then we have $v(1) \in \{\pm 2\}$
so $\{1,2\} \subset \{a_1,a_2,\dots,a_m\}$
and $z(1) \in \{\pm 2\}$. In this case $w^{-1} s_{\pm 1} w$ is  $ s_{\pm 1}$ if $e$ is even or $s_{\mp 1}$ if $e$ is odd; as this commutes with $z$ either way, we have $s \in \Des^=(\tau;\sigma)$.
If $s_i \in \Des^=(v)$ for $i>1$, then $v(i) \in \{\pm (i+1)\}$
so $\{i,i+1\} \subset \{a_j,a_{j+1}\}$ for some $j \in [m-1]$. But then   $z(j) \in \{\pm (j+1)\}$ and $w^{-1} s_i w = s_j$,
so  $s_j zs_j = z$ and again $s_i \in \Des^=(\tau;\sigma)$.

If instead  $s_{- 1} \in \Asc^=(v)$, then either $v(1) <-n<n<v(2)$ or $v(2) < -n < n < v(1)$.
In the first case we have $b_j=-1$ and $b_{j+1} = 2$ for some $j \in [n-m]$
and in the second case we have $b_j =-2$ and $b_{j+1} = 1$ for some $j \in [n-m]$,
and in either situation $w^{-1} s_{- 1}w = s_{m+j}$ so $s_{-1} \in \Asc^=(\tau;\sigma)$.
Similarly, if $s_i \in\Asc^=(v)$ for some $i \in [n-1]$ then 
either $n<v(1)< v(2)$ or $v(1)<v(2) < -n $, so  for some $j \in [n-m]$
we have $b_j=1<b_{j+1}=2$ or $b_j=-2<b_{j+1}=-1$, and therefore $w^{-1}s_i w=s_{m+j}$
whence $s_i \in \Asc^=(\tau;\sigma)$.
This completes our argument that
$\Des^=(v) = \Des^=(\tau;\sigma)$
 and
  $\Asc^=(v) = \Asc^=(\tau;\sigma)$
   when $W= \WD_n$.
\end{proof}

Define the model $\sP=\sP(W)$ as in Definition~\ref{main-thm1-def}.
For each $\TT = (J,z_{\min},\sigma) \in \sP$, let  $\cG|_\TT := \{ v \in \cG : \vartheta_W(v) = (J,\sigma)\}.$
Observe that  $u,v \in \cG$ both belong $\cG|_\TT$
for some $\TT \in \sP$ if and only if 
$\{ i \in [n] : |u(i)| \in [n]\}$ and $\{ i \in [n] : |v(i)| \in [n]\}$ have the same cardinality.

Define $\cM = \LL\spanning\{ M_z : z \in \cG\}$
and $\cN = \LL\spanning\{ N_z : z \in \cG\}$ as in Theorem~\ref{main-thm2}.
Let $\cM|_\TT$ and $\cN|_\TT$
denote the respective $\LL$-submodules of $\cM$ and $\cN$
spanned by the basis elements $M_v$ and $N_v$ with $v \in \cG|_\TT$.

\begin{lemma}\label{varphi-lem}
If $\TT \in \sP$ then $\varphi_W : \cG|_\TT \to \cK^\TT$ is a bijection, and  $\cG = \bigsqcup_{\TT \in \sP} \cG|_\TT$.
\end{lemma}

\begin{proof}
Suppose  $v \in \cG$ and $(J,\sigma) = \vartheta_W(v)$. 
Let $(w,z) = \varphi_W(v)\in W^J \times \sI_J$. 
We can recover $v$ from $(w,z)$ as $v = w \cdot\iota(z) \cdot w^{-1}$
where $\iota$ is defined as in the proof of Lemma~\ref{vartheta-lem},
so 
 $\varphi_W : \cG|_\TT \to \cK $ is injective.
 It is easy to see that if
 $\TT = (J,z_{\min},\sigma) \in \sP$ then
 the image of $\cG|_\TT$ under $\varphi_W$ is a subset of $\cK^\TT$
  containing $(1,z_{\min})$. 
Since  $\varphi_W(svs) = s\cdot \varphi_W(v)$ for $s \in \Des^<(v)  \sqcup \Asc^<(v) $
by Lemma~\ref{vartheta-lem},
the map
$\varphi_W : \cG|_\TT \to \cK^\TT$ is also surjective.
Finally, each $v \in \cG$ 
with $2k = |\{ i \in [n] : |v(i)| \in [n]\}|$ belongs to $\cG|_\TT$ for the model triple
$\TT$ indexed by $k$ in  Definition~\ref{main-thm1-def}. 
\end{proof}

\begin{proof}[Proof of Theorem~\ref{main-thm2}]
Lemma~\ref{varphi-lem}
shows that the linear maps
\be\label{fg-eq}
f_W : \cM \to \cM^\sP
\quand 
g_W:\cN \to \cN^\sP
\ee
with 
$f_W:
M_v \mapsto M_{\varphi_W(v)}^\sigma
$
and
$g_W:
N_v \mapsto N_{\varphi_W(v)}^\sigma$
for each $\TT=(J,z_{\min},\sigma) \in \sP$ and $v \in \cG|_\TT$
are bijections. By Lemma~\ref{vartheta-lem}, the desired
$\H$-module structures on $\cM$ and $\cN$
are precisely the ones given by transferring the $\H$-module structures on $\cM^\sP$ and $\cN^\sP$ via these maps.
We conclude from Theorem~\ref{abs-main-thm} that $\cM_\QQ \cong (\cM^\sP)_\QQ$ and $\cN_\QQ \cong (\cN^\sP)_\QQ$ 
are Gelfand models for $\H_\QQ$ when $\sP$ is a 
perfect model, which happens if $W$ is $S_n$, $\W_n$, or $\WD_{2n+1}$ for some $n$ by Theorem~\ref{main-thm1}.
\end{proof}

\begin{proof}[Proof of Theorem~\ref{main-thm3}]
The maps $f_W$ and $g_W$ from \eqref{fg-eq} are $\H$-module isomorphisms,
so setting  $\overline{M} = f_W^{-1}(\overline{f_W(M)})$ for $M \in \cM$
and $\overline{N} = g_W^{-1}(\overline{g_W(N)})$ for $N \in \cN$
gives the desired (unique) $\H$-compatible bar operators in view of Theorem~\ref{m-thm}  and Lemma~\ref{vartheta-lem}.
It follows from \eqref{vzw-eq} that if $\varphi_W(u) \preceq \varphi_W(v)$ for $u,v \in \cG$
then $\ell(u) \leq \ell(v)$, so Theorem~\ref{m-thm}
also implies that the bar invariant 
elements defined by $\underline M_z := f_W^{-1}(\underline M_{\varphi_W(z)})$
and
$\underline N_z := g_W^{-1}(\underline N_{\varphi_W(z)})$ for $z \in \cG$
have the triangular expansion specified in the statement of Theorem~\ref{main-thm3}. 
It remains only to show the uniqueness of these elements.
This follows from Lemma~\ref{webster-lem}, since
\eqref{ut-eq} in Theorem~\ref{m-thm} 
shows that $(M\mapsto \overline{M}, \{M_z :z \in \cG\})$, with $\cG$ ordered such that $u<v$ if $\ell(u)<\ell(v)$,
is a pre-canonical structure on $\cM$, and $\{ \underline M_z : z \in \cG\}$ is evidently
the unique canonical basis afforded by this structure.
A similar argument shows that the elements $\underline{N}_z$ are likewise unique.
\end{proof}

Recall the definitions of   $\Gamma^\m=\Gamma^\m(W)$ and $\Gamma^\n=\Gamma^\n(W)$
from Theorem~\ref{main-thm4}.

\begin{proof}[Proof of Theorem~\ref{main-thm4}]
In view of Lemma~\ref{vartheta-lem} and the proof of Theorem~\ref{main-thm3}, 
 $\Gamma^\m$ and $\Gamma^\n$ are the quasi-admissible $W$-graphs
 respectively formed by transferring the $W$-graph structures 
$ \Upsilon^{\m}(\sP)$ and $ \Upsilon^{\n}(\sP)$ to the set $\cG$
via the invertible map $v \mapsto \varphi_W(v)$. 
The claim that the linear maps $Y_z \mapsto \underline M_z$ and $Y_z \mapsto \underline N_z$ 
are isomorphisms $\cY(\Gamma^\m) \cong \cM$ and $\cY(\Gamma^\n)\cong \cN$
follows by comparing 
Theorems~\ref{mw-thm} and \ref{nw-thm} with
our definitions of  $ \underline M_z$ and $\underline N_z$
in the proof of Theorem~\ref{main-thm3}.
\end{proof}

\begin{remark}\label{bar-bar-rmk}
Given $\TT  \in \sP$, 
form $\Gamma^\m|_\TT$ and $\Gamma^\n|_\TT$ 
by restricting $\Gamma^\m$ and $\Gamma^\n$ 
to  $\cG|_\TT$.
Since $\varphi_W$ maps $\cG|_\TT$ bijectively onto $\cK^\TT$,
it follows 
 that
 $
 \Gamma^\m = \bigsqcup_{\TT \in \sP} \Gamma^\m|_\TT
 $ and $
 \Gamma^\n = \bigsqcup_{\TT \in \sP} \Gamma^\n|_\TT.
 $
\end{remark}

\subsection{Model equivalence}\label{equiv-sect}

Again let $(W,S)$ be a finite Coxeter system. If $\sP$ is a perfect model for $W$,
then it is a natural problem of interest to classify the cells in the $W$-graphs $\Upsilon^\m(\sP)$ and $\Upsilon^\n(\sP)$.
In this section we describe a form of equivalence for perfect models 
that ensures that the cells in these $W$-graphs are related in a simple way for different choices of $\sP$.

If $\Gamma = (V,\omega,I)$ and $\Gamma' = (V',\omega',I')$ are $W$-graphs,
then an \emph{isomorphism} $\phi : \Gamma \to \Gamma'$
  is a bijection $\phi : V \to V'$ with
$\omega  = \omega' \circ (\phi \times \phi)$
and
$ I  = I' \circ \phi$.
Such a map sends cells to cells.
Choose a model triple $\TT=(J,z_{\min},\sigma)$ for $W$ and an automorphism $\alpha \in \Aut(W,S)$.
Define 
\be \TT^\alpha := (\alpha(J), \alpha(z_{\min}), \sigma\circ \alpha^{-1}).\ee
It is easy to see that $\TT^\alpha$ is another model triple for $W$.

If  $\Gamma = (V,\omega,I)$ is a $W$-graph,
then let $\Gamma^\alpha := (V, \omega, I^\alpha)$ where $I^\alpha(v) := \{ \alpha(s) : s \in I(v)\}$ for $v \in V$.
As $\cY(\Gamma^\alpha)$
is just the $\H$-module $\cY(\Gamma)$ twisted by the $\H$-automorphism sending $H_w \mapsto H_{\alpha^{-1}(w)}$,
$\Gamma^\alpha$ 
 is also a $W$-graph
with the same directed graph structure as $\Gamma$.

\begin{proposition}\label{alpha-prop}
If $\TT$ is a model triple for $W$ and $\alpha \in \Aut(W,S)$ then
$(w,z) \mapsto (\alpha(w), \alpha(z))$ 
is an isomorphism of $W$-graphs $\Upsilon^{\m}(\TT)^\alpha\xrightarrow{\sim} \Upsilon^{\m}(\TT^\alpha )$ 
and $\Upsilon^{\n}(\TT)^\alpha \xrightarrow{\sim} \Upsilon^{ \n}(\TT^\alpha )$.
\end{proposition}

\begin{proof}
The set of strict/weak ascents/descents for $(w,z) \in \cK^\TT$ relative to $\sigma$
is the same as the corresponding set of strict/weak ascents/descents for $(\alpha(w),\alpha(z)) \in \cK^{\TT^\alpha}$ relative to $\sigma\circ \alpha$, so this proposition is immediate from Corollary~\ref{alpha-cor}.
\end{proof}

Suppose $\TT = (J,z_{\min},\sigma)$ and $\TT' = (J',z_{\min}',\sigma')$
are model triples for $W$ such that 
\[ J=J',\quad \sigma = \sigma',\quand C_{W_J}(z_{\min}) = C_{W_J}(z'_{\min}).\]
To indicate this situation, we write $\TT \equiv\TT'$.
Given an element $z = w\cdot z_{\min} \cdot w^{-1}$ for some $w \in W_J$,
define $z' := w\cdot z'_{\min} \cdot w^{-1}$.
Then $z \mapsto z'$ is a well-defined bijection from the $W_J$-orbit of $z_{\min}$
to the $W_J$-orbit of $z'_{\min}$
and the map 
$(w,z) \mapsto (w,z')$ is an isomorphism of $W$-sets 
$\cK^\TT \xrightarrow{\sim} \cK^{\TT'}$.

\begin{proposition}\label{equiv-prop}
If $\TT \equiv \TT'$ are model triples for $W$ then the bijection 
$\cK^\TT \xrightarrow{\sim} \cK^{\TT'}$ just described
is an isomorphism of $W$-graphs $\Upsilon^{\m}(\TT)  \xrightarrow{\sim} \Upsilon^{\m}(\TT')$
and $\Upsilon^{\n}(\TT)  \xrightarrow{\sim} \Upsilon^{\n}(\TT') $. 
\end{proposition}

\begin{proof}
Assume $\TT=(J,z_{\min},\sigma)\equiv \TT'$ and
let $\phi:\cK^\TT \xrightarrow{\sim} \cK^{\TT'}$
be the isomorphism of $W$-sets given above. 
This bijection preserves the height function \eqref{h-def-eq}
by \cite[Prop. 2.15]{RV}. The
linear map sending $M^\sigma_\tau \mapsto M^\sigma_{\phi(\tau)}$ is therefore an isomorphism of $\H$-modules 
$\cM^\TT \cong \cM^{\TT'}$ that commutes with the relevant bar operators and sends
$\underline M^\sigma_\tau \mapsto \underline M^\sigma_{\phi(\tau)}$ for all $\tau \in \cK^\TT$.
As $\phi$ preserves all relevant weak/strict ascent/descent sets, 
 $\phi :\Upsilon^{\m}(\TT)  \xrightarrow{\sim} \Upsilon^{\m}(\TT')$
is an isomorphism of $W$-graphs. Replacing $\sigma$ by $\sgn\cdot \sigma$ shows that $\phi :\Upsilon^{\n}(\TT)  \xrightarrow{\sim} \Upsilon^{\n}(\TT')$ is 
also an isomorphism.
\end{proof}

For a model triple $\TT = (J,z_{\min},\sigma)$ for $W$, let
$\overline{\TT}$ denote the model triple  $(J, z_{\min}, \sgn\cdot \sigma)$.

\begin{proposition}\label{overline-prop}
If $\TT$ is a model triple for $W$ then $\Upsilon^{\m}(\TT)  = \Upsilon^{\n}(\overline\TT) $
and $\Upsilon^{\n}(\TT)  = \Upsilon^{\m}(\overline\TT) $.
\end{proposition}

\begin{proof}
This is obvious from the definitions.
\end{proof}

Write $w_0$ and $w_J$ for the longest elements of $W$ and $W_J$.
Suppose $\TT = (J,z_{\min}, \sigma)$ is a model triple.
The $W_J$-orbit of $z_{\min}$ contains a unique $W_J$-maximal element $z_{\max}$ \cite[Cor. 2.10]{RV}.
Define
\be \TT^\vee := (J^\vee, z_{\max}^\vee, \sigma^\vee)\ee
where $J^\vee = w_0 J w_0$, $\sigma^\vee = \sigma \circ \Ad(w_0)$,
and the operation $z\mapsto z^\vee$ is defined as in \eqref{wz-vee-eq}.

\begin{proposition}\label{dual-prop}
If $\TT$ is a model triple for $W$ then
so is $\TT^\vee$, and the $W$-graphs 
$\Upsilon^{\m}(\TT) $ and $ \Upsilon^{\n}(\TT^\vee)$,
as well as $\Upsilon^{\n}(\TT) $ and $ \Upsilon^{\m}(\TT^\vee) $,
are dual via the map $\tau \mapsto \tau^\vee$ defined by \eqref{tau-vee-eq}.
\end{proposition}

\begin{proof}
Write $\TT = (J,z_{\min}, \sigma)$ and define $z_{\max}$  as above.
It is easy to check from \eqref{wz-vee-eq} and our assumption that $z_{\max}$ is $W_J$-maximal that 
 $z_{\max}^\vee$ is $W_{J^\vee}$-minimal, so
$\TT^\vee$ is a model triple for $W$. 
In view of Theorem~\ref{dual-thm}, we just need to show that $\tau\mapsto \tau^\vee$ 
restricts to a bijection $\cK^\TT \to \cK^{\TT^\vee}$.
This is clear from 
Proposition~\ref{same-orbit-prop}
and the first claim in Lemma~\ref{vee-lem}.
\end{proof}

These propositions suggest the following definition of equivalence for (sets of) model triples.

\begin{definition}
Let $\sim$ be the transitive closure of the relation on model triples for $W$
that has $\TT \sim \TT'$ whenever $\TT' \equiv \TT$ or $\TT' = \overline\TT$ or $\TT' = \TT^\vee$
or  $\TT' =\TT^\alpha$ for some $\alpha \in \Aut(W,S)$.
We say that $\TT$ and $\TT'$ are \emph{equivalent} when $\TT\sim \TT'$.
Two sets  $\sP$ and $\sP'$ of model triples for $W$ are \emph{equivalent} if there is a bijection $ \sP \to \sP'$ such that if $ \TT\mapsto\TT'$ then $\TT \sim \TT'$.
\end{definition}

If we understand the cells in 
$\Upsilon^{\m}(\sP)$ and $ \Upsilon^{\n}(\sP)$ for some perfect model $\sP$,
then we also understand the cells in the $W$-graphs associated to any model in the equivalence class
of $\sP$, in the following sense. 
An \emph{isomorphism} of weighted directed graphs is a  bijection
between vertex sets that preserves edge weights and edge orientations;
an \emph{anti-isomorphism} is a bijection
between vertex sets that preserves edge weights while reversing edge orientations.

\begin{corollary}\label{equiv-cor}
If $\sP$ and $\sP'$ are equivalent sets of model triples for $W$  then there is a bijection between the vertex sets of $\Upsilon^{\m}(\sP)\sqcup \Upsilon^{\n}(\sP) $ and $ \Upsilon^\m(\sP')\sqcup \Upsilon^\n(\sP')$ whose restriction to each weakly-connected
component of the underlying weighted directed graph 
 is either an isomorphism or an anti-isomorphism onto its image.
This bijection
maps cells to cells.
\end{corollary}

\begin{proof}
Propositions~\ref{alpha-prop}, \ref{equiv-prop}, \ref{overline-prop}, and \ref{dual-prop} show that 
if $\TT$ and $\TT'$ are model triples for $W$ with 
$\TT' \equiv \TT$, $\TT' = \overline\TT$, $\TT' = \TT^\vee$
or  $\TT' =\TT^\alpha$ for some $\alpha \in \Aut(W,S)$,
then the weighted directed graphs underlying
 $\Upsilon^{\m}(\TT)\sqcup \Upsilon^{\n}(\TT) $
and
$\Upsilon^{\m}(\TT')\sqcup \Upsilon^{\n}(\TT') $
are isomorphic or anti-isomorphic.
The corollary now follows from our definitions of $\Upsilon^\m(\sP)$, $\Upsilon^\n(\sP)$, and equivalence for models.
\end{proof}

We will show in \cite{MZ} that if $W \in \{S_n, \W_n, \WD_{2n+1}\} \setminus\{S_4, \W_3, \WD_3\}$
then
there is an essentially 
unique equivalence class of perfect models $\cP$.
(The rank 3 case is exceptional, as one can observe in Figure~\ref{abcd3-fig},
which  shows $W$-graphs corresponding
to inequivalent perfect models for the isomorphic Coxeter systems of types $\A_3$ and $\D_3$.)
Via Corollary~\ref{equiv-cor}, this will give a precise sense in which the Gelfand $W$-graphs $\Gamma^\m(W) \cong \Upsilon^\m(\sP)$ and $\Gamma^\n(W) \cong \Upsilon^\n(\sP)$ should be considered
as the canonical ones for classical Weyl groups.

\subsection{Duality revisited}\label{revisit-sect}
 
We now 
prove Theorems~\ref{bcbc-dual-thm} and \ref{dd-dual-thm}
concerning the $W$-graphs $\Gamma^\m=\Gamma^\m(W)$ and $\Gamma^\n=\Gamma^\n(W)$ for $W \in \{\W_n,\WD_n\}$.
%
Recall the notation $\underline w $ for elements of $ \cG^\BC_n \supset \cG^\D_n$ from Definition~\ref{suggests-def}.
For $w \in \W_{n}$, let $-w\in \W_{n}$ be the permutation $i \mapsto -w(i) $, and when $w=w^{-1}$ define 
\be\label{iotaBC-eq}
\dualityBC_n(\underline w) = \underline{-w}.\ee This map restricts to an involution of $\cG^\BC_n$,
which is the vertex set of $\Gamma^\m$ and $\Gamma^\n$ when $W=\W_n$.
For example,
if $z =\underline{\bar 3 2 \bar 1 \bar 4 \bar5} = \bar 3, 8, \bar 1, \bar 7, \bar 6, \bar 5, \bar 4, 2,10,9 \in \cG^\BC_5$
as in
\eqref{bc-under-eq}
then
\[ \dualityBC_n(z) = \underline{ 3 \bar 2  1  4 5} = 3,\bar 6,1,7,8,\bar 2,4,5,10,9.
\]
The longest element $w_0 \in \W_n$ is $w_0 = \bar 1 \bar 2 \cdots \bar n$ and $-w = w_0 w= ww_0$
 for all $w \in W$.
The following result is an effective version of  Theorem~\ref{bcbc-dual-thm}.
\begin{theorem}\label{bc-dual-thm}
If $W=\W_n$ then $\Gamma^\m$ and $\Gamma^\n$
are dual via  $\dualityBC_n$.
\end{theorem}

\begin{proof}
Let 
$\iota = \dualityBC_n$ and write $\cG = \cG^\BC_n$.
Fix $k \in \NN$ with $2k\leq n$ and 
define $\TT = (J,z_{\min},\sigma)$ as in part ($\BC)$ of Definition~\ref{main-thm1-def},
so that $J = \{ s_0,s_1,s_2,\dots,s_{n-1}\} \setminus \{s_{2k}\}$ and $z_{\min} = s_1s_3s_5\cdots s_{2k-1}$.
One can show that the unique $W_J$-maximal element $z_{\max}$ in the $W_J$-orbit of $z_{\min}$
is $z_{\max} = s_{-1}s_{-3}s_{-5}\cdots s_{-2k+1}$
where $s_{-i} := (i,-i-1)(-i,i+1)$ for $i>0$.

The longest element $w_0 \in \W_n$ is central 
so $\Ad(w_0)$ is the identity map.
The longest element $w_J \in W_J \cong \W_{2k}\times S_{n-2k}$ is 
$w_J = u_J v_J= v_Ju_J$ where $u_J := (-1,1)(-2,2)\cdots (-2k,2k)$ and $v_J $ is the permutation in $\W_n$
fixing $i \in [2k]$ and mapping $2k+i \mapsto n + 1 - i$ for $i \in [n-2k]$.
Since $u_J$ is central in $W_J$, we have $\Ad(w_J) = \Ad(v_J)$.
Thus, in the notation before Proposition~\ref{dual-prop}, 
\[J^\vee=J, \quad \sigma^\vee=\sigma, \quand
z^\vee =  z \cdot  u_J \cdot (v_J, \Ad(v_J)) \in \sI_J.\]
In particular, $z_{\max}^\vee = z_{\min}  v_J^+$ 
for $v_J^+ := (v_J, \Ad(v_J))$. But $v_J^+$ is central in $W_J^+$, so we have 
 $\TT \equiv \TT^\vee$ and the bijection $\cK^\TT \xrightarrow{\sim} \cK^{\TT^\vee}$
 in Proposition~\ref{equiv-prop} is the map $(w,z) \mapsto (w, z  v_J^+)$.
Proposition~\ref{dual-prop} asserts that $\Upsilon^{\m}(\TT)$ is dual to $\Upsilon^{\n}(\TT)$
via the map $(w,z) \mapsto (w w_J  w_0, z  u_J)$.

In view of Remark~\ref{bar-bar-rmk},
the map $z \mapsto \varphi_{W}(z)$ is an isomorphism  
$\Gamma^\m|_\TT \xrightarrow{\sim}\Upsilon^{\m}(\TT)$
and
$\Gamma^\n|_\TT \xrightarrow{\sim}\Upsilon^{\n}(\TT)$.
To finish this proof, it suffices to check that if $v \in \cG|_\TT$
has $\varphi_{W}(v) = (w,z)$ and $ \varphi_{W}\circ \iota(v) = (w',z')$, then 
$w' = w\cdot w_J\cdot w_0$ and $z'=z\cdot u_J$.
It is easy to verify the second identity, and
  in the notation of \eqref{wab-eq} with $m=2k$
we have $ w w_J w_0 =  a_1a_2\cdots a_m \bar b_{n-m} \cdots \bar b_2 \bar b_1 =w'$.
\end{proof}


The longest element $w_0 \in \WD_n$ is $\bar 1 \bar 2 \cdots \bar n$ when $n$ is even
and $1 \bar 2 \bar 3\cdots \bar n$ when $n$ is odd.
Recall that $\diamond$ is the Coxeter automorphism of $\WD_n$
mapping $w\mapsto w^\diamond := s_0ws_0$. 
If $n$ is even then this is an outer automorphism and $w_0$ is central,
while if $n$ is odd then $w^\diamond = w_0ww_0$ so $\diamond = \Ad(w_0)$.
If $\sigma : \WD_n \to \ZZ$ is any map then we define $\sigma^\diamond : \WD_n \to \ZZ$ by
$\sigma^\diamond(w) = \sigma(w^\diamond)$.
For $z \in \cG^\D_n$ define 
\be\label{iotaD-eq}
\dualityD_n(z)
=
 \begin{cases}
\dualityBC_n(z) 
&\text{if $n + |\{ i \in [n] : |z(i)| > n\}|$ is divisible by $4$} \\
 \dualityBC_n(z) ^\diamond &\text{if $n + |\{ i \in [n] : |z(i)| > n\}|$ is not divisible by $4$}.\end{cases}
 \ee
This is an involution of $\cG^\D_n$,
which is the vertex set of $\Gamma^\m$ and $\Gamma^\n$ when $W=\WD_n$.
 For example, if $z =\underline{\bar 3 2 \bar 1 54} = \bar 3, 6, \bar 1, 5,4, 2, 8,7,10,9 \in \cG^\D_5$
then
\[ 
\dualityD_n(z) = (\underline{ 3\bar 2  1 \bar5 \bar4})^\diamond =  \bar3, \bar6,  \bar1, \bar 5,\bar4, \bar2, 8,7,10,9 \neq \dualityBC_n(z).
 \]

Write $\Gamma^\m = \Gamma^\m(\WD_n)$ and $\Gamma^\n = \Gamma^\n(\WD_n)$ and let $\sP=\sP(\WD_n)$ as in 
Definition~\ref{main-thm1-def}.
Recall the labeled graphs  $\widetilde \Gamma^\m$ and $\widetilde \Gamma^\n$ 
introduced before Theorem~\ref{dd-dual-thm}.
Proposition~\ref{alpha-prop} implies that
both of these are $\WD_n$-graphs: for example, 
 $\widetilde \Gamma^\m$   is obtained from 
$\Gamma^\m = \bigsqcup_{\TT \in \sP} \Gamma^{\m}|_\TT$
 by replacing the summand
 $\Gamma^{\m}|_\TT$
 by 
  $(\Gamma^{\m}|_\TT  )^\diamond$ whenever 
$\TT=(J,z,\sigma)$ and $n+ |\{i \in [n]: |z(i)|>n\}| \notin 4\NN$.

The vertex sets of $\widetilde \Gamma^\m$ and $\widetilde \Gamma^\n$ are again both
given by $\cG = \cG^\D_n$.
  If $\TT\in \sP$
  then we write $\widetilde\Gamma^\m|_\TT$ and $\widetilde\Gamma^\m|_\TT$
for the $\WD_n$-graphs formed by restricting 
$\widetilde\Gamma^\m$ and $\widetilde\Gamma^\n$ to the vertex set $\cG|_\TT$.
We now also have an effective version of  Theorem~\ref{dd-dual-thm}.

\begin{theorem}\label{d-dual-thm}
If $W=\WD_n$ then $\Gamma^\m$ and $\widetilde\Gamma^\n$ (respectively, $\Gamma^\n$ and $\widetilde\Gamma^\m$)
are dual via  $\dualityD_n$.
\end{theorem}

\begin{proof}
The broad structure of the proof is the same as for Theorem~\ref{bc-dual-thm}.
Let  $\iota = \dualityD_n$ and write $\cG=\cG^\D_n$.
Fix $k \in \NN$ with $2k\leq n$ and 
define $\TT=(J,z_{\min},\sigma)$ as in part ($\D)$ of Definition~\ref{main-thm1-def},
so that $J = \{ s_{-1},s_1,s_2,\dots,s_{n-1}\} \setminus \{s_{2k}\}$ when $k>0$ and
$J = \{s_1,s_2,\dots,s_{n-1}\}$ when $k=0$,
 and $z_{\min} = s_1s_3s_5\cdots s_{2k-1}$.
To the prove the theorem, it suffices to show that $\Gamma^\m|_\TT$ is dual to $\widetilde\Gamma^\n|_\TT$
and that $\Gamma^\n|_\TT$ is dual to $\widetilde\Gamma^\m|_\TT$
via the map $\iota$ restricted to $\cG|_\TT$.
We explain below how this claim reduces to checking that certain diagrams commute.

One can check that the unique $W_J$-maximal element $z_{\max}$ in the $W_J$-orbit of $z_{\min}$
is  
$ z_{\max} = s_{-1}s_{-3}s_{-5}\cdots s_{-2k+1}$ 
 if $k$ is even
or
$ z_{\max} = s_{1}s_{-3}s_{-5}\cdots s_{-2k+1}$ if $k$ is odd.
In either case, the longest element $w_J \in W_J \cong \WD_{2k}\times S_{n-2k}$ is 
$w_J = u_J v_J= v_Ju_J$ where $u_J $ and $v_J $ are as in the proof of Theorem~\ref{bc-dual-thm},
so $\Ad(w_J) = \Ad(v_J)$.

Suppose $n$ is even. Then $w_0$ is central and 
 in the notation of Proposition~\ref{dual-prop} we have
\[J^\vee=J, \quad \sigma^\vee=\sigma
\quand 
z^\vee =  z \cdot  u_J \cdot v_J^+ \in \sI_J\quad\text{
for $v_J^+ := (v_J, \Ad(v_J)) \in W_J^+$.}\]
If $k$ is even then
 $z_{\max}^\vee = z_{\min}   v_J^+$ 
and
$\iota(v)  =\dualityBC_n(v)$ for all $v \in \cG|_\TT$,
while
 \be\label{gamma-1}\widetilde\Gamma^\m|_\TT = \Gamma^\m|_\TT
 \quand
  \widetilde\Gamma^\n|_\TT = \Gamma^\n|_\TT.\ee
In this case  $\TT \equiv \TT^\vee$ 
and  the bijection $\cK^{\TT}  \xrightarrow{\sim} \cK^{\TT^\vee}$
 in Proposition~\ref{equiv-prop} is again $(w,z) \mapsto (w, z  v_J^+)$,
 so it follows from
 Proposition~\ref{dual-prop} that $\Upsilon^{\m}(\TT)$ is dual to $\Upsilon^{\n}(\TT)$
via the map 
\be\label{psi-eq} \psi: (w,z) \mapsto (w w_J  w_0, z  u_J).\ee
To show that $\Gamma^\m|_\TT$ is dual to $\Gamma^\n|_\TT$
via $\iota$
it suffices to check that the diagram
\be\label{cd-1}
\begin{tikzcd}
\cG|_\TT \arrow[rr, "\varphi_W"] \arrow[d, "\iota"] && \cK^{\TT} \arrow[d, "\psi"] \\
\cG|_\TT \arrow[rr, "\varphi_W" ] & & \cK^{\TT}
\end{tikzcd}
\ee
commutes. 
%

On the other hand, if $n$ is even and $k$ is odd then we have  $z_{\max}^\vee = z_{\min}^\diamond  v_J^+$, $J^\diamond = J$,  $\sigma^\diamond =  \sigma$,
and
$\iota(v)  =\dualityBC_n(v)^\diamond$ for all $v \in \cG|_\TT$, while
 \be\label{gamma-2}
 \widetilde\Gamma^\m|_\TT = (\Gamma^\m|_\TT)^\diamond\quand 
  \widetilde\Gamma^\n|_\TT = (\Gamma^\n|_\TT)^\diamond.\ee
In this case 
 $\TT^\diamond \equiv \TT^\vee = (J^\diamond, z_{\min}^\diamond  v_J^+, \sigma^\diamond)$ 
 and the bijection $\cK^{\TT} = \cK^{\TT^\diamond} \xrightarrow{\sim} \cK^{\TT^\vee}$
 in Proposition~\ref{equiv-prop} is the map $(w,z) \mapsto (w, z  v_J^+)$,
so it follows from Proposition~\ref{dual-prop} that $\Upsilon^{\m}(\TT)$ is dual to $\Upsilon^{\n}(\TT^\diamond)$
and that $\Upsilon^{\n}(\TT)$ is dual to $\Upsilon^{\m}(\TT^\diamond)$
again via  \eqref{psi-eq}.
To show that $\Gamma^\m|_\TT$ is dual to $\widetilde\Gamma^\n|_\TT$
(respectively, that $\Gamma^\n|_\TT$ is dual to $\widetilde\Gamma^\m|_\TT$)
via $\iota$
it suffices to check that 
\be\label{cd-2}
\begin{tikzcd}
\cG|_\TT \arrow[rr, "\varphi_W"] \arrow[d, "\iota"]
&& \cK^{\TT} \arrow[d, "\psi"] \\
\cG|_\TT \arrow[r, "\varphi_W" ]  & \cK^{\TT} \arrow[r, "\diamond"] 
& \cK^{\TT^\diamond}
\end{tikzcd}\ee
commutes, where the horizontal arrow $\xrightarrow{\ \diamond\ }$ is the map  $(w,z) \mapsto (w^\diamond,z^\diamond)$.

Now suppose $n$ is odd. Then 
$w_0$ is not central and 
 in the notation of Proposition~\ref{dual-prop}  
we have $J^\vee=J^\diamond$, $ \sigma^\vee=\sigma^\diamond$
and
$z^\vee =  z^\diamond \cdot  u_{J} \cdot \tilde v_{J}^+$ where $\tilde v_{J}^+ := (v_{J}^\diamond, \Ad(v_{J}^\diamond))$, which is equal to $v_J^+$ when $k\neq 0$.
If $k$ is  odd then $J=J^\diamond$, $\sigma = \sigma^\diamond$, 
$z_{\max}^\vee = z_{\min} \tilde v_{J}^+$,
 $\iota(v) = \dualityBC_n(v)$ for all $v \in \cG|_\TT$, and \eqref{gamma-1} holds.
In this case  $\TT \equiv \TT^\vee$ 
and showing that $\Gamma^\m|_\TT$ is dual to $\Gamma^\n|_\TT$
via $\iota$
reduces as above to checking that \eqref{cd-1} commutes after replacing the map $\psi$ with 
\be\label{psi-eq2}
 (w,z) \mapsto (w w_J  w_0, z^\diamond  u_J).\ee
If $k$ is even then $z_{\max}^\vee = z_{\min}^\diamond \tilde v_{J}^+$,
 $\iota(v) = \dualityBC_n(v)^\diamond$ for all $v \in \cG|_\TT$, and \eqref{gamma-2} holds.
In this case  $\TT^\diamond \equiv \TT^\vee$ 
and showing that 
$\Gamma^\m|_\TT$ is dual to $\widetilde\Gamma^\n|_\TT$
and that $\Gamma^\n|_\TT$ is dual to $\widetilde\Gamma^\m|_\TT$
via $\iota$
reduces to checking that \eqref{cd-2} commutes after replacing $\psi$ with \eqref{psi-eq2}.

All of the maps in the diagrams \eqref{cd-1} and \eqref{cd-2} have been explicitly defined.
Checking the desired commutativity in the four cases of interest
(according to the parities of $n$ and $k$) is a straightforward but somewhat tedious exercise.
We omit the specific details.
\end{proof}

\end{document}

%% file: preamble.tex
\usepackage{latexsym}
\usepackage{amssymb}
\usepackage{amsthm}
\usepackage{amscd}
\usepackage{amsmath}
\usepackage{mathrsfs}
\usepackage{graphicx}
\usepackage{hyperref}
\usepackage{shuffle}
\usepackage{tikz-cd}
\usepackage[all]{xy}
\input xy \xyoption{frame}
\usepackage{ytableau}

\usepackage{enumerate,tikz}
\usepackage{color}
\usepackage{mathdots}
\usepackage[vcentermath,enableskew]{youngtab}

\numberwithin{equation}{section}

\theoremstyle{definition}
\newtheorem* {theorem*}{Theorem}
\newtheorem* {conjecture*}{Conjecture}
\newtheorem{theorem}{Theorem}[section]

\theoremstyle{definition}

\theoremstyle{definition}

\newtheorem* {example*}{Example}

\newtheorem{lemma}[theorem]{Lemma}
\theoremstyle{definition}
\newtheorem{definition}[theorem]{Definition}
\theoremstyle{definition}

\newtheorem{conjecture}[theorem]{Conjecture}
\newtheorem{proposition}[theorem]{Proposition}
\newtheorem{corollary}[theorem]{Corollary}

\newtheorem* {remark*}{Remark}
\newtheorem {remark}[theorem]{Remark}
\theoremstyle{definition}
\newtheorem {example}[theorem]{Example}
\theoremstyle{definition}

\theoremstyle{definition}

\theoremstyle{definition}

\theoremstyle{definition}

\def\modu{\ (\mathrm{mod}\ }

\def\({\left(}
\def\){\right)}
\newcommand{\sP}{\mathscr{P}}
\newcommand{\FF}{\mathbb{F}}
\newcommand{\CC}{\mathbb{C}}
\newcommand{\QQ}{\mathbb{Q}}
\newcommand{\cP}{\mathcal{P}}

\newcommand{\cK}{\mathcal{K}}

\def\cY{\mathcal{Y}}
\def\NN{\mathbb{N}}

\def\CC{\mathbb{C}}

\def\ZZ{\mathbb{Z}}
\def\Aut{\mathrm{Aut}}

\def\Ind{\mathrm{Ind}}
\def\GL{\mathrm{GL}}

\def\Res{\mathrm{Res}}

\def\spanning{\textnormal{-span}} 
\def\Irr{\mathrm{Irr}}

\newcommand{\cM}{\mathcal{M}}
\newcommand{\cN}{\mathcal{N}}
\newcommand{\cL}{\mathcal{L}}

\newcommand{\one}{{1\hspace{-.11cm} 1}}

\newcommand{\sgn}{\mathrm{sgn}}

\def\barr{\begin{array}}
\def\earr{\end{array}}
\def\ba{\begin{aligned}}
\def\ea{\end{aligned}}
\def\be{\begin{equation}}
\def\ee{\end{equation}}

\def\qquand{\qquad\text{and}\qquad}
\def\quand{\quad\text{and}\quad}

\def\quord{\quad\text{or}\quad}

\def\cH{\mathcal H}
\def\cM{\mathcal M}

\def\hs{\hspace{0.5mm}}

\def\ds{\displaystyle}

\def\dash{\hs\text{---}\hs}

\def\ben{\begin{enumerate}}
\def\een{\end{enumerate}}

\def\cT{\mathscr{T}}

\def\cF{\mathcal F}

\def\hs{\hspace{0.5mm}}

\def\D{\hat D}

\def\Des{\mathrm{Des}}

\newcommand{\xRightarrow}[2][]{\ext@arrow 0359\Rightarrowfill@{#1}{#2}}

\renewcommand{\O}{\operatorname{O}}
\newcommand{\Sp}{\operatorname{Sp}}

\newcommand{\cA}{\mathcal{A}}
\newcommand{\cB}{\mathcal{B}}

\def\sI{\mathscr{I}}

\def\cG{\mathcal{G}}

\def\arcstart{\ \xy<0cm,-.06cm>\xymatrix@R=.1cm@C=.2cm }
\newcommand{\arcstartc}[1]{\ \xy<0cm,-.15cm>\xymatrix@R=.1cm@C=#1cm}

\def\cF{\mathcal F}

\def\h {\mathrm{ht}}

\def\m{\mathfrak{m}}
\def\n{\mathfrak{n}}

\def\invsim_i{\overset{\mathrm{i}}{\underset{\mathrm{inv}}{\sim}}}

\def\LL{\ZZ[x,x^{-1}]}
\def\H{\mathcal{H}}
\def\A{\mathsf{A}}
\def\BC{\mathsf{BC}}
\def\D{\mathsf{D}}

\def\cG{\mathcal{G}}

\def\Asc{\mathrm{Asc}}
\def\cP{\mathscr{P}}

\def\WB{W^{\mathsf{BC}}}
\def\WD{W^{\mathsf{D}}}

\def\m{\mathbf{m}}
\def\n{\mathbf{n}}

\def\Ad{\mathrm{Ad}}
\def\W{W^{\BC}}
\def\WD{W^{\D}}

\def\TT{{\mathbb T}}

\usepackage[colorinlistoftodos]{todonotes}

\def\KK{\mathbb{K}}